\documentclass{amsart}
\usepackage{amssymb, amsmath, amsfonts, amsthm, graphics,enumerate,tikz,tikz-cd}
\usepackage{hyperref}
\usepackage{mathtools}
\usepackage[all,cmtip]{xy}
\usetikzlibrary{backgrounds}

\DeclareFontFamily{U}{wncy}{}
    \DeclareFontShape{U}{wncy}{m}{n}{<->wncyr10}{}
    \DeclareSymbolFont{mcy}{U}{wncy}{m}{n}
    \DeclareMathSymbol{\Sha}{\mathord}{mcy}{"58} 

\newenvironment{sbm}
    {\left[ \begin{smallmatrix}
    }
    { 
     \end{smallmatrix} \right]
    }

\newcommand{\Ta}{
\begin{tikzpicture}[scale=.2]
  \draw (0,0) circle(0.3)  [fill=orange!50] ;
  \draw (0.5,0.87) circle(0.3)  [fill=orange!50];
  \draw  (1,0)  circle(0.3) [fill=orange!50];
\end{tikzpicture}}

\newcommand{\Tb}{
\begin{tikzpicture}[scale=.2]
  \draw (0,0) circle(0.3)  [fill=blue!50] ;
  \draw (0.5,0.87) circle(0.3)  [fill=orange!50];
  \draw  (1,0)  circle(0.3) [fill=orange!50];
\end{tikzpicture}}

\newcommand{\Tc}{
\begin{tikzpicture}[scale=.2]
  \draw (0,0) circle(0.3)  [fill=blue!50] ;
  \draw (0.5,0.87) circle(0.3)  [fill=blue!50];
  \draw  (1,0)  circle(0.3) [fill=orange!50];
\end{tikzpicture}}

\newcommand{\Td}{
\begin{tikzpicture}[scale=.2]
  \draw (0,0) circle(0.3)  [fill=blue!50] ;
  \draw (0.5,0.87) circle(0.3)  [fill=blue!50];
  \draw  (1,0)  circle(0.3) [fill=blue!50];
\end{tikzpicture}}

\newcommand{\Te}{
\begin{tikzpicture}[scale=.2]
  \draw (0,0) circle(0.3)  [fill=orange!50] ;
  \draw (0.5,0.87) circle(0.3)  ;
  \draw  (1,0)  circle(0.3) [fill=blue!50];
\end{tikzpicture}}

\newcommand{\newword}[1]{\textbf{\emph{#1}}}

\newcommand{\into}{\hookrightarrow}
\newcommand{\onto}{\twoheadrightarrow}

\newcommand{\from}{\leftarrow}

\newcommand{\covers}{\gtrdot}
\newcommand{\covered}{\lessdot}

\newcommand{\torsion}{T}
\newcommand{\free}{F}

\DeclareMathOperator{\Hom}{Hom}
\DeclareMathOperator{\Ext}{Ext}

\DeclareMathOperator{\End}{End}
\DeclareMathOperator{\Tor}{Tor}
\DeclareMathOperator{\Ker}{Ker}
\DeclareMathOperator{\CoKer}{CoKer}

\DeclareMathOperator{\Span}{Span}

\newcommand{\GL}{\mathrm{GL}}

\newcommand{\LCM}{\mathrm{LCM}}

\newcommand{\bigmeet}{\bigwedge}

\newcommand{\bigjoin}{\bigvee}

\newcommand{\inv}{\mathrm{inv}}

\newcommand{\hW}{\widehat{W}}
\newcommand{\hA}{\widehat{A}}
\newcommand{\hV}{\widehat{V}}

\def\shard{\sigma} %I'm open to other notations
\newcommand{\Sc}{\mathbb S}
\newcommand{\CC}{\mathbb{C}}

\newcommand{\FF}{\mathbb{F}}

\newcommand{\HH}{\mathbb{H}}

\newcommand{\KK}{\mathbb{K}}

\newcommand{\PP}{\mathbb{P}}

\newcommand{\RR}{\mathbb{R}}

\newcommand{\ZZ}{\mathbb{Z}}

\newcommand{\cS}{\mathcal{S}}
\newcommand{\cT}{\mathcal{T}}

\newcommand{\cX}{\mathcal{X}}
\newcommand{\cY}{\mathcal{Y}}

\newcommand{\Stab}{\mathrm{Stab}}
\newcommand{\Mod}{\mathrm{mod}} %Lowercase according to Hugh to indicate finite dimensional

\newcommand{\NoSub}{\mathrm{NoSub}}
\newcommand{\NoQuot}{\mathrm{NoQuot}}

\newcommand{\DGamma}{\overline{\Gamma}}

\newcommand{\preceqR}{\preceq_{\Phi}}
\newcommand{\precR}{\prec_{\Phi}}

\newcommand{\succR}{\succ_{\Phi}}

\newcommand{\Brick}{\mathrm{Brick}}

\newcommand{\Pairs}{\operatorname{Pairs}}

\theoremstyle{plain}
\newtheorem{thm}{Theorem}[section]

\newtheorem{theorem}[thm]{Theorem}
\newtheorem{Theorem}[thm]{Theorem}
\newtheorem{prop}[thm]{Proposition}
\newtheorem{Prop}[thm]{Proposition}
\newtheorem{proposition}[thm]{Proposition}
\newtheorem{Proposition}[thm]{Proposition}
\newtheorem{DP}[thm]{Definition/Proposition}
\newtheorem{lem}[thm]{Lemma}
\newtheorem{lemma}[thm]{Lemma}
\newtheorem{Lemma}[thm]{Lemma}
\newtheorem{cor}[thm]{Corollary}

\newtheorem{conj}[thm]{Conjecture}

\newtheorem*{prop*}{Proposition}
\newtheorem*{thrm*}{Theorem}

\theoremstyle{definition}

\newtheorem{example}[thm]{Example}
\newtheorem{eg}[thm]{Example}

\newtheorem{remark}[thm]{Remark}
\newtheorem{Remark}[thm]{Remark}

\newtheorem{Definition}[thm]{Definition}

\usepackage{color}

\newcommand{\margincolor}{red}      
\definecolor{darkgreen}{rgb}{0,0.7,0}

\newcounter{margincounter}
\newcommand{\marginnum}{
\ifnum\value{margincounter}<10
\textcolor{\margincolor}{\begin{picture}(0,0)\put(2.2,2.4){\circle{9}}\end{picture}\footnotesize\arabic{margincounter}}
\else\ifnum\value{margincounter}<100
\textcolor{\margincolor}{\begin{picture}(0,0)\put(4.256,2.5){\circle{11}}\end{picture}\footnotesize\arabic{margincounter}}
\else
\textcolor{\margincolor}{\begin{picture}(0,0)\put(6.8,2.5){\circle{14}}\end{picture}\footnotesize\arabic{margincounter}}
\fi\fi
}

\newcommand{\sgn}{\operatorname{sgn}}
\newcommand{\tr}{\operatorname{tr}}
\newcommand{\ot}{\leftarrow}

\usepackage{todonotes}

\title{Shard Modules}
\author{Will Dana, David E Speyer and Hugh Thomas}

\begin{document}

\maketitle
\begin{abstract}
Motivated by the goal of studying cluster algebras in infinite type, we study the stability domains of modules for the preprojective algebra in the corresponding infinite types.
Specifically, we study real bricks: those modules whose endomorphism algebra is a division ring and which have no self-extensions.
We define ``shard modules" to be those real bricks whose stability domain is as large as possible (meaning, of dimension one less than the rank of the preprojective algebra). 
We show that all real bricks are obtained by applying the  Baumann-Kamnitzer reflection functors to simple modules, and we give a recursive formula for the stability domain of a real brick. 
We show that shard modules are in bijection with Nathan Reading's ``shards", and that their stability domains are the shards; we also establish many foundational results about shards in infinite type which have not previously appeared in print.
With an eye toward applications to cluster algebras, our paper is written to handle skew-symmetrizable as well as skew-symmetric exchange matrices, and we therefore discuss the basics of the theory of species for preprojective algebras.
We also give some counterexamples to show ways in which infinite type is more subtle than the well-studied finite type cases.
\end{abstract}

\setcounter{tocdepth}{1} %Comment out this line to get the big table of contents back.
\tableofcontents

%\margin{A handy reference for conventions: \\
% $\cT^{\theta}$ is modules $M$ which have $\langle Q, \theta \rangle \leq 0$ for all quotients $Q$ of $M$. \\
% $\cF^{\theta}$ is modules $M$ which have $\langle K, \theta \rangle \geq 0$ for all submodules $K$ of $M$. \\
% $M$ is $\theta$-semistable if $\langle M, \theta \rangle = 0$ and $\langle K, \theta \rangle \geq 0$ for all submodules $K$ of $M$. \\
% $\Sigma_i : \NoQuot_i \to \NoSub_i$. See bottom of page 7 in \cite{BK}.\\
%  ${\shard_i^+(K)= s_i \left( K  \cap \{ x : (x, \alpha_i) \geq 0 \} \right)}$.\\
%  When calling two roots $\beta$ and $\beta'$, the root $\beta$ should be larger in the root poset. \\
%  Symmetrizable Cartan matrix has $d_i A_{ij} = d_j A_{ji}$ and  $\alpha_i = d_i \alpha_i^{\vee}$.\\
%  $(\alpha_i^{\vee}, \alpha_j) = A_{ij}$\\
%  $\omega_i$ is the dual basis to $\alpha_i$\\
%  $d_{\beta} = (\beta, \beta)/2$. $\beta = d_{\beta} \beta^{\vee}$.
%}
%
\section{Introduction}

The aim of this paper is to generalize to infinite root systems various results concerning preprojective algebras in Dynkin type. 
In this introduction, we summarize the results we want to generalize, state our main results, and explain why we want to establish such a generalization.
We begin with a review of the general theory of torsion classes.

\subsection{Bricks, torsion classes, stability} Let $k$ be a field and let $R$ be a $k$-algebra. We write $\Mod(R)$ for the abelian category of left $R$-modules which are finite dimensional over $k$, and call such an $R$-module a \newword{finite dimensional $R$-module}.

A \newword{torsion class} $\torsion$ is a collection of isomorphism classes of finite dimensional $R$-modules which contains $0$ and is closed under quotients and extensions.
A \newword{torsion free class} $\free$ is a collection of isomorphism classes of finite dimensional $R$-modules which contains $0$ and is closed under submodules and extensions.
If $C$ is any collection of finite dimensional $R$-modules, we define
\[  \begin{array}{ccl}
C^{\perp} &=& \{ M \in \Mod(R) : \Hom_{R}(X,M) = 0 \quad \forall X \in C \} \\ 
{}^{\perp} C &=& \{ M \in \Mod(R) : \Hom_{R}(M,X) = 0 \quad \forall X \in C \} \\
\end{array} . \]
The collection $C^{\perp}$ is always a torsion free class; ${}^{\perp} C$ is always a torsion class, and the maps $\torsion \mapsto \torsion^{\perp}$ and $\free \mapsto {}^{\perp} \free$ are mutually inverse bijections from torsion classes to torsion free classes and vice versa. 
The pair $(\torsion, \free)$ is called a \newword{torsion pair}.

We define a finite-dimensional $R$-module $B$ to be a \newword{brick} if $\Hom_R(B,B)$ is a division ring. In other words, $B$ is a brick if $B \neq 0$ and every nonzero endomorphism of $B$ is bijective. We write $\Brick(R)$ for the set of bricks of $R$.

The following results are straightforward:
%\margin{Add citations?} 
\begin{Prop} \label{TorClassesLattice}
Torsion classes form a lattice with respect to containment. 
\end{Prop}

\begin{Prop}  \label{TorClassesFromBricks}
Given two torsion classes $\torsion_1$ and $\torsion_2$, we have  $\torsion_1 \subseteq \torsion_2$ if and only if $\torsion_1 \cap \Brick(R) \subseteq \torsion_2 \cap \Brick(R)$; in particular, $\torsion_1 = \torsion_2$ if and only if $\torsion_1 \cap \Brick(R) = \torsion_2 \cap \Brick(R)$. The same holds for torsion free classes.
\end{Prop}

These results raise the question of finding a concrete description of the lattice of torsion classes of $R$-modules, and in particular describing such torsion classes as subsets of the set of bricks.

One way to obtain a torsion class or torsion free class is using a stability condition. 
Let $e_1$, $e_2$, \dots, $e_n$ be a complete collection of primitive orthogonal idempotents of $R$.
 %(meaning that $e_i^2 = e_i$, that $e_i e_j=0$ for $i \neq j$ and that $\sum e_i = 1$). 
%\margin{HT: we want something else to say that the choice is as refined as possible. DES: Technically, what I wrote isn't wrong for a non-maximally-refined collection; that's a bad writing choice, though. I added the word ``primitive", which is the word you are looking for, but took out the parenthetical because the definition of primitive is too wordy.}
%\inline{HT: Things will not work quite the way I expect if the algebra is not basic. An algebra is basic if $Ae_i \not\simeq Ae_j$ for $i\ne j$. The example to have in mind is the algebra of  $2\times 2$ matrices, which only has one simple module up to isomorphism (a row vector of length 2), so is not basic. The way this messes with my expectations is that the Grothendieck group is only rank 1. But nothing obvious goes wrong: there is a (non-surjective) linear map from the Grothendieck group to the space of dimension vectors, so $\theta$ does define a linear function on the Grothendieck group. I am leaving this note here for now in case it suggests any issues to the two of you (or to me when I look back at it). WD: To my knowledge, we don't use the interpretation of dimension vectors as elements of the Grothendieck group in a substantial way, and I don't think your point here raises any issues.}
So every finite dimensional $R$-module $M$ admits a vector space decomposition as $\bigoplus e_i M$.  We abbreviate $e_i M$ to $M_i$ and we define $\dim M$ to be the vector $(\dim_k M_i)$ in $\RR^n$. 
\begin{remark}
In all of our motivating examples, the algebra $A$ is basic, meaning that the $A$-modules $A e_i$ and $A e_j$ are not isomorphic for $i \neq j$. 
If $A$ is not basic, with $A e_i \cong A e_j$, then $\dim_k M_i = \dim_k M_j$ for every $M \in \Mod_A$. 
This does not make anything that we write false, and will also not occur for any of the examples we care about, but if we wanted to study non-basic algebras seriously, we would pay more attention to this.
\end{remark}

Let $\theta \in \RR^n$.
Define $\torsion^{\theta}$ to be the set of modules $M$ such that $\langle \theta, Q \rangle \leq 0$ for all quotients $Q$ of $M$, and define $\free^{\theta}$ to be the set of modules $M$ such that $\langle \theta, K \rangle \geq 0$ for all submodules $K$ of $M$. 
%\margin{We want the identity region of the Coxeter arrangment to correspond to the zero torsion class. With usual sign conventions, $\langle \alpha_i, \theta \rangle \geq 0$ for $\theta \in e D$. That forces the sign choice in this paragraph.} 
%(One could mix strong and weak inequalities in different ways here, but we gloss over this issue for the moment.\margin{Okay? HT: Off the top of my head, not my first choice, but I'll see how it works out later.})
It is easy to check that $\torsion^{\theta}$ is a torsion class and $\free^{\theta}$ is a torsion free class. If, for every brick $B \in \Mod(R)$, we have $\langle \theta, \dim B \rangle \neq 0$, then $\torsion^{\theta} = {}^{\perp}(\free^{\theta})$ and $\free^{\theta} = (\torsion^{\theta})^{\perp}$. 

By Proposition~\ref{TorClassesFromBricks}, if $\theta_1$ and $\theta_2$ lie on the same side of $(\dim B)^{\perp}$ for all bricks $B$, then $\torsion^{\theta_1} = \torsion^{\theta_2}$. 
Thus, for each region of the hyperplane arrangement complement $\RR^n \setminus \bigcup_{B \in \Brick(R)} (\dim B)^{\perp}$, we obtain a torsion class which occurs as $\torsion^{\theta}$ for every $\theta$ in that region. %\margin{WD: It's a little unclear how this follows from the preceding. DES: Better now? WD: Yep.}

It is important to understand, as $\theta$ passes through a wall $(\dim M)^{\perp}$, when the module $M$ jumps from $\free^{\theta}$ to $\torsion^{\theta}$. 
This happens when $M$ is $\theta$-semistable, which is defined as follows: Let $M \in \Mod(R)$ and let $\theta \in \RR^n$. Then $M$ is \newword{$\theta$-semistable} if and only if $M\in \free^\theta \cap \torsion^\theta$,  that is to say, if $\langle \theta, \dim K \rangle \geq 0$ for all submodules $K$ of $M$ and $\langle \theta, \dim Q \rangle \leq 0$ for all quotient modules $Q$ of $M$. Note that this implies that $\langle \theta, \dim M \rangle = 0$. Also, if we impose that  $\langle \theta, \dim M \rangle = 0$, then the conditions on submodules and on quotients become equivalent to each other.

Given a module $M$, only finitely many dimension vectors can occur as dimensions of submodules of $M$, so the set of $\theta$ for which $M$ is $\theta$-semistable is a closed polyhedral cone in $\RR^n$; we denote it by $\Stab(M)$ and call it the \newword{stability domain of $M$}.

\subsection{Dynkin type preprojective algebras}
To make our introduction simpler, in this subsection we will restrict ourselves to the simply laced types $ADE$. 
Many of these results have also been established in types $BCFG$; see Remark~\ref{NonSimplyLacedHistory}
 All of our results apply to any symmetrizable crystallographic Cartan matrix. 
For standard notation concerning Coxeter groups and root systems, see Section~\ref{sec Coxeter}.

Let $\Gamma$ be a 
%connected simply laced finite type 
Dynkin diagram  of type $A_n$, $D_n$, $E_6$, $E_7$ or $E_8$. 
We write $\Gamma_0$ for the vertices of $\Gamma$ and $\Gamma_1$ for the edges.
We write $W$ for the corresponding Coxeter group, with simple generators $s_i$ for $i \in \Gamma_0$. We write $\Phi$ for the root system, which lives in a vector space $V$; we write $\Phi^+$ for the set of positive roots and $\alpha_i$ for the simple roots, which are again indexed by $i \in \Gamma_0$.

We define $\DGamma$ to be the directed graph obtained by replacing each edge of $\Gamma$ by two directed edges, one in each direction. We write $\DGamma_1$ for the edges of $\DGamma$.
For each edge of $\Gamma$, choose one orientation of this edge to call $a$ and the other to call $a^{\ast}$. The \newword{preprojective algebra} $\Lambda$ is the quotient of the quiver path algebra $k \DGamma$ (for some field $k$) by the relation $\sum_{a \in \Gamma_1} a^{\ast} a=\sum_{a \in \Gamma_1} a a^{\ast}$.
Thus, a finite dimensional $\Lambda$-module $U$ breaks up as $\bigoplus U_i$ for a sequence of vector spaces $U_i$ indexed by $i \in \Gamma_0$, with maps $\rho(a) : U_i \to U_j$ and $\rho(a^{\ast}) : U_j \to U_i$ for each edge $i \overset{a}{\longrightarrow} j$. These maps obey $\sum \rho(a^{\ast}) \rho(a) = \sum \rho(a) \rho(a^{\ast})$ where the left sum runs over $a$ with source $i$ and the right sum runs over $a$ with target $i$.

Given a finite dimensional representation $U$, we define $\dim U$ to be the vector $\sum_{i \in Q_0} (\dim_k U_i) \alpha_i$. 
We write $S_i$ for the simple $\Lambda$-module with dimension $\alpha_i$.
Thus, our stability conditions $\theta$ now live in the dual vector space $V^{\ast}$.

%We have the following deep results: 

\begin{Theorem} [{\cite[Theorem 1.2]{IRRT}}]\label{ADEBricksAreRoots}
Under our hypothesis that $\Gamma$ is of type $ADE$, all bricks  in $\Mod(\Lambda)$ have dimension vectors in $\Phi^{+}$. The endomorphism rings of these bricks are all isomorphic to $k$.
\end{Theorem}

There can be more than one brick with the same dimension vector. For example, let $\Gamma$ have type $A_2$ and consider representations with dimension vector $(1,1)$. So such a representation is of the form $k \overset{a}{\underset{a^{\ast}}\rightleftarrows} k$  where the maps $a$ and $a^{\ast}$ obey $a a^{\ast}=0$ and $a^{\ast} a = 0$.
We get two non-isomorphic bricks by taking $(a,a^{\ast}) = (1,0)$ or $(a,a^{\ast}) =(0,1)$.

Let $W$ be the reflection group corresponding to $\Gamma$. Recall that $W$ acts simply transitively on the regions of the hyperplane arrangement complement $V^{\ast} \setminus \bigcup_{\beta \in \Phi^+} \beta^{\perp}$. Let $D$ be the region $\{ \theta \in V^{\ast} : \langle \alpha_i, \theta \rangle \geq 0 \ \mbox{for all}\ i \in \Gamma_0 \}$.

\begin{Theorem} [{\cite[Corollary 4.1]{Mizuno}}]\label{ADETorClassesAreW}
Under our hypothesis that $\Gamma$ is of type $ADE$, the torsion classes of $\Mod(\Lambda)$ are in bijection with the Coxeter group $W$. Namely, the bijection sends $w \in W$ to $\torsion^{\theta}$ for an arbitrarily chosen $\theta$ in the interior of  $wD$. The containment order on torsion classes corresponds to right weak order on $W$.
\end{Theorem}

By the observations of the previous section, the torsion class $\torsion^\theta$ changes when $\theta$ crosses through $\Stab(B)$ for a brick $B$ of $\Lambda$. And by Theorem \ref{ADETorClassesAreW}, this happens when $\theta$ crosses through some hyperplane $\beta^\perp$ for $\beta\in\Phi^+$. This raises the natural question of how stability domains of bricks sit inside the hyperplanes $\beta^\perp$.
%This \margin[Hugh]{I think we should unpack the ``this'' a bit. Crossing from one torsion class happens along the hyperplanes, so the hyperplanes must be domains of stability for some bricks, and thus we want to understand which ones...} raises the natural question of what the stability domains of bricks in $\Mod(\Lambda)$ are. 
They turn out to correspond to shards, which are polyhedral objects introduced by Nathan Reading, driven by pure lattice-theoretic questions concerning weak order. 
We now give the definition of shards:

We define a subset $R$ of $\Phi^+$ to be a \newword{rank two subsystem} if $\#(R) \geq 2$ and if $R$ is the intersection of $\Phi^+$ with a $2$-dimensional linear space $L$.
In every rank two subsystem $R$, there are two vectors, $\beta_1$ and $\beta_2$, such that every other vector in $R$ is a nonnegative linear combination of $\beta_1$ and $\beta_2$; these are called the \newword{fundamental roots} of $R$.
We will say that $R$ \newword{cuts}  $\gamma^\perp$ if $R$ is a rank two subsystem containing $\gamma$ in which $\gamma$ is \emph{not} fundamental. 
Even when $\Phi^+$ is infinite, any given hyperplane $\gamma^\perp$ will only be cut by finitely many root two subsystems (Corollary~\ref{CuttingBound}).
If $R$ cuts $\gamma^{\perp}$ then $\gamma \in R$, so the codimension two subspace $R^{\perp}$ is contained in the codimension one subspace $\gamma^{\perp}$. 
We call the collection of hyperplanes $\{ R^{\perp} \}_{R \ \text{cuts}\ \gamma^\perp}$ the \newword{shard arrangement} in $\gamma^{\perp}$. We call the closures of the connected components of $\gamma^{\perp} \setminus \bigcup_{R \ \text{cuts}\ \gamma^\perp} R^{\perp}$ \newword{shards} or, when we need to specify the vector $\gamma$, we call them \newword{shards of $\gamma^{\perp}$}.

%Namely, let $\beta_1$ and $\gamma$ be positive roots. We say that $\beta_1$ \newword{cuts} $\gamma$ if there is another positive root $\beta_2$ with $\gamma = c_1 \beta_1 + c_2 \beta_2$ for $c_1$, $c_2 > 0$. \margin[Will]{Must $\beta_1$ be fundamental in order to cut? (Or do we just want to talk about rank 2 subsystems, rather than roots, cutting?)}\margin[Hugh]{Maybe it doesn't make a difference?}
%We define a \newword{shard with normal vector $\gamma$} to be a closure of a region of the hyperplane arrangment $\gamma^{\perp} \setminus \bigcup_{\beta_1 \text{ cuts }\gamma} (\beta_1^{\perp} \cap \gamma^{\perp})$.
 
 \begin{Theorem} [{\cite[Theorem 6]{Thom}}]\label{ADEStabsAreShards}
Under our hypothesis that $\Gamma$ is of type $ADE$, the bricks in $\Mod(\Lambda)$ of dimension $\gamma$ are in bijection with the shards of $\gamma^{\perp}$. More precisely, a brick $B$ corresponds to the shard $\Stab(B)$.
\end{Theorem}

 \begin{remark}\label{NonSimplyLacedHistory}
   Analogues of the above results follow for the non-simply laced Dynkin
   preprojective algebras by combining results of \cite{AH+} and \cite{MizunoShards}.

   In more detail, let $A$ be a finite-dimensional algebra, which, for
   convenience, we assume is basic. Associated to $A$ there is a fan $\mathcal G_A$ called the $g$-fan. If $\mathcal G_A$ has only finitely many cones, then
  the map from chambers of the fan to torsion
classes, sending a chamber $C$ to $\mathcal T^\theta$ for $\theta$ in the
interior of the chamber, is a bijection, as shown by Brüstle, Smith, and
Treffinger \cite{BST}.

If $\Lambda$ is a Dynkin type preprojective algebra, then Aoki, Higashitani, Iyama, Kase and Mizuno \cite{AH+} show
that the $\mathcal G_\Lambda$ is the corresponding Coxeter fan. In particular it is finite, so the
considerations of the previous paragraph apply. This establishes
Theorem \ref{ADETorClassesAreW} for all Dynkin-type preprojective algebras.

Again, let $A$ be a basic finite-dimensional algebra as above, and
assume the $\mathcal G_A$ is finite. Mizuno \cite{MizunoShards} defines a notion of shards for $\mathcal G_A$; each shard is the closure
of a union of codimension-one cones of $\mathcal G_A$
all lying in one hyperplane, and every codimension-one cone lies in
exactly one shard. Mizuno shows further that there is a bijection between
bricks and shards, such that the stability domain for a brick is given by the corresponding shard.

Since the Coxeter fan originates from a hyperplane arrangement, Mizuno's
definition of shards restricts to the usual definition. As Mizuno points
out, this establishes Theorem \ref{ADEStabsAreShards} for
all Dynkin-type preprojective algebras.\end{remark}

 \subsection{Motivation for studying the infinite case} 
 Now, let $\Phi$ be the real roots of an infinite root system and $\Phi^+$ the positive real roots. %\margin{DES Does this address the (commented out) discussion adequately? HT Yes, though I would like to replace the final ``this'' of this paragraph by ``such an embedding of $W$ into a combinatorially tractable complete lattice''. DES Added something slightly more concise.}
 In this setting, for every element $w$ of the Coxeter group, there is a corresponding region of $V^{\ast} \setminus \bigcup_{\beta \in \Phi^+} \beta^{\perp}$: The union of the closures of these regions in $V^{\ast}$ form the so-called ``Tits cone", and these regions correspond to torsion classes with finitely many bricks (up to isomorphism).
 The other regions of $V^{\ast} \setminus \bigcup_{\beta \in \Phi^+} \beta^{\perp}$ do not correspond to elements of the Coxeter group.
 We would like to embed weak order on $W$ into a combinatorially tractable complete lattice, indexed by something like (though not necessarily the same as) the set of regions of  $V^{\ast} \setminus \bigcup_{\beta \in \Phi^+} \beta^{\perp}$.
 To this end, we want to study the torsion classes of $\Lambda$ for $\Lambda$ the preprojective algebra corresponding to $\Phi$.
 Here are some of the reasons we want to embed weak order into such a lattice:

% We would like to extend the results of the previous section to infinite root systems.
%  We want this because we hope it will give us a lattice larger than weak order, related to the regions of $V^{\ast} \setminus \bigcup_{\beta \in \Phi^+} \beta^{\perp}$.
%   \inline{HT: I think we need to explain more (a) that $\Phi^+$ are the positive real roots (b) why this is still the right thing to be thinking about. This is also the point of my old comment \ref{say-infinite}: the elements of $W$ correspond to a proper subset of these regions, and that's why we hope to be able to get a poset extending $W$ from all the regions. WD: Hugh, by (a) do you mean that we should emphasize that we're not talking about the imaginary roots, and/or that we should say more to introduce the definition of $\Phi^+$ in the infinite case? HT: Sorry, what I wrote was not very clear. My point in (a) is just the rather trivial one that here we are using the notation $\Phi^+$ when we have only defined it in the finite case. So, yes, it just amounts to saying that here $\Phi^+$ means the positive real roots.}

First, many authors have described the combinatorics of cluster algebras in terms of the combinatorics of Coxeter groups.
(See, for example~\cite{YSystems, CA2, CambFan, RS, Affine, SpeyerThomas, FeliksonTumarkin,  BrodskyStump}.) This work is essentially complete for cluster algebras of Dynkin type, which correspond to finite Coxeter groups. In order to study cluster algebras of non-Dynkin type, it appears we will need to introduce some lattice containing weak order as a lower ideal. The lattice of torsion classes of $\Tor(\Lambda)$ is a reasonable candidate for such a lattice although, as we will discuss below, it is not exactly what we want. Thus, we want to understand $\Tor(\Lambda)$ when $\Lambda$ is of non-Dynkin type.

We note the following clues that this project is reasonable. First of all, to every cluster algebra and choice of initial cluster there is associated a $g$-vector fan. Each cone of the $g$-vector fan is simplicial. The normals to the  defining hyperplanes of this cone are called $c$-vectors. For acyclic cluster algebras, $c$-vectors are always real roots. Thus, it is natural to hope to define the $g$-vector fan as a coarsening of the hyperplane arrangement $V^{\ast} \setminus \bigcup_{\beta \in \Phi^+} \beta^{\perp}$, and thus to describe the clusters using some sort of equivalence relation on some lattice which is related to the regions of this hyperplane arrangement. Reading and Speyer have carried this project out for the cluster algebras of affine type. 

Garver and McConville~\cite{GM1, GM2} studied type $A$ cluster algebras with respect to initial quivers which are mutation equivalent to a type $A$ quiver but which need not themselves be type $A$ Dynkin quivers.
They categorified these cluster algebras using modules for a ``tiling algebra". Since tiling algebras are quotients of the path algebras for the quivers associated to these seeds, they are also quotients of the corresponding preprojective algebras.
 %and the bricks which Garver and McConville consider are all shard modules.
Garcia and Garver~\cite{GG} make explicit connections between this work and semistable categories.
Combining their work with ours, one should be able to describe these $g$-vector fans as coarsenings of the Coxeter fans for the Coxeter groups of quivers which are mutation equivalent to a type $A$ quiver. 
 Garver and his collaborators also describe the lattice of clusters as a quotient of a lattice called the ``lattice of biclosed sets".
Garver, McConville and Mousavand~\cite{GMM} categorify the lattice of biclosed sets using an algebra $\Pi(A)$ which they describe as ``analogous to a preprojective algebra"~\cite[Section 6, first paragraph]{GMM}; we imagine that they could also work directly with the preprojective algebra, in which case the results of this paper would be useful.
We should acknowledge that, for the particular problems that these papers solve, the tiling algebra and the algebra $\Pi(A)$ are more convenient than the preprojective algebra; we hope that the merit of this paper is to show what sort of categorification might be powerful enough to handle a broader range of cluster algebras.

% \margin[David]{Okay, I started researching this and am leaving a comment here to record my progress. When I wrote this comment, I was remembering conversations with Alex Garver about his papers https://mathscinet.ams.org/mathscinet-getitem?mr=3919623 and https://mathscinet.ams.org/mathscinet-getitem?mr=4039006 . What I took away from those conversations is that he (and coauthors) were describing type $A$ cluster combinatorics in terms of torsion classes for the quivers with potential which we get from mutations of the type $A$ quiver. In particular, for any vertex in the type $A$ exchange graph, he had a way to make the exchange graph into the Hasse diagram of a lattice while putting that vertex at the bottom, which is something that Nathan and I could only do for acyclic vertices. This note is to go back and read those papers, in order to figure out how much of my memory of these conversations is in the papers.}

Santos, Stump and Welker~\cite{SSW} discovered the Grassmann-Tamari lattice, whose elements are cones in a certain fan.
McConville~\cite{McConville} proved that the Grassmann-Tamari lattice is a quotient of a larger lattice, the lattice of biclosed sets; many but not all of the elements in the lattice of biclosed sets are regions in a hyperplane arrangement. %\margin{Hugh I am surprised by this statement. I thought the biclosed sets are torsion classes for the appropriate gentle algebra, which has only finitely many, so is perfectly describable by a fan (though I don't see why it should be a hyperplane arrangement not a fan). DES The biclosed sets are torsion classes for the algebra $\Pi(A)$ of GMM, not the original gentle algebra; this is presumably the origin of the problem. I don't have any conceptual insight, but see Remark 6.2 in~\cite{McConville} for the statement that not all biclosed sets correspond to regions of the hyperplane arrangement, and that deleting the non-realizable ones doesn't give a lattice.}
%\margin[Hugh]{For the record, I understand what my confusion here was about, and I have commented out our discussion. In my view, probably the biclosed sets are simply the wrong thing to consider. But I don't think there is any need to change anything we say.}
The normal vectors to that hyperplane arrangement are certain real roots in the root system of type $A_k \boxtimes A_{n-k}$ (this is the wild type Coxeter group whose Coxeter diagram is the product of the $A_k$ and $A_{n-k}$ diagrams); Palu, Pilaud and Plamondon \cite{PPP}, and Brüstle et al.~\cite{BruestleEtAl}, categorify the Grassmann-Tamari lattice using torsion classes for a certain gentle algebra which is, again, a quotient of the preprojective algebra of type $A_k \boxtimes A_{n-k}$; 
the categorification of biclosed sets of Garver, McConville and Mousavand~\cite{GMM}, using an algebra ``analogous to a preprojective algebra", applies to this situation as well. 
It would be particularly nice to understand these examples better because a large portion of the fan for the Grassmann-Tamari lattice coincides the $g$-vector fan for the cluster structure on the Grassmannian $G(k,n)$. 
(Specifically, facets of the Grassmann-Tamari lattice are indexed by maximal non-kissing collections of $k$-element subsets of $[n]$. The ``non-kissing" condition is implied by a stronger condition called ``weakly separated". 
Maximal weakly-separated collections correspond to those clusters for $G(k,n)$ in which all Pl\"ucker variables are cluster variables~\cite{OPS}.

%McConville \margin{Others?} has also found other interesting lattices which are  related to infinite root systems $\Phi^{+}$, such as the Grassman-Tamari lattice. %\margin{HT: Is there really a root system here? I wouldn't have thought so. DES I think so! At some point, I convinced myself that this came from the Dynkin diagram which is a product of two paths, looking only at the roots which are supported on certain type $A$ subdiagrams. But the statement that it is a coarsening of the hyperplane arrangement was flawed, because of the usual issue that there are biclosed sets which don't come from stability conditions. I replaced this with something weaker. HT: fine}

Finally, Dyer~\cite{Dyer} has conjectured a description of an infinite lattice extending weak order in all cases.
Relating Dyer's conjectures to our approach strikes us as a challenging question; we will remark further on this in Remark~\ref{rem:Biclosed}.

In order to study all of these examples systematically, we want to build a larger lattice, and we view torsion classes as a reasonable route toward this goal.
Of course, the preprojective algebras are an important class of algebras, and it also seems natural to study their lattices of torsion classes, both for their intrinsic interest, and for a better understanding of the kind of phenomena which can arise in infinite lattices of torsion classes. 

\subsection{Difficulties of the infinite case} We now explain why thinking about the set of all torsion classes is definitely not what we want, as well as being, in a sense, too hard.
Let $k$ be an algebraically closed field and let $\Gamma$ be the Kronecker quiver, consisting of two vertices joined by two edges. 

The corresponding Coxeter group is the infinite dihedral group $\langle s_1, s_2 \ : \ s_1^2 = s_2^2 = e \rangle$. 
The positive real roots are $\{ m \alpha_1 + (m+1) \alpha_2 ,\ (m+1) \alpha_1 + m \alpha_2 \ : \ m \in \ZZ_{\geq 0} \}$. 
The hyperplanes $\alpha_1^{\perp}$ and $\alpha_2^{\perp}$ are the stability domains of the simple modules, and the simple modules are the only bricks of dimensions $\alpha_1$ and $\alpha_2$. 
For $m \geq 1$, there are two bricks each of the dimensions $m \alpha_1 + (m+1) \alpha_2$ and $(m+1) \alpha_1 + m \alpha_2$; we write them as $B(m \longleftarrow m+1)$, $B(m \longrightarrow m+1)$, $B(m+1 \longleftarrow m)$, $B(m+1 \longrightarrow m)$, where the nonzero maps go in the direction of the arrow. The two stability domains $\Stab(B(m \longleftarrow m \pm 1))$ and $\Stab(B(m \longrightarrow m \pm 1 ))$ are two rays pointing in opposite directions along the line $(m \alpha_1 + (m\pm 1) \alpha_2)^{\perp}$. 

There are two more families of nilpotent bricks, all of dimension vector $\alpha_1+\alpha_2$. Namely, for any $[c_1 : c_2] \in \PP^1(k)$, let $B(1 \overset{[c_1 : c_2]}{\longrightarrow} 1)$ be the representation where one map $k \to k$ is multiplication by $c_1$ and the other by $c_2$. For $[c_1:c_2]\ne [c_1':c_2']$, the corresponding bricks are non-isomorphic. 
Similarly, we can define a family $B(1 \overset{[c_1 : c_2]}{\longleftarrow} 1)$. 
(We will gloss over non-nilpotent bricks, since we don't regard them as the important difficulty here. For more details on nilpotency, see Section~\ref{preproj ssec gen}.)

%\inline{HT: There are actually two more families of bricks that we usually exclude. If the two arrows, pointing in the same direction, are $a$ and $b$, we can set $a$ and $b^*$ to zero (or $b$ and $a^*$). Note that they are simple modules of dimension 2, which should be a bit of a surprise. This is the kind of thing that can happen once our algebra is infinite dimensional. One normally excludes these guys by saying that we are working in the category of finite-dimensional modules on which the algebra acts nilpotently, and I guess we probably want to do that too. WD: Certainly all the modules we care about are acted upon nilpotently. DES: Actually, there are more than this. Any solution to $a a^{\ast} = b b^{\ast}$ gives a brick of dimension vector $(1,1)$. This is a three dimensional   space of bricks, or, quotienting by changes of basis, a two dimensional one. This is not the issue I want to get hung up on here; what is the best way to finesse this?}

Choose $\theta$ with $\langle \theta, \alpha_1 \rangle = -1$ and $\langle \theta, \alpha_2 \rangle = 1$. Then $\torsion^{\theta}$ contains the bricks $B(m+1 \longrightarrow m)$ for $m \geq 0$ and $B(1 \overset{[c_1 : c_2]}{\longrightarrow} 1)$ for all $[c_1 : c_2] \in \PP^1(k)$. The torsion class ${}^{\perp} \free^{\theta}$ contains the bricks $B(m+1 \longrightarrow m)$ for $m \geq 0$, but not the bricks $B(1 \overset{[c_1 : c_2]}{\longrightarrow} 1)$.

The interval of torsion classes $[ {}^{\perp} \free^{\theta}, \torsion^{\theta}]$ is then the boolean lattice of all subsets of $\PP^1(k)$. For any subset $\mathcal{C}$ whatsoever of $\PP^1(k)$, there is a torsion class with bricks 
\[\{ B(m \longrightarrow m+1) : m \geq 0 \} \cup \{ B(1 \overset{[c_1 : c_2]}{\longrightarrow} 1) : [c_1 : c_2] \in \mathcal{C} \} . \]  So the lattice of torsion classes has cardinality (at least) $2^{|k|}$; if $k = \CC$, this is $2^{2^{\aleph_0}}$. The authors are combinatorialists at heart, and we do not want to study such large lattices. Moreover, for all of the goals which motivate us, it appears that whether or not a torsion class contains bricks like $B(1 \overset{[c_1 : c_2]}{\longrightarrow} 1)$ is irrelevant.

To ignore the complexity introduced by these seemingly irrelevant bricks, we consider a particular subclass of bricks, which we term \newword{shard modules}, and consider two torsion classes to be equivalent if they contain the same shard modules. 
This paper contains a proposed definition of shard modules, and studies their properties. We now describe our main results.
Attempting to use this definition to study weak order and cluster algebras will be the project of future papers.

\subsection{Results}
Let $A$ be an $n \times n$ symmetrizable crystallographic Cartan matrix, meaning an $n \times n$ integer matrix with $A_{ii}=2$, with $A_{ij} \leq 0$ for $i\ne j$  and where there is a set of positive integers $d_i$ such that $d_i A_{ij} = d_j A_{ji}$. 
In the usual ways (see Section~\ref{sec Coxeter}), we define a Coxeter group $W$ and a reflection representation of $W$ on $V$ with $\Phi \subset V$ the set of real roots. We write $\alpha_i$ for the simple roots. 
The reflection group $W$ preserves a symmetric bilinear form $(-,-)$ on $V$, with respect to which $(\beta,\beta) > 0$ for every real root $\beta$.
We set $d_{\beta} = \tfrac{(\beta,\beta)}{2}$ and $\beta^{\vee} = d_{\beta}^{-1} \beta$; each $d_{\beta}$ lies in $\{ d_1, d_2, \ldots, d_n \}$.

%If $\beta$ is a root which is equal to both $u \alpha_i$ and $v \alpha_j$, then $d_i=d_j$. We set this common value to be $d_{\beta}$ and we set $\beta^{\vee} = d_{\beta} \beta$.

We abbreviate $d_{ij} = \LCM(d_i, d_j)$ and $L = \LCM(d_1, d_2, \ldots, d_n)$. %\margin[Hugh]{Corrected to agre with section 4.2.}
Let $\kappa(L)/\kappa$ be a field extension with Galois group cyclic of order $L$ and, for $d$ dividing $L$, let $\kappa(d)$ be the unique subextension with $[\kappa(d):\kappa] = d$. 
We define a preprojective algebra $\Lambda$ (see Sections~\ref{preproj ssec sym} and~\ref{preproj ssec gen}) whose representations are $\bigoplus U_i$ where $U_i$ is a $\kappa(d_i)$ vector space and we have $(-d_iA_{ij})/d_{ij}$ maps $\kappa(d_{ij})\otimes_{\kappa(d_i)}U_i \to U_j$ which are $\kappa(d_j)$ linear.
For a finite dimensional $\Lambda$-module $M$, we set $\dim M = \sum (\dim_{\kappa(d_i)} M_i ) \alpha_i$. 
We then define stability domains $\Stab(M)$ in $V^{\ast}$ as before. 

\begin{Definition}
We define a \newword{real brick} to be a brick whose dimension vector is a real root. 
\end{Definition}

There are equivalent ways to formulate the definition of a real brick without mentioning root systems:
{\renewcommand{\thethm}{\ref{RealBrickProperties}}
\begin{proposition} %\label{RealBrickProperties}
Let $B$ be a brick. The following are equivalent:
\begin{enumerate}
\item The vector $\dim B$ in $V$ is a real root.
\item We have  $(\dim B, \dim B)>0$.
\item The brick $B$ is rigid, meaning that $\Ext^1_{\Lambda}(B,B) =0$. 
\end{enumerate}
\end{proposition}
\addtocounter{thm}{-1}
}
We prove Proposition~\ref{RealBrickProperties} at the end of Section~\ref{sec:bricks}. Rigid bricks are also called \newword{stones} in some sources; see, for example,~\cite{KL}.

\begin{Definition}
We define a module $B$ to be a \newword{shard module} if $B$ is a real brick and $\Stab(B)$ is a polyhedral cone of dimension $n-1$.
\end{Definition}

Note that $\Stab(M)$ is contained in the hyperplane $(\dim M)^{\perp}$, so $n-1$ is the highest possible dimension.

\begin{remark}
It is not easy to give an example of a real brick which is not a shard module. 
However, such examples do exist! See Section~\ref{RealBrickNotShard}.
\end{remark}

Here is our main theorem:
{\renewcommand{\thethm}{\ref{StabsAreShards}}
\begin{Theorem} %\label{StabsAreShards}
The stability domains of shard modules are precisely the shards, and each shard is the stability domain of precisely one isomorphism class of shard modules.
\end{Theorem}
\addtocounter{thm}{-1}
}

Our proof works by providing parallel recursive descriptions of shards and of shard modules. We first describe the recursion for shards.
For any positive root $\beta$, define a \newword{positive expression for $\beta$}  to be a formula $\beta = s_{i_r} \cdots s_{i_2} s_{i_1} \alpha_j$ where, for all $1 \leq k \leq r$, we have \[ s_{i_k} s_{i_{k-1}} \cdots s_{i_2} s_{i_1} \alpha_j - s_{i_{k-1}} \cdots s_{i_2} s_{i_1} \alpha_j  \in \RR_{>0} \alpha_{i_k}.\]
We will show (Lemma~\ref{ShardRecursion}) that every positive root has a positive expression.
 Let $K$ be a convex polyhedral cone in $V^{\ast}$ and define
\[ \shard_i^+(K):= s_i \left( K  \cap \{ x : \langle x, \alpha_i\rangle \geq 0 \} \right) \ \mbox{and}\  \shard_i^-(K):= s_i \left( K \cap \{ x : \langle x, \alpha_i\rangle \leq 0 \} \right) . \]
Note that we could also write 
\[ \shard_i^+(K):= s_i (K)  \cap \{ x : \langle x, \alpha_i\rangle \leq 0  \} \ \mbox{and}\  \shard_i^-(K):= s_i (K) \cap \{ x : \langle x, \alpha_i\rangle \geq 0 \} . \]

In Section~\ref{ShardBackgroundSection}, we prove:
{\renewcommand{\thethm}{\ref{ShardRecursionTheorem}}
\begin{Theorem} %\label{ShardRecursionTheorem}
Let $\beta$ be any positive root and let $s_{i_r} \cdots s_{i_2} s_{i_1} \alpha_j$ be a positive	 expression for $\beta$. Then the set of shards of $\beta^{\perp}$ is the set of polyhedral cones of the form $\shard_{i_r}^{\pm_r} \cdots \shard_{i_2}^{\pm_2} \shard_{i_1}^{\pm_1} \left( \alpha_j^{\perp} \right)$ which are of dimension $n-1$, where the signs $\pm_1$, $\pm_2$, \dots, $\pm_r$ may be chosen independently.
\end{Theorem}
\addtocounter{thm}{-1}
}
We give an example of a product  $\shard_{i_r}^{\pm_r} \cdots \shard_{i_2}^{\pm_2} \shard_{i_1}^{\pm_1} \left( \alpha_j^{\perp} \right)$ which has dimension less than $n-1$ in Example~\ref{D4Example}.

We now describe the recursion for shard modules. We define $\NoSub_i$ to be the full subcategory of $\Mod(\Lambda)$ of modules for which the simple module $S_i$ is not a submodule, and define $\NoQuot_i$ similarly to be the full subcategory of $\Mod(\Lambda)$ of modules for which $S_i$ is not a quotient. 
Baumann and Kamnitzer~\cite{BK} define, and we will review in Section~\ref{ReflectionFunctorsDefn}, mutually inverse isomorphisms of categories
\[ \Sigma_i : \NoQuot_i \to \NoSub_i \ \mbox{and}\  \Sigma_i^{-1} : \NoSub_i \to \NoQuot_i  .\] 

In Section~\ref{ProofOfRealBrickRecursion}, we prove the following recursive description of real brick modules:
{\renewcommand{\thethm}{\ref{RealBrickRecursion}}
\begin{Theorem} %\label{RealBrickRecursion}
Let $\beta$ be a positive root and let $s_{i_r} \cdots s_{i_2} s_{i_1} \alpha_j$ be a positive expression for $\beta$. The bricks of dimension $\beta$ are precisely the modules of the form $\Sigma^{\pm_r}_{i_r} \cdots \Sigma^{\pm_2}_{i_2} \Sigma^{\pm_1}_{i_1} S_j$, where the signs must be chosen such that the expression is well-defined.
\end{Theorem}
\addtocounter{thm}{-1}
}
To be clear, the condition that $\Sigma^{\pm_r}_{i_r} \cdots \Sigma^{\pm_2}_{i_2} \Sigma^{\pm_1}_{i_1} S_j$ is well defined means that each partial product $\Sigma^{\pm_{q-1}}_{i_{q-1}} \cdots \Sigma^{\pm_2}_{i_2} \Sigma^{\pm_1}_{i_1} S_j$ lies in $\NoQuot_{i_{q}}$ or $\NoSub_{i_{q}}$, according to whether $\pm_{i_q} = +$ or $\pm_{i_q} = -$. 

We then describe the stability domains of the real brick modules in terms of this recursion, proved in Section~\ref{ProofOfStabOfAShardModule}.
{\renewcommand{\thethm}{\ref{StabOfAShardModule}}
\begin{Theorem} %\label{StabOfAShardModule}
Let $\beta$ be a positive root, let $s_{i_r} \cdots s_{i_2} s_{i_1} \alpha_j$ be a positive expression for $\beta$ and let $\pm_1$, $\pm_2$, \dots, $\pm_r$ be a choice of signs such that $\Sigma^{\pm_r}_{i_r} \cdots \Sigma^{\pm_2}_{i_2} \Sigma^{\pm_1}_{i_1} S_j$ is defined. Then the brick $\Sigma^{\pm_r}_{i_r} \cdots \Sigma^{\pm_2}_{i_2} \Sigma^{\pm_1}_{i_1} S_j$ has stability domain $\shard_{i_r}^{\pm_r} \cdots \shard_{i_2}^{\pm_2} \shard_{i_1}^{\pm_1} \left( \alpha_j^{\perp} \right)$.
\end{Theorem}
\addtocounter{thm}{-1}
}

Combining Theorems~\ref{RealBrickRecursion} and~\ref{StabOfAShardModule}, we immediately obtain
\begin{Theorem} \label{ShardModuleRecursion}
Let $\beta$ be a positive root and let $s_{i_r} \cdots s_{i_2} s_{i_1} \alpha_j$ be a positive expression for $\beta$.
The shard modules  of dimension $\beta$ are precisely those modules  $\Sigma^{\pm_r}_{i_r} \cdots \Sigma^{\pm_2}_{i_2} \Sigma^{\pm_1}_{i_1} S_j$  which are well defined and for which  $\shard_{i_r}^{\pm_r} \cdots \shard_{i_2}^{\pm_2} \shard_{i_1}^{\pm_1} \left( \alpha_j^{\perp} \right)$ has dimension $n-1$.
\end{Theorem}

Theorems~\ref{ShardRecursionTheorem} and~\ref{ShardModuleRecursion} do most of the work to prove Theorem~\ref{StabsAreShards}. We finish the proof of Theorem~\ref{StabsAreShards} in section~\ref{ProofOfStabsAreShards}; in particular, we resolve the tension that the recursion of Theorem \ref{RealBrickRecursion} appears to be more restrictive than the 
recursion of Theorem \ref{ShardRecursionTheorem}, since the former includes the condition that the reflection functors be well-defined, which is absent in the latter.

\subsection{Acknowledgments}
The first author has been supported by NSF grants  DMS-1854225 and DMS-1855135.
 The second author has been supported by NSF grants DMS-1854225, DMS-1855135 and DMS-1600223.
He thanks the Auslander Lectures 2019 for their excellent hospitality and for the helpful feedback of their attendees. The third author was partially supported by the Canada Research Chairs program and an NSERC Discovery Grant.  
He gratefully acknowledges the hospitality of the Center for Advanced Studies, Oslo, which provided excellent working conditions for the completion of this paper. 
The authors also thank Nathan Reading for his many helpful comments on this project.

\section{Background on Coxeter groups and root systems} \label{sec Coxeter}

This section provides our notation and conventions for discussing Coxeter groups and root systems. 
%Most will be familiar to the experts, but the material in Section~\ref{ShardBackgroundSection} will probably not be.

\subsection{Cartan matrices, Coxeter groups, root systems}
All the material in this section is standard. For example, see~\cite[Chapter 5]{Humphreys}.

Let $d_1$, $d_2$, \dots, $d_n$ be positive integers. Let $A$ be an $n \times n$ integer matrix obeying
\[ \begin{array}{rcl} 
A_{ii} &=& 2 \\
d_i A_{ij} &=& d_j A_{ji} \\
A_{ij} &\leq& 0 \ \mbox{for}\ i \neq j . \\
\end{array} \]
Such a matrix is called a \newword{crystallographic symmetrizable Cartan matrix}.

\begin{remark} \label{nonCrystalRemark}
 A \newword{symmetrizable Cartan matrix}, without the adjective ``crystallographic", would impose that the $A_{ij}$ and $d_i$ were real numbers, rather than integers, with $A_{ij} A_{ji} \in \{ 4 \cos^2 \tfrac{\pi}{m} : m \in \ZZ,\ m \geq 2 \} \cup [4, \infty)$ and obeying the other conditions above.
The non-crystallographic symmetrizable Cartan matrices are not related to preprojective algebras. 
However, as much of the theory of shards has never been published for infinite Coxeter groups, we will point out how to do some of the foundational work for non-crystallographic symmetrizable Cartan matrices where it is not difficult to do so.
\end{remark}

Let $V$ be a vector space with basis $\alpha_1$, \dots, $\alpha_n$ and let $\alpha_i^{\vee} = d_i^{-1} \alpha_i$. We define a symmetric bilinear form $(-,-)$ on $V$ by $(\alpha_i, \alpha_j) = d_i A_{ij}$, so $(\alpha_i^{\vee}, \alpha_j) = A_{ij}$. We let $s_i$ be the orthogonal reflection negating $\alpha_i$, given by the formula:
\[ s_i(x) = x - (\alpha_i^{\vee}, x) \alpha_i. \]
So the Coxeter group $W$ acts by transformations which are orthogonal for the form $(-,-)$.
We also record explicitly the action of $s_i$ on the $\alpha_j$:
\[ s_i(\alpha_j) = \alpha_j - A_{ij} \alpha_i. \]

Let $W$ be the subgroup of $\GL(V)$ generated by the $s_i$. Then $W$ is a Coxeter group, generated by the $s_i$ with relations $(s_i s_j)^{m_{ij}} = 1$ where
\[ m_{ij} = \begin{cases} 1 & i=j \\ 2 & A_{ij} A_{ji} = 0 \\ 3 & A_{ij} A_{ji} = 1 \\ 4 & A_{ij} A_{ji} = 2 \\ 6 & A_{ij} A_{ji} = 3 \\ \infty & A_{ij} A_{ji} \geq 4,\ i \neq j \end{cases}. \]

The set of \newword{real roots} $\Phi$ is defined to be $\{ w \alpha_i : w \in W,\ 1 \leq i \leq n \}$. 
In this paper, we will often shorten this to \newword{roots}, as we consider no other kind of root.
 %(These would be called real roots in some contexts, but we will usually omit the word ``real" since we consider no other type.)
  If $\beta$ is a root and we write $\beta = \sum c_i \alpha_i$, then either all the $c_i$ are $\geq 0$ or all the $c_i$ are $\leq 0$; we call $\beta$ a \newword{positive root} or \newword{negative root} accordingly and write $\beta \in \Phi^+$ or $\beta \in \Phi^-$. If $\beta$ is a root then $-\beta$ is also a root.

We state a lemma on roots for future use: 
\begin{Lemma} \label{GCD1}
Let $A$ be a symmetrizable crytallographic Cartan matrix and $\beta = \sum c_i \alpha_i$ a real root. Then $\mathrm{GCD}(c_1, c_2, \ldots, c_n) = 1$.
\end{Lemma}

\begin{proof}
  Since $s_j \sum c_i \alpha_i = \sum_{i \neq j} c_i \alpha_i + \left( -c_j + \sum_{i \neq j} A_{ij} c_i \right) \alpha_j$, 
  we see that the simple reflections preserve the GCD. Thus, it is enough to check the claim for $\beta$ a simple root, in which case it is obvious.
\end{proof}

For any real root $\beta$, we put $d_{\beta} = (\beta, \beta)/2$ and put $\beta^{\vee} = d_{\beta}^{-1} \beta$.
\begin{prop} \label{prop:dbeta}
If $\beta$ is in the same orbit as $\alpha_i$ for the $W$ action on $\Phi$, then $d_{\beta} = d_i$. 
\end{prop}

\begin{proof}
The group $W$ preserves the inner product $(-,-)$, so, if $\beta = u \alpha_i$, then $d_{\beta} = (\beta, \beta)/2 = (\alpha_i, \alpha_i)/2 = d_i$.
\end{proof}

If $\beta = u \alpha_i$ then the \newword{reflection over $\beta^{\perp}$} is the element $t:=u s_i u^{-1}$ in $W$. This element is determined by $\beta$, and in turn the pair $\pm \beta$ is determined by $t$. The reflection $t$ acts on $V$ by
\[ t(x) = x - (\beta^{\vee}, x) \beta .\]
The set of all reflections in $W$ is denoted $T$.

Let $V^{\ast}$ be the dual vector space to $V$. We write $\langle -,-\rangle$ for the pairing between $V^{\ast}$ and $V$. Let $\omega_i$ be the dual basis to $\alpha_i$, so $\langle \omega_i, \alpha_j \rangle = \delta_{ij}$. We let $W$ act on $V^{\ast}$ by the dual action, so that $\langle w \eta, w \beta \rangle = \langle \eta, \beta \rangle$. 
The action on the $\omega$ basis is given by
\[ s_i(\omega_j) = \begin{cases} -\omega_i + \sum_{i \neq k} (-A_{ik}) \omega_k & i=j \\ \phantom{-} \omega_j & i \neq j \end{cases} . \]
For $\beta \in \Phi$, we write $\beta^{\perp}$ for the hyperplane $\{ x\in V^{\ast} : \langle x, \beta \rangle = 0 \}$. If $t$ is the reflection over $\beta$, then $\beta^{\perp}$ is the fixed plane for the $t$-action on $V^{\ast}$.

Let $w$ be in the Coxeter group $W$. A \newword{word} for $w$ is any sequence $(s_{i_1}, s_{i_2}, \ldots, s_{i_m})$ with $s_{i_1} s_{i_2} \cdots s_{i_m} = w$. 
Such a word is called \newword{reduced} if it is of minimal length among all such words; this minimal length is denoted $\ell(w)$.

A reflection $t \in T$ is called an \newword{inversion} of $w$ if $w^{-1} \beta_t$ is a negative root. The number of inversions of $w$ is $\ell(w)$ and, if $s_{i_1} s_{i_2} \cdots s_{i_{\ell}}$ is any reduced word for $w$, then the inversions of $w$ are $\{ s_{i_1}, s_{i_1} s_{i_2} s_{i_1}, s_{i_1} s_{i_2} s_{i_3} s_{i_2} s_{i_1}, \ldots, s_{i_1} s_{i_2} \cdots s_{i_{\ell}} \cdots s_{i_2} s_{i_1} \}$. We write $\inv(w)$ for the set of inversions of $w$.
It is often more convenient to work with $\inv_{\Phi}(w) := \{ \beta_u : u \in \inv(w) \}$. 

The \newword{(right) weak order} %\margin[Will]{Specify \emph{right} weak order?}\margin[Hugh]{Agreed and done.} 
is the partial order on $W$ where $u \leq v$ if the following equivalent conditions hold:
\begin{itemize} 
\item[(1)] $\ell(v) = \ell(u) + \ell(u^{-1} v)$ 
\item[(2)] $\inv(u) \subseteq \inv(v)$ 
\item[(3)] there is a reduced word $s_{i_1} s_{i_2} \cdots s_{i_{\ell}}$ for $v$ such that $s_{i_1} s_{i_2} \cdots s_{i_k}$ is a reduced word for $u$ for some $k\leq \ell$. 
\end{itemize} Weak order is ranked by the length function $\ell$. If $u \covered v$ then $u = v s_i$ for some $i$ with $\ell(v) = \ell(u)+1$; in this case, we have $\inv(v) = \inv(u) \sqcup \{ u s_i u^{-1} \}$. We call $u s_i u^{-1}$ the \newword{cover reflection} of the cover $u \covered v$.

We write $D$ for the simplicial cone $\{ x \in V^{\ast} : \langle x, \alpha_i \rangle \geq 0 \}$ and $D^{\circ}$ for its interior $\{ x \in V^{\ast} : \langle x, \alpha_i \rangle > 0 \}$. The open cones $w D^{\circ}$ are disjoint, and the closed cones $w D$ are the maximal cones of a fan called the \newword{Coxeter fan}. The support of this fan, $\bigcup_{w \in W} w D$, is called the \newword{Tits cone}. The two cones $u D$ and $v D$ border along a codimension one face if only if $u \covers v$ or $u \covered v$. If the cover reflection of this cover is $t$, then $uD$ and $vD$ border along a portion of the hyperplane $\beta_t^{\perp}$.

Finally, we address finiteness conditions. The following conditions are all equivalent:
\begin{itemize}
\item[(1)] The group $W$ is finite. 
\item[(2)] The real root system $\Phi$ is finite. 
\item[(3)] The symmetric matrix $d_i A_{ij}$ is positive definite. 
\end{itemize}In this case we say that $W$, $\Phi$ and $A$ are \newword{of Dynkin type} or \newword{finite type}. We label the Dynkin type root systems by their standard names $A_n$, $B_n$, $C_n$, $D_n$, $E_6$, $E_7$, $E_8$, $F_4$ and $G_2$ in the usual way.

\subsection{The root poset and positive expressions}

Let $\beta$ be a real root and $s_i$ a simple reflection.
Then $s_i(\beta) - \beta \in \RR \alpha_i$. 
Define a partial order on $\Phi^+$ by the transitive closure of the relation that $\beta \succR s_i \beta$ if $\beta - s_i \beta \in \RR_{>0} \alpha_i$.
(This is clearly an antisymmetric relation, since a positive linear combination of the $\alpha_i$ cannot be zero.)
We will refer to this order as the \newword{root poset}. 
Basic references for the root poset are~\cite{BH}, \cite{HEriksson}, \cite{KEriksson}, \cite[Chapter 4.6]{BB} and~\cite{Stem}; we favor~\cite{Stem}.

In the introduction, we introduced the notion of a \newword{positive expression}: Given a positive root $\beta$, a \newword{positive expression} for $\beta$ is a factorization $\beta = s_{i_r} \cdots s_{i_2} s_{i_1} \alpha_j$ such that for all $1 \leq k \leq r$, we have $s_{i_k} s_{i_{k-1}} \cdots s_{i_2} s_{i_1} \alpha_j - s_{i_{k-1}} \cdots s_{i_2} s_{i_1} \alpha_j  \in \RR_{>0} \alpha_{i_k}$. Thus, $s_{i_r} \cdots s_{i_2} s_{i_1} \alpha_j$ is a positive expression if and only if $\alpha_j \precR s_{i_1} \alpha_j \precR s_{i_2} s_{i_1} \alpha_j \precR \cdots \precR s_{i_r} \cdots s_{i_2} s_{i_1} \alpha_j$.
Our purpose in discussing the root poset is to prove foundational results about positive expressions.

\begin{remark}
There is another relation which is sometimes called the ``root poset" and which we do not consider outside this remark.
That relation is defined by $\beta \preceq_{\text{other}} \gamma$ if, writing $\beta = \sum b_i \alpha_i$ and $\gamma = \sum c_i \alpha_i$, we have $b_i \leq c_i$ for all $i$. 
Our relation $\beta \preceqR \gamma$ implies $\beta \preceq_{\text{other}} \gamma$ but not vice versa. 
For example, in type $B_2$, let $\alpha_1$ be the  short root and $\alpha_2$ the long root. Then $\alpha_2 \prec_{\text{other}} \alpha_1 + \alpha_2 \prec_{\text{other}} 2 \alpha_1 + \alpha_2$. 
In $\precR$, the relation $\alpha_2 \precR 2 \alpha_1+\alpha_2$ still holds, but $\alpha_1+\alpha_2$ is incomparable to both of these.
%For example, consider the root system of type $B_2$. The solid edges below are the Hasse diagram of $\preceqR$; the dashed edges are cover relations that occur in $\preceq_{\text{other}}$ but not in $\preceqR$:
%\[ \xymatrix@R=0.125 in{
%& 2 \alpha_1 + \alpha_2 \ar@{--}[dl] \ar@{-}[dd]  \\
%\alpha_1 + \alpha_2 \ar@{-}[d] \ar@{--}[dr] & \\
%\alpha_1\ar@{-}[d] \ar@{--}[dr] & \alpha_2 \ar@{-}[d] \ar@{--}[dl] \\
%-\alpha_1 & -\alpha_2 \\
%-\alpha_1 - \alpha_2 \ar@{-}[u] \ar@{--}[ur] & \\
%& -2 \alpha_1 - \alpha_2 \ar@{--}[ul] \ar@{-}[uu]  \\
%}\]
\end{remark}

We have described $\preceqR$ geometrically in terms of roots, but it can also be described in terms of reflections as described by the following lemma:
\begin{lemma} \label{lem:CombinatorialPoset}
Let $t$ be a reflection and let $s_i$ be a simple reflection not equal to $t$. 
Then $\beta_t - s_i \beta_t \in \RR_{>0} \alpha_i$ if and only if $s_i$ is an inversion of $t$; in this case,  $s_i \beta_t = \beta_{s_i t s_i}$.
\end{lemma}

\begin{proof}
Let $c = (\alpha_i^{\vee}, \beta_t )$. Then $s_i \beta_t = \beta_t - c \alpha_i$, so $\beta_t \precR s_i \beta_t$ if and only if $c>0$. But also, $t \alpha_i = \alpha_i - \tfrac{d_i}{d_\beta} c \beta_t$. Since $t \neq s_i$, the root $\beta_t$ has a nonnegative coefficient on some $\alpha_j$ which is not $\alpha_i$. Then the coefficient of $\alpha_j$ in $t \alpha_i$ has sign matching that of $c$, so $t \alpha_i$ is a negative root if and only if $c>0$.
\end{proof}

Thus, $\precR$ can also be described as a partial order on $T$ which is the transitive closure of $s_its_i \precR t$ if $s_i$ is an inversion of $t$ (for $t \neq s_i$).

We now summarize the main results on $\precR$:
\begin{theorem}[{\cite[Proposition 2.1]{Stem}}]\label{stem1}
The root poset is graded by $\ZZ_{\geq 0}$. The grading function is denoted $d$ and called \newword{depth}, and we have $d(\beta_t) = (\ell(t)-1)/2$.
%So the positive roots have positive depths and the negative roots have negative depths.
 For any root $\beta$ and any simple reflection $s_i$, we have $d(\beta) = d(s_i \beta)$ if and only if $\beta = s_i \beta$. If equality does not hold, then $d(\beta) = d(s_i \beta) \pm 1$; in this case, $s_i \beta$ covers $\beta$ if $d(s_i \beta) = d(\beta)+1$ and $\beta$ covers $s_i \beta$ if $d(s_i \beta) = d(\beta)-1$.
The $\precR$ minimal elements of $\Phi^+$ are the simple roots, $\{ \alpha_1, \alpha_2, \ldots, \alpha_n \}$.
\end{theorem}

%\begin{theorem}[{\cite[Corollary 2.2]{Stem}}]\label{stem2}
%If $\beta \precR s_{i_1} \beta \precR s_{i_2} s_{i_1} \beta \precR \cdots s_{i_k} \cdots s_{i_2} s_{i_1} \beta$ is a saturated chain in the root poset, then the word $s_{i_k} \cdots s_{i_2} s_{i_1}$ is reduced. 
%\end{theorem}

\begin{cor} \label{EverythingAboutPositiveExpressions}
  Let $t$ be a reflection and let $\beta_t$ be the corresponding positive root. 
  Let $s_{i_k} s_{i_{k-1}} \cdots s_{i_1} s_j s_{i_1} s_{i_2} \cdots s_{i_k} = t$. 
  Then the following are equivalent: 
  \begin{enumerate}
\item $s_{i_k} \cdots s_{i_2} s_{i_1} \alpha_j$ is a positive expression for $\beta_t$ \label{condPos}
\item $\alpha_j \precR s_{i_1} \alpha_j \precR s_{i_2} s_{i_1} \alpha_j \cdots \precR s_{i_k} \cdots s_{i_2} s_{i_1} \alpha_j$ in $\precR$.  \label{condChain}
\item $\alpha_j \precR s_{i_1} \alpha_j \precR s_{i_2} s_{i_1} \alpha_j \cdots \precR s_{i_k} \cdots s_{i_2} s_{i_1} \alpha_j $ in $\precR$ and every comparison in this chain is a cover. \label{condSatChain}
\item We have $k = d(\beta_t)$. \label{condDepth}
\item We have $\ell(t) = 2k+1$. \label{condLength}
\item The word $t = s_{i_k} s_{i_{k-1}} \cdots s_{i_1} s_j s_{i_1} \cdots s_{i_{k-1}} s_{i_k}$ is reduced.  \label{condPalindrome}
\end{enumerate}
\end{cor}

\begin{proof}
Note that, from the hypothesis $s_{i_k} s_{i_{k-1}} \cdots s_{i_1} s_j s_{i_1} s_{i_2} \cdots s_{i_k} = t$, we have $s_{i_k} \cdots s_{i_2} s_{i_1} \alpha_j = \pm \beta_t$. Under each condition,  the sign will be $+$; sometimes this will be obvious and sometimes it will take a small argument.

If Condition~\ref{condChain} holds, then in particular $s_{i_k} \cdots s_{i_2} s_{i_1} \alpha_j$ is a positive root, so $s_{i_k} \cdots s_{i_2} s_{i_1} \alpha_j = \beta_t$.
By definition, Conditions~\ref{condPos} and~\ref{condChain} are then equivalent.
From Theorem~\ref{stem1}, if $\gamma \precR s_i \gamma$ then this relation is a cover, showing that Conditions~\ref{condChain} and~\ref{condSatChain} are equivalent.

Assuming Condition~\ref{condSatChain}, we must have $\beta_t = s_{i_k} \cdots s_{i_2} s_{i_1} \alpha_j$, and the depth of $s_{i_k} \cdots s_{i_2} s_{i_1} \alpha_j$ must be the length of the chain, namely $k$.

We now show that Condition~\ref{condDepth} implies~\ref{condSatChain}. 
It is convenient (as Stembridge does) to extend the definition of the root poset to negative roots. This extended poset is now graded by $\ZZ$, with $-\beta_t$ at depth $(-\ell(t)-1)/2$; see \cite[Proposition 2.1]{Stem}. With this extended definition, we can say that, in the sequence of roots $(\alpha_j, \ s_{i_1} \alpha_j,\  s_{i_2} s_{i_1} \alpha_j,\  \cdots,\  s_{i_k} \cdots s_{i_2} s_{i_1} \alpha_j)$, each root is either equal to its neighbor or adjacent to its neighbor in the Hasse diagram of $\precR$, so $1-k \leq d(s_{i_k} \cdots s_{i_2} s_{i_1} \alpha_j) \leq k$. So, assuming Condition~\ref{condDepth}, we must have  $s_{i_k} \cdots s_{i_2} s_{i_1} \alpha_j = \beta_t$ and every step in the sequence must be a cover, showing Condition~\ref{condSatChain}.

We have $d(\beta_t) = (\ell(t)-1)/2$, so the equivalence of~\ref{condDepth} and~\ref{condLength} is immediate. The equivalence of conditions~\ref{condLength} and~\ref{condPalindrome} is also immediate.
\end{proof}

Since the only  $\precR$ minimal elements of $\Phi^+$ are the simple roots, $\{ \alpha_1, \alpha_2, \ldots, \alpha_n \}$, this in particular means
\begin{cor} \label{ShardRecursion}
Any positive root has a positive expression. 
\end{cor}
Corollary~\ref{ShardRecursion} is important to us because we describe bricks of dimension vector $\beta$ recursively, with the recursion taking place along a positive expression for $\beta$.

Given a reflection $t$ and a simple reflection $s_j$, Stembridge defines $x$ to be an \newword{agent of $t$ for $s_j$} if $t = x s_j x^{-1}$ and $\ell(t) = 2 \ell(x)+1$. 
In this vocabulary, we can restate Condition~\ref{condPalindrome} from Corollary~\ref{EverythingAboutPositiveExpressions}:

\begin{cor}
Let $t$ be a reflection and let $\beta_t$ be the corresponding positive root. Then $s_{i_k} \cdots s_{i_2} s_{i_1} \alpha_j$ is a positive expression for $\beta_t$ if and only if $x:=s_{i_k} \cdots s_{i_2} s_{i_1}$ is an agent of $t$ for $s_j$ and $s_{i_k} \cdots s_{i_2} s_{i_1}$ is a reduced word for $x$. In this case, $s_{i_k} \cdots s_{i_2} s_{i_1} s_j s_{i_1} s_{i_2} \cdots s_{i_k}$ is a reduced word for $t$. This is a bijection between positive expressions for $\beta_t$ and palindromic reduced words for $t$. 
\end{cor}

We now discuss the relation between positive expressions for $\beta_t$ and inversions of $t$.
Let $t$ be a reflection. Recall that $\inv_{\Phi}(t)$ is the set of positive roots $\gamma$ such that $t\gamma = \gamma - (\beta_t^{\vee}, \gamma) \beta_t$ is a negative root. 
This implies:
\begin{cor}
The map $\gamma \mapsto - t \gamma$ is an involution on the set $\inv_{\Phi}(t)$; the only fixed point of this involution is $\beta_t$. 
\end{cor}

\begin{proof}
By the definition of an inversion, if $\gamma \in \inv_{\Phi}(t)$, then $-t \gamma$ is a positive root. We have $t(-t \gamma) = - \gamma$, which is a negative root, so $\gamma$ is also in $\inv_{\Phi}(t)$ in this situation. This shows that $\gamma \mapsto - t \gamma$ maps $\inv_{\Phi}(t)$ to itself. We clearly have $-t(-t \gamma) = t^2 \gamma = \gamma$, so this is an involution. 
Let $\gamma$ be fixed by this involution. We know that $t \gamma = \gamma + c \beta_t$ for some scalar $c$, so we have $- \gamma + c \beta_t = \gamma$ or $\gamma = (c/2) \beta_t$. The only positive root proportional to $\beta_t$ is $\beta_t$, so this only happens for $\gamma = \beta_t$. 
\end{proof}

\begin{prop} \label{PairInversions}
Let $t$ be a reflection, let $\beta_t$ be the corresponding positive root, and let $s_{i_r} s_{i_{r-1}} \cdots s_{i_2} s_{i_1} \alpha_j$ be a positive expression for $\beta_t$. 
Set $\delta_k = s_{i_r} \cdots s_{i_{k+1}} \alpha_k$.
 Let $\gamma$ be an element of $\inv_{\Phi}(t)$ other than $\beta_t$. Then exactly one of $\gamma$ and $- t \gamma$ occurs in the list $\{ \delta_1, \delta_2, \ldots, \delta_r \}$. 
\end{prop}

\begin{proof}
We abbreviate $x = s_{i_r} s_{i_{r-1}} \cdots s_{i_2} s_{i_1}$.
Since $s_{i_r} s_{i_{r-1}} \cdots s_{i_2} s_{i_1} \alpha_j$ is a positive expression for $\beta_t$, we know that $s_{i_r} s_{i_{r-1}} \cdots s_{i_1} s_j s_{i_1} \cdots s_{i_{r-1}} s_{i_r}$ is a reduced word for $t$. 
Thus, the $2r+1$ inversions of $t$ are associated to the $r$ positive roots of the form $ s_{i_r} \cdots s_{i_{k+1}} \alpha_k = \delta_k$, to the positive root $s_{i_r} s_{i_{r-1}} \cdots s_{i_2} s_{i_1} \alpha_j = \beta_t$, and to the $r$ positive roots $s_{i_r} s_{i_{r-1}} \cdots s_{i_2} s_{i_1} s_j s_{i_1} s_{i_2} \cdots s_{i_{k-1}} \alpha_k$. In the last case, we have
\begin{multline*} s_{i_r} s_{i_{r-1}} \cdots s_{i_2} s_{i_1} s_j s_{i_1} s_{i_2} \cdots s_{i_{k-1}} \alpha_k = x s_j s_{i_1} s_{i_2} \cdots s_{i_{k-1}} \alpha_k = x s_j x^{-1} s_{i_r} \cdots s_{i_k} \alpha_k = \\
-  x s_j x^{-1} s_{i_r} \cdots s_{i_{k+1}} \alpha_k = - t \delta_k .  
\end{multline*}
In short, the elements of $\inv_{\Phi}(t)$ are $\{ \delta_r, \delta_{r-1}, \ldots, \delta_1, \beta_t, -t \delta_1, \ldots, -t \delta_{r-1}, -t \delta_r \}$ so exactly one of $\{ \gamma, - t \gamma \}$ occurs in the list $\{ \delta_1, \delta_2, \ldots, \delta_r \}$ and the other occurs in the list $\{ -t\delta_1, -t\delta_2, \ldots, -t \delta_r \}$. 
\end{proof}

\begin{remark}
In many cases, given $t$ and $s_j$, there is a unique agent of $t$ for $s_j$. See Theorem~2.6 of~\cite{Stem} for conditions under which this holds; in particular, it holds whenever the Coxeter diagram of $W$ is a forest, hence in all Dynkin types  and in all affine types except for $\widetilde{A}$. However, this is not always the case: In type $\widetilde{A}_2$, both $(s_2 s_3 s_1 s_2 s_3) \alpha_1$ and $(s_3 s_2 s_1 s_3 s_2) \alpha_1$ are positive expressions for $2 \alpha_1 + 3 \alpha_2 + 3 \alpha_3$, but $s_2 s_3 s_1 s_2 s_3 \neq s_3 s_2 s_1 s_3 s_2$.
\end{remark}

\subsection{Rank two subsystems and cutting}

We define a subset $R$ of $\Phi$ to be a \newword{rank two subsystem} if $\Span_{\RR} R$ is two dimensional and $R = \Phi \cap \Span_{\RR} R$. 
We write $R^+ = R \cap \Phi^+$ and $R^- = R \cap \Phi^-$. There are always two roots $\beta_1$ and $\beta_2$ in $R^+$ such that all roots in $R^+$ can be written as $c_1 \beta_1+c_2 \beta_2$ for $c_1$, $c_2 \geq 0$; these are called the \newword{fundamental roots} in $R$.
(See~\cite[Section 2.4]{RS} for a presentation very close to this paper's perspective; see~\cite{Deodhar} and~\cite{Dyer} for early references.)
% If we order $R^+$ according to the ratio $c_1/c_2 \in [0, \infty]$, we get a linear order which is well defined up to reversal. \margin[Will]{I find this linear order construction a little confusing, and if we end up removing the current proof of Lemma \ref{FiniteCutting}, it may not be necessary; we can rephrase any instance of ``is between $x$ and $y$ in the linear order'' in terms of ``is a positive linear combination of $x$ and $y$''. But are there other reasons this perspective is useful?} If $R^+$ is infinite, this order has order type isomorphic to the subset $0 < \tfrac{1}{2} < \tfrac{2}{3} < \tfrac{3}{4} < \cdots < \tfrac{4}{3} < \tfrac{3}{2} < \tfrac{2}{1} < \infty$ of $[0, \infty]$.  We'll write $T(R)$ for the set of reflections $\{ t \in T : \beta_t \in R \}$, the linear ordering of $R$ also gives a linear ordering of $T(R)$.
%
%If $R$ is a rank two  subsystem and $\theta \in V^{\ast}$ is a point not on any of the hyperplanes $\{ \beta^{\perp} : \beta \in R \}$, then $\{ \beta \in R^+ : \langle \theta, \beta \rangle > 0 \}$ will be either an initial or a final subsequence of this order. 
%In particular, if $w$ is an element of $W$ then we can apply this fact to a point in $wD^{\circ}$ to obtain that $\inv(w) \cap T(R)$ is either an initial or final segment of $T(R)$.

\begin{Remark} \label{CombinatorialRankTwo}
We will always discuss rank two sub-root systems using the geometric language above. However, bijecting $\Phi^+$ with $T$, it is possible to define them as subsets of $T$ instead. 
Namely, $R \subset T$ corresponds to a rank two subsystem if (1) the subgroup $\langle R \rangle$ generated by $R$ is dihedral (it may be the infinite dihedral group) and (2) there is not a larger subset $R' \supsetneq R$ such that $\langle R' \rangle$ is dihedral (again, including the  infinite dihedral group). 
If $R \subset T$ corresponds to a rank two subsystem and $p \in R$, then $p$ corresponds to a fundamental root if and only if $\inv(p) \cap R = \{ p \}$. 
\end{Remark}

Let $\beta$ be a positive root and let $R$ be a rank two subsystem containing it. We will say that $R$ \newword{cuts} $\beta^{\perp}$ if $\beta$ is not a fundamental root of $R$.

\begin{prop} \label{CuttingList}
Let $t$ be a reflection, let $\beta_t$ be the corresponding positive root, and let $s_{i_r} \cdots s_{i_2} s_{i_1} \alpha_j$ be a positive expression for $\beta_t$. Then the following sets of rank two subsystems are equal:
\begin{enumerate}
\item The set of rank two subsystems $R$ such that $R$ cuts $\beta_t^{\perp}$.
\item The set of rank two subsystems of the form $\Phi \cap \Span_{\RR}(\beta_t, \gamma)$ where $\gamma$ is an element of $\inv_{\Phi}(t)$ other than $\beta_t$.
\item The set of rank two subsystems of the form $\Phi \cap \Span_{\RR}(\beta_t, \delta_k)$ where $\delta_k =s_{i_r} \cdots s_{i_{k+1}} \alpha_k$.
\end{enumerate}
If $x$ is any agent of $t$ for $s_j$, this is also the same as the set of rank two subsystems of the form $\Phi \cap \Span_{\RR}(\beta_t, \delta)$ for $\delta \in \inv(x)$. 
\end{prop}

\begin{proof}
The last sentence is a restatement of the third numbered condition, since $x$ is an agent of $t$ for $s_j$ if and only if $s_{i_r} \cdots s_{i_2} s_{i_1} \alpha_j$ is a positive expression for $\beta_t$ with $s_{i_r} \cdots s_{i_2} s_{i_1}$ a reduced word for $x$.
So we check equivalence of the numbered conditions.

First, suppose that $R$ cuts $\beta_t^{\perp}$. Let $\beta_1$ and $\beta_2$ be the fundamental roots of $R$ and let $\beta_t = c_1 \beta_1 + c_2 \beta_2$.
We have $- \beta_t = t (\beta_t) = c_1 t \beta_1 + c_2 t \beta_2$, so at least one of $t \beta_1$ and $t \beta_2$ is a negative root, and thus, either $\beta_1$ or $\beta_2$ is an inversion of $t$, without loss of generality say $\beta_1$. Since we assumed that $R$ cuts $\beta_t^{\perp}$, we have $\beta_t \neq \beta_1$. Then $R = \Phi \cap \Span_{\RR}(\beta_1, \beta_t)$. 

Now, suppose that $R$ is of the form $\Phi \cap \Span_{\RR}(\beta_t, \gamma)$ where $\gamma$ is an element of $\inv_{\Phi}(t)$ other than $\beta_t$. Then we also have $R = \Phi \cap \Span_{\RR}(\beta_t, -t\gamma)$.
By Proposition~\ref{PairInversions}, either $\gamma$ or $-t \gamma$ occurs among the $\delta_k$.

Finally, suppose that $R$ is of the form $\Phi \cap \Span_{\RR}(\beta_t, \delta_k)$. Then $\delta_k$ and $- t \delta_k$ are both inversions of $t$ and, in particular are both positive roots. 
The root $\beta_t$ is a positive linear combination of $\delta_k$ and $-t \delta_k$, so $\beta_t$ is not fundamental in $R$.
\end{proof}

\begin{cor} \label{CuttingBound}
Given a positive root $\beta_t$, the number of rank two subsystems cutting $\beta_t^{\perp}$ is at most $(\ell(t)-1)/2$. In particular, it is finite.
\end{cor}

\begin{proof}
In the notation of the previous proposition, $k = (\ell(t)-1)/2$. So $k$ is the number of roots $\{ \delta_1, \delta_2, \ldots, \delta_k \}$, and at most $k$ rank two subsystems cut $\beta_t^{\perp}$. (The number may be less than $k$ because we may have $\Span_{\RR}(\beta_t, \delta_i) = \Span_{\RR}(\beta_t, \delta_j)$ for $i \neq j$.)
\end{proof}

\begin{cor} \label{CuttingRecursion}
Let $\beta \precR s_i \beta$ be a cover in the root poset. Then the set of rank two subsystems cutting $(s_i \beta)^{\perp}$ is the union of
\begin{enumerate}
\item The set of rank two subsystems of the form $s_i R$, where $R$ cuts $\beta^{\perp}$ and
\item The rank two subsystem $\Span_{\RR}(\alpha_i, \beta) \cap \Phi$. 
\end{enumerate}
\end{cor}

\begin{proof}
Let $s_{i_{r-1}} \cdots s_{i_2} s_{i_1} \alpha_j$ be a positive expression for $\beta$.
Then  $s_i s_{i_{r-1}} \cdots s_{i_2} s_{i_1} \alpha_j$ is a positive expression for $s_i \beta$.
We use criterion~(3) from Proposition~\ref{CuttingList}.
Let  $\delta_k =s_{i_{r-1}} \cdots s_{i_{k+1}} \alpha_k$. 
Criterion~(3) shows that $\beta^{\perp}$ is cut by those subsystems of the form $\Phi \cap \Span_{\RR}(\beta, \delta_k)$, and that $(s_i \beta)^{\perp}$ is cut by subsystems of the form $\Phi \cap \Span_{\RR}(s_i \beta, s_i \delta_k)$, and also by $\Phi \cap \Span_{\RR}(s_i \beta, \alpha_i)$. We have $\Span_{\RR}(s_i \beta, s_i \delta_k) = s_i \Span_{\RR}(\beta, \delta_k)$ and  $\Span_{\RR}(s_i \beta, \alpha_i) = \Span_{\RR}(\beta, \alpha_i)$, as desired.
\end{proof}

\begin{remark} \label{CombinatorialCuttingSet}
Let $W$ be a Coxeter group and let $t$ be a reflection in $W$. Then the set of rank two subsystems $R$ cutting $\beta_t$, when considered as subsets of $T$ (using Remark~\ref{CombinatorialRankTwo}) depends only on $W$ and $t$ and not on the choice of Cartan matrix. This is because the notion of $x$ being an agent of $t$ for $s_j$ can be defined without the Cartan matrix (it is equivalent to $t = x s_j x^{-1}$ with $\ell(t) = 2 \ell(x) +1$), the notion of being an inversion of $x$ can be defined without the Cartan matrix, and the notion of the rank two subsytem spanned by two reflections can be defined without the Cartan matrix (see Remark~\ref{CombinatorialRankTwo}). 
\end{remark}

\section{Shards}\label{ShardBackgroundSection}

Let $\beta$ be a positive root and let $R$ be a rank two subsystem containing $\beta$.
Let $\beta^{\perp}$ and $R^{\perp}$ be the $n-1$ dimensional and $n-2$ dimensional subspaces in $V^{\ast}$ orthogonal to $\beta$ and $R$ respectively, so $R^{\perp}$ is a hyperplane in $\beta^{\perp}$. 
For a fixed $\beta$, there are only finitely many rank two subsystems that cut $\beta$ (Corollary~\ref{CuttingBound}); the corresponding linear spaces $R^{\perp}$ are thus an arrangement of finitely many hyperplanes in $\beta^{\perp}$.
We'll call this hyperplane arrangement the \newword{shard arrangement} in $\beta^{\perp}$, and we call the (closed) regions of the hyperplane arrangement complement \newword{shards of $\beta^{\perp}$}.

\begin{eg}
Figure~\ref{shardsfig} depicts the six shards of the $B_2$ reflection arrangement. 
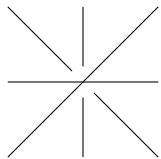
\begin{figure}[h]
  \begin{tikzpicture}
    \draw (-1,1)--(1,-1);
    \draw (0,-1)--(0,1);
    \draw[color=white,fill=white] (0,0) circle (0.2); 
    \draw (-1,0)--(1,0);
    \draw (-1,-1)--(1,1);
  \end{tikzpicture}
  \caption{\label{shardsfig}The six shards of the $B_2$ reflection arrangement}
  \end{figure}
\end{eg}

\begin{remark}
In contrast to Lemma~\ref{lem:CombinatorialPoset} and Remarks~\ref{CombinatorialRankTwo} and~\ref{CombinatorialCuttingSet}, the shards really do depend on the choice of Cartan matrix, not solely on the Coxeter group. See Section~\ref{ShardsSeeCartan} for further discussion.
\end{remark}

\begin{remark}
Shards were introduced by Nathan Reading~\cite{ReadingShard}.
There are two unimportant differences between our terminology and that of Reading. 
First, Reading has a standing hypothesis that $W$ is finite, which we do not need. 
Second, Reading says that $\beta_1$ and $\beta_2$ cut $\gamma^{\perp}$ where $\beta_1$ and $\beta_2$ are the fundamental roots of $R$, rather than saying that $R$ cuts $\gamma^{\perp}$. We believe that the focus on $R$ rather than $\beta_1$, $\beta_2$ is clarifying.
\end{remark}

We now define a very different, recursive description of shards. 
Let $K$ be a convex polyhedral cone in $V^{\ast}$. Recall from the introduction that we define
\[ \shard_i^+(K):= s_i \left( K  \cap \{ x : \langle x, \alpha_i\rangle \geq 0 \} \right) \ \mbox{and}\  \shard_i^-(K):= s_i \left( K \cap \{ x : \langle x, \alpha_i\rangle \leq 0 \}  \right). \]

%\begin{prop} \label{SimpleCuts}
%Let $\beta$ be a positive root and let $\beta = s_{i_r}  \cdots s_{i_2} s_{i_1} \alpha_j$ be a reduced expression for $\beta$.
%Then the set of all shards with normal vector $\beta$ is the set of all $n-1$ dimensional cones of the form
%\[ \shard_{i_r}^{\pm_r}  \cdots \shard_{i_2}^{\pm_2}  \shard_{i_1}^{\pm_1}  (\alpha_j^{\perp})  \]
%where the signs are chosen in all ways that yield an $n-1$ dimensional cone. 
%\end{prop}

\begin{prop} \label{SimpleCutsKey}
Let $\beta \succR \beta'$ be a cover in the root poset, with $\beta = s_i \beta'$. Then the set of shards of $\beta^{\perp}$ is the set of cones $\shard^+_i(K')$ and $\shard_i^-(K')$ which are of dimension $n-1$, where $K'$ ranges over shards of  $(\beta')^{\perp}$.
\end{prop}

\begin{proof}
We have $(s_i R)^{\perp} = s_i (R^{\perp})$.
Thus, by Corollary~\ref{CuttingRecursion}, if $H'_1$, $H'_2$, \dots, $H'_k$ are the hyperplanes forming the shard arrangement in $(\beta')^{\perp}$, then the hyperplanes forming the shard arrangement in $\beta^{\perp}$ are $s_i H'_1$, $s_i H'_2$, \dots, $s_i H'_k$ and $\alpha_i^{\perp} \cap \beta^{\perp}$.
The regions of the hyperplane arrangement formed by  $s_i H'_1$, $s_i H'_2$, \dots, $s_i H'_k$ are simply the images under $s_i$ of the regions of the shard hyperplane arrangement in $\beta'$.
Adding in  $\alpha_i^{\perp} \cap \beta^{\perp}$ cuts each of these regions by intersecting it with the half planes where $\langle -,\alpha_i\rangle$ is positive or negative. Since shards are $(n-1)$-dimensional cones by definition,
the shards of $\beta^{\perp}$ are the $(n-1)$-dimensional cones of the form $s_i \left( K'  \right)  \cap \{ x :  \pm \langle x, \alpha_i\rangle \geq 0 \}$, where $K'$ runs over shards of $\beta'$ and we consider both signs.
This is the same as  $s_i \left( K'    \cap \{ x :  \mp \langle x, \alpha_i\rangle \geq 0 \} \right) = \shard^{\mp}(K')$.
\end{proof}

The statement of Proposition~\ref{SimpleCutsKey} requires that the shards $\shard^{\pm}_i(K')$   be of dimension $n-1$, so we pause to give an example where this fails:

\begin{eg}
\label{D4Example}
Let us consider type $D_4$, with Cartan matrix
\[ \begin{bmatrix}
2&-1&-1&-1 \\
-1 & 2 & 0 & 0 \\
-1 & 0 & 2 & 0 \\
-1 & 0 & 0 & 2 \\ 
\end{bmatrix}. \]
Put $\beta' = s_2 s_3 s_4 \alpha_1 = \alpha_1+\alpha_2+\alpha_3 +\alpha_4$ and $\beta = s_1 \beta' = 2 \alpha_1 + \alpha_2 + \alpha_3 + \alpha_4$.  
The root $\beta'$ is cut by the three subsystems $\{ \alpha_i, \beta', \beta' - \alpha_i \}$ for $i \in \{ 2,3,4 \}$. These define three hyperplanes in the three dimensional space $(\beta')^{\perp}$, cutting $(\beta')^{\perp}$ into eight shards. 
Explicitly, $(\beta')^{\perp} = \{ \sum_{i=1}^4 c_i \omega_i \in V^{\ast} : \sum_{i=1}^4 c_i=0 \}$, and the shards correspond to the eight choices of signs for $(c_2, c_3, c_4)$. 

The six cones where $c_2$, $c_3$ and $c_4$ do not all have the same sign are cut in half by $\alpha_1^{\perp}$, so the reflections of these halves each contribute shards of $\beta$. The two cones where $c_2$, $c_3$ and $c_4$ do have the same sign lie entirely to one side of $\alpha_1^{\perp}$, so they only contribute one shard of $\beta$. To put it another way, for each of these two cones, applying one of $\sigma_1^\pm$ yields the reflection of the shard, while applying the other yields $\{0\}$.
There are, in total, fourteen shards with normal vector $\beta$. 
%
%Note that this does not mean that there are real bricks which are not shards
%in type $D_4$. Consider a cone where $c_2$, $c_3$, and $c_4$ have the same sign.
%The corresponding brick has $S_1$ in either its socle or its top, which means that only one of the two reflection functors can be applied to it. \margin[Hugh]{added paragraph} \margin[Will]{Coming before any post-introduction discussion of representation theory, this seems a little premature, at least without referring back to the definition of shard modules above.}
\end{eg}

Applying Proposition~\ref{SimpleCutsKey} repeatedly, we obtain Theorem~\ref{ShardRecursionTheorem} from the Introduction, which we restate here:

\begin{Theorem} \label{ShardRecursionTheorem}
%\margin{Introduce a systematic way to deal with restated results.} 
Let $\beta$ be any positive root and let $s_{i_r} \cdots s_{i_2} s_{i_1} \alpha_j$ be a positive expression for $\beta$. Then the set of shards of $\beta^{\perp}$ is the set of polyhedral cones of the form $\shard_{i_r}^{\pm_r} \cdots \shard_{i_2}^{\pm_2} \shard_{i_1}^{\pm_1} \left( \alpha_j^{\perp} \right)$ which are of dimension $n-1$, where the signs $\pm_1$, $\pm_2$, \dots, $\pm_r$ may be chosen independently.
\end{Theorem}

\section{Background from representation theory}

\subsection{The preprojective algebra, symmetric case} \label{preproj ssec sym} 

In this section, we explain the standard construction of the preprojective algebra associated to a symmetric Cartan matrix; we will discuss symmetrizable Cartan matrices in the next section. 
Thus, in this section,  let $d_1 = d_2 = \cdots = d_n=1$, so the Cartan matrix $A$ is symmetric. Let $\Gamma$ be the quiver with vertices $1$, $2$, \dots, $n$ and $-A_{ij}$ arrows $i \to j$. We have $-A_{ij} = - A_{ji}$, and we choose a bijection $a \mapsto a^{\ast}$ between the arrows $i \to j$ and those from $j \to i$, with $(a^{\ast})^{\ast} = a$. We write $i \overset{a}{\longrightarrow} j$ or $j \overset{a}{\longleftarrow} i$ to indicate that $a$ is the label of an arrow from $i$ to $j$. 

Fix a ground field $\kappa$, and we write $\kappa \Gamma$ for the path algebra of $\gamma$ over $\kappa$. 
 Our convention for the order of multiplication in the path algebra is that, if $k \overset{b}{\longleftarrow} j$ and $j \overset{a}{\longleftarrow} i$ then $ba$ is the path $k \overset{b}{\longleftarrow} j \overset{a}{\longleftarrow} i$.

Let $\sgn$ be any map from the set of arrows to $\{\pm 1\}$ such that $\sgn(a^\ast) = -\sgn(a)$ for all $a$. Different choices will produce isomorphic algebras.

We define the \newword{preprojective algebra} $\Lambda$ to be the quotient of $\kappa \Gamma$ by the relations
\[ \sum_j \sum_{i \overset{a}{\longrightarrow} j} \sgn(a)a^\ast a = 0\ \mbox{for} \ 1 \leq i \leq n. \]
A module over the preprojective algebra is thus a collection of $\kappa$-vector spaces $M_i$ for $1 \leq i \leq n$ and, for each arrow $i \overset{a}{\longrightarrow} j$, a map $a : M_i \to M_j$ obeying the above relations.

%We define a representation to be \newword{finite dimensional} if all the $M_i$ are finite dimensional over $\kappa$. 
%We encode dimensions of finite dimensional representations by vectors in $V$; for finite dimensional $M$ we put
%\[ \dim M = \sum_i (\dim_{\kappa} M_i) \alpha_i. \]
%If $A$ is of Dynkin type, then $\Lambda$ is a finite dimensional algebra over $\kappa$ and all finitely generated $\Lambda$-modules are finite dimensional, but this does not hold for $A$ not of Dynkin type.
%
%We define a representation to be \newword{nilpotent} if there is a positive integer $N$ such that, for any path $i_0 \overset{a_1}{\longrightarrow} i_1 \overset{a_2}{\longrightarrow} i_2 \longrightarrow \cdots \longrightarrow i_{N-1} \overset{a_N}{\longrightarrow} i_N$ through $\Gamma$, the induced map $M_{i_0} \to M_{i_N}$ is zero. 
%If $A$ is of Dynkin type, then all $\Lambda$-modules are nilpotent.

An easy way to satisfy the preprojective relation is to ensure that, for each pair $i \neq j$, either all the maps $M_i \to M_j$ are zero or vice versa. In that case, $M$ can be thought of as a module over a path algebra for some orientation of $\Gamma$. 
%We will say $M$ is \newword{hereditary} if this is the case.% and we will say $M$ is \newword{acyclic} if the resulting orientation of $\Gamma$ has no directed cycles.

\subsection{The preprojective algebra, symmetrizable case} \label{preproj ssec gen}  It is not difficult to extend the definition of the preprojective algebra to the case where $A$ is symmetrizable, although it is harder to find references for this case. We explain how it should be done. Our approach follows a much more general approach of Kulshammer~\cite{Kulshammer}.
The reader who does not want to think about this may always take all the $d_i$ to be $1$ and all the $\kappa(d_i)$ to be $\kappa$, at the expense of only considering symmetric Cartan matrices.

Let $d_1$, \dots, $d_n$ be positive integers and let $A_{ij}$ be a crystallographic symmetrizable Cartan matrix with $d_i A_{ij} = d_j A_{ji}$. 
Put $d_{ij} = \LCM(d_i, d_j)$ and let $L = \LCM(d_1, \ldots, d_n)$. Let $\kappa(L)/\kappa$ be an Galois extension of fields with Galois group cyclic of order $L$. For $d$ dividing $L$, let $\kappa(d)$ be the unique degree $d$ extension of $\kappa$ within $\kappa(L)$. 
For example, we could take $\kappa = \FF_p$ and $\kappa(L) = \FF_{p^L}$ or, if $L=2$, we could take $\kappa = \RR$ and $\kappa(2) = \CC$.

Let $\Gamma$ be the quiver with vertices $1,\ldots, n$ and $-d_i A_{ij} / d_{ij}$ arrows $i\to j$. Note that this quantity is symmetrical in $i$ and $j$, so we can again choose a bijection $a\mapsto a^\ast$ between arrows $i\to j$ and $j\to i$ with $(a^\ast)^\ast = a$. Let $E(j\ot i)$ be a $\kappa(d_{ij})$-vector space with a basis given by arrows $i\to j$, and define a $\kappa(d_{ij})$-bilinear pairing $E(j\ot i)\times E(i\ot j)\to \kappa(d_{ij})$ such that $a\mapsto a^\ast$ sends each basis to its dual. 

$E(j\ot i)$ is also a $\kappa(d_i)$-vector space and a $\kappa(d_j)$-vector space, by the inclusions of $\kappa(d_i)$ and $\kappa(d_j)$ into $\kappa(d_{ij})$. So by composing with the field trace maps $\tr_{\kappa(d_i)}:\kappa(d_{ij})\to \kappa(d_i)$ and $\tr_{\kappa(d_j)}:\kappa(d_{ij})\to \kappa(d_j)$, we get $\kappa(d_i)$- and $\kappa(d_j)$-bilinear pairings, respectively. This data determines a \newword{dualizable species of algebras} associated to $Q$ in the sense of \cite[Definitions 2.2 and 4.4]{Kulshammer}.

We define the \newword{path algebra} to be 
\[ \bigoplus E(i_N\ot i_{N-1}) \otimes_{\kappa(d_{i_{N-1}}}) E(i_{N-1}\ot i_{N-2}) \otimes_{\kappa(d_{i_{N-2}})} \otimes \cdots \otimes_{\kappa(d_{i_1})} E(i_{1}\ot i_0)  \]
where the direct sum is over all sequences $i_0, i_1, \cdots, i_N$ with adjacent elements distinct. Here a length zero path is still considered to have a given start and end point (which are equal) and the summand corresponding to the length zero path from $i$ to $i$ is $\kappa(d_i)$; we denote the path by $e_i$. Multiplication is defined in the obvious way by concatenating tensors. 
Further, a representation of the path algebra is a collection $V_i$, where $V_i$ is a $\kappa(d_i)$-vector space, with $\kappa(d_j)$-linear maps $V_{j\ot i}:E(j\ot i) \otimes_{\kappa(d_i)} V_i\longrightarrow V_j$.% such that, for $a \in E(j\ot i)$, $x \in V_i$, $f \in \kappa(d_i)$ and $g \in \kappa(d_j)$, we have $a(fx) = (fa)(x)$ and $ga(x) = (ga)(x)$. 

%We fix a perfect $\kappa(d_{ij})$-linear pairing $E(i \to j) \times E(j \to i) \to \kappa(d_{ij})$.\margin{Am I using the right ground fields here?}. 
Now for any pair $i\neq j$, let $b^{ji}_1,\ldots, b^{ji}_p$ be a $\kappa(d_j)$-basis of $\kappa(d_{ij})$, and let $(b^{ji}_1)^\ast,\ldots, (b^{ji}_p)^\ast$ be the dual basis under the trace pairing $(b,b')\mapsto \tr_{\kappa(d_j)}(bb')$. (See \cite[Section VI.5]{Lang} for a quick introduction to the trace pairing.)

Finally, define a map $\sgn$ from pairs of vertices $i, j$ to $\{\pm 1\}$ such that $\sgn(i,j) = -\sgn(j,i)$. 

Then the \newword{preprojective algebra}, $\Lambda$, is the quotient of the path algebra by the relations
\[ \sum_j \sgn(i, j)\sum_{i \overset{a}{\longrightarrow} j}\sum_{k} (b^{ji}_k)^\ast a^{\ast} \otimes a b^{ji}_k = 0 \ \mbox{for} \ 1 \leq i \leq n. \]
This relation does not depend on the above choice of basis \cite[Corollary 4.2]{Kulshammer}, and different choices of $\sgn$ will produce isomorphic algebras \cite[Lemma 4.9]{Kulshammer}.

We define a representation to be \newword{finite dimensional} if all the $M_i$ are finite dimensional over $\kappa$. 
If $A$ is of Dynkin type, then $\Lambda$ is a finite dimensional algebra over $\kappa$ \cite{DR} so all finitely generated $\Lambda$-modules are finite dimensional, but this does not hold for $A$ not of Dynkin type.
For a finite dimensional $\Lambda$-module $M$, we put
\[ \dim M = \sum (\dim_{\kappa(d_i)} M_i) \alpha_i . \]

We define a representation to be \newword{nilpotent} if there is a positive integer $N$ such that, for any path $i_0 \overset{a_1}{\longrightarrow} i_1 \overset{a_2}{\longrightarrow} i_2 \longrightarrow \cdots \longrightarrow i_{N-1} \overset{a_N}{\longrightarrow} i_N$ through $\Gamma$, the induced map $M_{i_0} \to M_{i_N}$ is zero. 
If $A$ is of Dynkin type, then all $\Lambda$-modules are nilpotent.

%We write $S_i$ for the simple module at vertex $i$, so $\dim S_i = \alpha_i$. We write $P_i$ and $I_i$ for the projective cover and injective hull of $S_i$; note that these need not be finite dimensional if $A$ is not of Dynkin type.\margin[Will]{Does this notation actually get used?}

\begin{eg} \label{eg:B2ExampleStarts}
We consider the $B_2$ Cartan matrix $\begin{sbm} \phantom{-}2&-2 \\ -1& \phantom{-}2 \end{sbm}$ with $d_1 = 1$ and $d_2 = 2$. We put $\kappa(1) = \RR$ and $\kappa(2) = \CC$. So a $\Lambda$ module consists of a real vector space $M_1$, a complex vector space $M_2$, a $\CC$-linear map $\CC\otimes_\RR M_1\to M_2$ (which we can identify with an $\RR$-linear map $f:M_1\to M_2$) and a $\RR$-linear map $\CC\otimes_\CC M_2\to M_1$ (which we can identify with an $\RR$-linear map $g:M_2\to M_1$). These maps must satisfy the relations $gf = 0$ and $\frac{i}{2}fgi - \frac12fg = 0$ (the latter is equivalent to $-ifg = fg i$).
%and $\RR$-linear maps $\alpha$, $\alpha' : M_1 \to M_2$ and $\beta$, $\beta': M_2 \to M_1$ such that $\alpha'(x) = i \alpha(x)$ and $\beta'(y) = \beta(iy)$\margin[Will]{Since $\alpha$ and $\beta$ determine $\alpha'$ and $\beta'$, this is oddly redundant to read; it might benefit from more explicitly employing the terminology of $E(i\to j)$.}, obeying $\alpha \beta = 0$ and $\beta \alpha = 0$. 
\end{eg}

The algebra $\Lambda$ satisfies a useful property relating it back to the combinatorics of the Cartan matrix.

\begin{theorem}\label{CB Identity}
  Let $M$ and $N$ be left $\Lambda$-modules, which are finite dimensional over our ground field $\kappa$. Then
%\label{crawley-boevey}
\[
\dim_\kappa(\Hom_\Lambda(M, N)) - \dim_\kappa(\Ext^1_\Lambda(M, N)) + \dim_\kappa(\Hom_\Lambda(N, M)) = (\dim M, \dim N)
\]
\end{theorem}

This theorem is well-known in the symmetric case \cite[Lemma 1]{Crawley-Boevey} but has not been documented in this model of the symmetrizable case. Appendix~\ref{appendix} outlines how to adapt a proof in the symmetric case to this setting.

%\begin{eg}
%We compute the values of $\Ext^j(M_1, M_2)$ for $0 \leq j \leq 2$ and $M_1$ and $M_2$ in Example~\ref{B2ExampleStarts}. The row index is $M_1$ and the column index is $M_2$.
%\[ \begin{array}{|c|cccccc|}
%\hline
%\Hom(M_1,M_2) & S_1 & S_2 & P_1 & I_1 & I_2 & P_2 \\
%\hline
%S_1 & \RR & 0 & 0 & \RR & 0 & \CC \\
%S_2 & 0 & \CC & \CC & 0 & \CC & 0 \\
%P_1 & \RR & 0 & \RR & \RR & \CC & \CC \\
%I_1 & 0 & \CC & \CC & \RR & \CC & 0 \\
%I_2 & \RR^2 & 0 & 0 & \RR^2 & \CC &  \RR^4 \\
%P_2 & 0 & \CC & \CC & \CC & \CC & \CC \\
%\hline
%\end{array} \]
%\end{eg} 

\subsection{Stability}

Let $M$ be a $\Lambda$-module. Let $\theta \in V^{\ast}$. Following \cite{King}, we say that $M$ is $\theta$-\newword{semistable} if $\langle \theta,  \dim M \rangle =0$ and, for any subrepresentation $N$ of $M$, we have $\langle \theta, \dim N \rangle \geq 0$. We say that $M$ is $\theta$-\newword{stable} if the latter inequality is strict for $N \neq 0$, $M$.  We remark that $\theta$-semistability of $M$ is equivalent to asking that  $\langle \theta, \dim Q \rangle \leq 0$ for all quotient representations $Q$ of $M$.

We write $\Stab(M)$ for the set of $\theta\in V^{\ast}$ for which $M$ is semistable and $\Stab^{\circ}(M)$ for those $\theta$ for which $M$ is stable. Since there are only finitely many possible dimension vectors of subrepresentations of $M$, we see that $\Stab(M)$ is a closed polyhedral cone in $(\dim M)^{\perp}$. $\Stab^{\circ}(M)$ is either empty (if $M$ contains a subrepresentation with dimension vector proportional to $\dim(M)$) or is the  interior of $\Stab(M)$ within $(\dim M)^{\perp}$ (otherwise). 
In particular, if $\Stab(M)$ has dimension less than $n-1$, then $\Stab^{\circ}(M)$ is empty.

\begin{remark} \label{QuotientRemark1}
Let $M$ be a $\Lambda$-module and suppose the $\Lambda$-action on $M$ factors through a quotient algebra $R$. Then being an $R$-submodule of $M$ is the same as being a $\Lambda$-submodule, so the notions of stability and semi-stability are not affected by passing to such a quotient. 
\end{remark}

\subsection{Reflection functors} \label{ReflectionFunctorsDefn}

Fix a vertex $i$ of $\Gamma$, and let $S_i$ be the module assigning the space $\kappa(d_i)$ to vertex $i$ and $0$ to all other vertices. Let $\NoSub_i$ be the full subcategory of the category of $\Lambda$-modules consisting of those modules for which $S_i$ is not a submodule and let $\NoQuot_i$ be the full subcategory where $S_i$ is not a quotient. 
%
%We pause for a useful lemma about when modules are in these subcategories.\margin[Will]{Why is this lemma up here, rather than next to Proposition~\ref{CanReflectBricks}?} Proposition~\ref{CanReflectBricks} also states an important result about this subject.

Baumann and Kamnitzer~\cite{BK} introduce, in the symmetric case, mutually inverse equivalences of categories
\[ \Sigma_i : \NoQuot_i \to \NoSub_i \qquad \Sigma_i^{-1} : \NoSub_i \to \NoQuot_i  \] 
and K\"ulshammer~\cite{Kulshammer} generalizes them to a setting including our symmetrizable case. We also have canonical morphisms $\Sigma_i^{-1}(M) \to M$ and $M \to \Sigma_i(M)$.  %which make $\Sigma_i$ and $\Sigma_i^{-1}$ adjoint equivalences.\margin{I always get category vocabulary wrong; Hugh, can you fill in?. }\margin[Will]{I'm not sure if the existence of these morphisms is an instance of any popular named categorical property. Regardless, any equivalence is both left and right adjoint to its inverse, so ``adjoint'' is unnecessary.}
We now describe these functors.

\begin{remark}
Baumann--Kamnitzer and K\"ulshammer define $\Sigma_i$ and $\Sigma_i^{-1}$ on all of $\Mod(\Lambda)$ and then prove their restrictions to $\NoSub_i$ and $\NoQuot_i$ are inverse. We will adopt the different convention that $\Sigma_i(M)$ is only defined for $M \in \NoQuot_i$ and $\Sigma_i^{-1}(M)$ for $M \in \NoSub_i$.
\end{remark}

For $M \in \Mod(\Lambda)$, define $M_{\partial i}$ to be the $\kappa(d_i)$-vector space
\[ M_{\partial i} := \bigoplus_{j \neq i} E(i\ot j) \otimes_{\kappa(d_j)} M_j.\]
We note that
\[ \dim_{\kappa(d_i)} M_{\partial i} = \sum_{j \neq i} (-A_{ij}) \dim_{\kappa(d_j)} M_j. \]

We have a $\kappa(d_i)$-linear map $M_{i, \text{in}}: M_{\partial i}\to M_i$ which, in the $j$th summand, is the map $\sgn(i, j)M_{i\ot j}: E(i\ot j)\otimes_{\kappa(d_j)} M_j\to M_i$.

Additionally, there is a natural isomorphism
\[
\Hom_{\kappa(d_j)}(E(j\ot i)\otimes_{\kappa(d_i)} M_i, M_j) \cong \Hom_{\kappa(d_i)}(M_i, E(i\ot j) \otimes_{\kappa(d_j)} M_j) .
\]

%\begin{align*}
%&\Hom_{\kappa(d_j)}(E(j\ot i)\otimes_{\kappa(d_i)} M_i, M_j) \\
%%&\cong \Hom_{\kappa(d_i)}(M_i, \Hom_{\kappa(d_j)}(E(j\ot i), M_j)) \\
%%&\cong \Hom_{\kappa(d_i)}(M_i, \Hom_{\kappa(d_j)}(E(j\ot i), \kappa(d_j))\otimes_{\kappa(d_j)} M_j) \\
%&\cong \Hom_{\kappa(d_i)}(M_i, E(i\ot j) \otimes_{\kappa(d_j)} M_j) .
%\end{align*}
%sending 
%\begin{align*}
%f &\mapsto (m\mapsto (a\mapsto f(a\otimes m))) \\
%&\mapsto (m\mapsto \sum_{i \overset{a}{\longrightarrow} j}\sum_{k} (b^{ji}_k)^\ast a^{\ast} \otimes f(a b^{ji}_k \otimes m)
%\end{align*}

Given $f:E(j\ot i)\otimes_{\kappa(d_i)} M_i\to M_j$, let $f^\vee: M_i\to E(i\ot j)\otimes_{\kappa(d_j)} M_j$ be the corresponding map under this isomorphism; given $g:M_i\to E(i\ot j)\otimes_{\kappa(d_j)} M_j$, let $g^\wedge: E(j\ot i)\otimes_{\kappa(d_i)} M_i\to M_j$ be the corresponding map on the other side.
The reader who only cares about the symmetric case %can think of $f$ and $f^{\vee}$ as equal, as are $g$ and $g^{\wedge}$, and
can think of all tensor products as being over the same ground field.
%%%%%%%%

Then we additionally have a $\kappa(d_i)$-linear map $M_{i, \text{out}}: M_i\to M_{\partial i}$ which, in the $j$th summand, is the map $M_{j\ot i}^\vee: M_i\to E(i\ot j)\otimes_{\kappa(d_j)} M_j$.
For $M$ to satisfy the relation defining the preprojective algebra is equivalent to $M_{i,\text{in}}\circ M_{i, \text{out}} = 0$ (\cite{Kulshammer}, Proposition 4.12). The condition that $M\in \NoSub_i$ (respectively $\NoQuot_i$) is equivalent to saying that $M_{i, \text{out}}$ (respectively $M_{i, \text{in}}$) is injective (respectively surjective).

%We have a $\kappa(d_i)$-linear map $M_i \to M_{\partial i}$ which, in the $j$-th summand, is the map $M_i \otimes_{\kappa_i} E(i \to j) \to M_j$. We also have a $\kappa(d_i)$ linear map $M_{\partial i} \to M_i$ which in the $j$-th summand is the map $M_j \otimes_{\kappa(d_j)} E(j \to i) \to M_i$. 
%The preprojective relation says the composition $M_i \to M_{\partial i} \to M_i$ is $0$. The condition that $M \in \NoSub_i$ (respectively $\NoQuot_i$) states that the first (respectively second) map is injective (respectively surjective). 

Let $M \in \NoQuot_i$.
We define $\Sigma_i(M)$ as follows: $\Sigma_i(M)_i = \Ker(M_{i, \text{in}})$ and $\Sigma(M)_j = M_j$ for $j \neq i$. 
Since the composition $M_{i, \text{in}}\circ M_{i, \text{out}}$ is $0$,  we have a factorization indicated by the vertical arrow below, and the maps in and out of $\Sigma_i(M)_i$ are given by the diagonal arrows. 
%\[ \xymatrix{
%& \Sigma_i(M)_i \ar@{(->}[dr]^{\Sigma_i(M)_{i, \text{out}}} \ar@{=}[r]& \Ker(M_{i, \text{in}}) & \\
%M_{\partial i} \ar@{->}[r]_{M_{i, \text{in}}} \ar[ur]^{\Sigma_i(M)_{i, \text{in}}} \ar@[ur] & M_i \ar[r]_{M_{i, \text{out}}} \ar@{-->}[u] & M_{\partial i} \ar@{->}[r]_{M_{i, \text{in}}} & M_i \\
%}\]
\[\begin{tikzcd}
& \Sigma_i(M)_i \arrow[hook]{dr}{\Sigma_i(M)_{i, \text{out}}} \arrow[equal]{r}& \Ker(M_{i, \text{in}}) & \\
M_{\partial i} \arrow[->>]{r}[swap]{M_{i, \text{in}}} \ar{ur}{\Sigma_i(M)_{i, \text{in}}}  & M_i \ar{r}[swap]{M_{i, \text{out}}} \arrow[dashed]{u} & M_{\partial i} \arrow[->>]{r}[swap]{M_{i, \text{in}}}  & M_i \\
\end{tikzcd}\]
For $j$, $k \neq i$, the map(s) $\Sigma_i(M)_j \to \Sigma_i(M)_k$ are the same as the map(s) $M_j \to M_k$.

The natural map $M \to \Sigma_i(M)$ is given by the vertical arrow at vertex $i$ and the identity at all other vertices.

Similarly, if $M \in \NoSub_i$, then we define $\Sigma_i^{-1}(M)_i = \CoKer(M_{i, \text{out}})$ and $\Sigma(M)_j = M_j$ for $j \neq i$. 
For $j$, $k \neq i$, the map(s) $\Sigma_i(M)_j \to \Sigma_i(M)_k$ are the same as the map(s) $M_j \to M_k$; the maps to and from $\Sigma_i^{-1}(M)_i$ are defined by the commutativity of the diagram
%\[ \xymatrix{
%M_i \ar@{(->}[r] & M_{\partial i} \ar[r]^{M_{i, \text{in}}} \ar@{->>}[dr]_{\Sigma^-_i(M)_{i, \text{in}}} & M_i \ar@{(->}[r] ^{M_{i, \text{out}}} & M_{\partial i} \\
%& \CoKer(M_{i, \text{out}}) \ar@{=}[r] & \Sigma_i^{-1}(M)_i \ar[ur]_{\Sigma^-_i(M)_{i, \text{out}}} \ar@{-->}[u] & \\
%} \]%\margin[Will]{I believe the middle dashed arrow was pointed the wrong way here; I have reversed it.}
\[ \begin{tikzcd}
M_i \arrow[hook]{r}{M_{i, \text{out}}} & M_{\partial i} \arrow{r}{M_{i, \text{in}}} \arrow[two heads]{dr}[swap]{\Sigma^-_i(M)_{i, \text{in}}} & M_i \arrow[hook]{r}{M_{i, \text{out}}} & M_{\partial i} \\
& \CoKer(M_{i, \text{out}}) \arrow[equal]{r} & \Sigma_i^{-1}(M)_i \arrow{ur}[swap]{\Sigma^-_i(M)_{i, \text{out}}} \arrow[dashed]{u} & \\
\end{tikzcd} \]
The natural map $\Sigma_i^{-1}(M) \to M$ is given by the vertical arrow at vertex $i$ and the identity at all other vertices.

We refer to the functors $\Sigma_1$, $\Sigma_2$, \dots, $\Sigma_n$, $\Sigma^{-1}_1$, $\Sigma^{-1}_2$, \dots, $\Sigma^{-1}_n$ collectively as \newword{reflection functors}. This name is explained by the following result:

\begin{prop} \label{ReflectionsReflect}
Let $M$ be a finite dimensional $\Lambda$ module. Whenever the left hand sides of the following equalities are defined, the equality holds:
\[ \dim \Sigma_i(M) = s_i (\dim M) \qquad \dim \Sigma_i^{-1}(M) = s_i (\dim M). \]
\end{prop}

\begin{proof}
We do the case of $\Sigma_i(M)$, the case of $\Sigma_i^{-1}(M)$ is analogous.
Since $M\in \NoQuot_i$, $M_{\partial i} \to M_i$ is surjective, and we have 
\[ \dim_{\kappa(d_i)}  \Sigma_i(M)_i = \dim_{\kappa(d_i)} M_{\partial i} - \dim_{\kappa(d_i)} M_i = \sum_{j \neq i} -A_{ij} \dim_{\kappa(d_j)} M_j - \dim_{\kappa(d_i)} M_i  \]
and, for $j \neq i$,
\[ \dim_{\kappa(d_j)} \Sigma_i(M)_j = \dim_{\kappa(d_j)}  M_j . \]
Thus
\[ \dim \Sigma_i(M) = \dim M - ( \alpha_i^{\vee}, \dim M ) \alpha_i = s_i(\dim M).  \qedhere \]
\end{proof}

%\margin{I don't know if this will get used; they were in our notes.} Suppose now that $M$ is in both $\NoSub_i(M)$ and $\NoQuot_i(M)$. So $\Sigma_i(M)$ and $\Sigma_i^{-1}(M)$ are both defined and we have maps $\Sigma_i^{-1}(M) \to M \to \Sigma_i(M)$. Tracing through the construction, the first map is an injection, the second a surjection and we have 
%\[ 
%0 \to \Sigma_i^{-1}(M) \to M \to S_i^{ \oplus ( -\alpha_i^{\vee}, \dim M)} \to 0 \qquad 0 \to S_i^{\oplus  ( - \alpha_i^{\vee}, \dim M)} \to M \to \Sigma_i(M) \to 0 .
%\]
%The exponent is nonnegative by Proposition~\ref{alphaCuts}.

\subsection{Bricks} \label{sec:bricks}

Let $B$ be a finite dimensional $\Lambda$-module. Then $B$ is defined to be a \newword{brick} if every nonzero endomorphism of $B$ is invertible. So $\End(B)$ is a division algebra.

\begin{eg} \label{eg:B2ExampleContinues}
Continuing Example~\ref{eg:B2ExampleStarts}, we list the bricks of $\Lambda$ and their dimension vectors:
\[ \begin{array}{|ccc|r@{}c@{}r|}
\hline
 M_1 && M_2 & \dim M&& \\
\hline
 \RR && 0 & \alpha_1 &&\\
 0 && \CC & && \alpha_2 \\
 \RR & \overset{1}{\longrightarrow} & \CC & \alpha_1 &+& \alpha_2 \\
 \RR & \overset{ \mathrm{Re}}{\longleftarrow} & \CC & \alpha_1 &+& \alpha_2 \\
 \CC & \overset{1}{\longrightarrow} & \CC & 2 \alpha_1 &+& \alpha_2 \\
 \CC & \overset{1}{\longleftarrow} & \CC & 2\alpha_1 &+& \alpha_2 \\
\hline
\end{array} \]
The reader can check that the dimension vectors of these representations are precisely the real roots, and the bricks of a given dimension $\beta$ are in bijection with the shards of $\beta^\perp$, as shown in Figure \ref{shardsfig}.
%\margin{WD We haven't shown the example of shards of $B_2$ prior to this, so I'm not sure this is something the reader can note along with. HT Rewritten to avoid talking about the dimension of a shard, and to address Will's comment. DES Current version looks good to me. HT Replaced ``irreducible'' with ``brick''. Good catch, Will.}
%These are not the only indecomposable modules.  For example, we can take $M_1 = \RR \oplus \RR$ and $M_2 = \CC$ with $f = \begin{sbm} 1 & 0 \end{sbm}$ and $g = \begin{sbm} 0 \\ \Im \end{sbm}$. This latter example, however, is not a brick.
\end{eg}

\begin{remark}\label{QuotientRemark2}
If the $\Lambda$ action on $B$ factors through a quotient algebra $R$ of $\Lambda$, then $\End_{\Lambda}(B) = \End_R(B)$, so $B$ is a brick when considered as a $\Lambda$-module if and only if it is when considered as an $R$-module.
\end{remark}

\begin{remark}
In many cases, when $B$ is a brick, $\End(B)$ is a field. For example, if $\kappa$ is algebraically closed, then the only division algebra which is finite dimensional over $\kappa$ is $\kappa$, so $\End(B) = \kappa$ for any brick $B$. If $\kappa$ has trivial Brauer group (for example, if $\kappa$ is finite), then all division algebras which are finite dimensional over $\kappa$ are fields. We will prove in Corollary~\ref{CorOfMT1} that, for our primary focus, the shard modules, we always have $\End(B) \cong \kappa(d_i)$ for some $i$.
\end{remark}

\begin{remark}
In general, it is possible to obtain non-commutative division algebras.
%Let $\Gamma$ be the quiver with two vertices and three arrows in the same direction.
%Let $k = \RR$ and consider representations with dimension vector $(4,4)$. Write $\HH$ for the ring of quaternions. Then we can consider the representation $B$ where $B_1 \cong B_2 \cong \HH$ and the three maps are $q \mapsto q$, $q \mapsto qi$ and $q \mapsto qj$. Then the endomorphism ring of $B$ is $\HH$, acting on $B_1$ and $B_2$ by left multiplication.  
This occurs in affine type $\tilde{C}_2$, with the representation $\CC^2 \from \RR^4 \to \CC^2$ where the left map is $(x_1, x_2, x_3, x_4) \mapsto (x_1+i x_2, x_3-i x_4)$ and the right map is $(x_1, x_2, x_3, x_4) \mapsto (x_1+i x_3, x_2+i x_4)$.
The endomorphism ring of this example is the ring of quaternions $\HH$. (Proof sketch: Think of the middle vector space as $\HH$, with $(x_1, x_2, x_3, x_4)$ being the quaternion $x_1+x_2I+x_3J+x_4K$. Then multiplication by $i$ in the left and the right positions corresponds to $q \mapsto qI$ and $q \mapsto qJ$. An endomorphism of this quiver representation is an element of $\text{End}_{\RR}(\HH)$ which commutes with both of these; such endomorphisms are of the form $q \mapsto (a+bI+cJ+dK)q$.)  We have defined this brick as a representation of the $\widetilde C_2$ path algebra. We can lift it to a representation of the corresponding preprojective algebra by letting the reverse arrows act by zero. As discussed in Remark \ref{QuotientRemark2}, the lifted representation is still a brick, but, in agreement with Corollary \ref{CorOfMT1}, it is not a real brick, as $(\dim B,\dim B)=0$.
\end{remark}

We record some basic observations:

%\margin[Hugh]{Removed the proposition here which said that if $\Stab^\circ B$ is non-empty, then $B$ is a brick, since it is wrong. (The sub and quotient module could be a scalar multiple of the dimension of $B$.) Fortunately we did not use this proposition.}
%\begin{prop} \label{WallsAreBricks}
%If $\Stab^{\circ}(B)$ is nonempty, then $B$ is a brick.
%\end{prop}

%\begin{proof}
%  We prove the contrapositive. Suppose that $B$ is not a brick. Let $\phi: B \to B$ be a non-zero, non-invertible, endomorphism of $B$. Put $N = \Im(\phi)$. Then $N$ is both a submodule and a quotient module of $B$ so $\langle \ , \dim N\rangle$ is both positive and negative on $\Stab^{\circ}(B)$, and thus $\Stab^{\circ}(B)$ is empty.
%  \end{proof}

\begin{prop} \label{CanReflectBricks}
Let $B$ be a brick, $i$ a vertex of $\Lambda$ and assume $B \not \cong S_i$. Then $B$ is in at least one of $\NoSub_i$ and $\NoQuot_i$.
\end{prop}

\begin{proof}
We prove the contrapositive. Suppose there is an inclusion $S_i \into B$ and a surjection $B \onto S_i$. Then the composition $B \onto S_i \into B$ is a non-zero endomorphism of $B$. So either $B$ is not a brick, or this composition is an isomorphism; in the latter case, $B \cong S_i$.
\end{proof}

\begin{prop} \label{alphaCuts}
If $(\alpha_i, \dim M) > 0$ then $M$ is in at most one of $\NoSub_i$ and $\NoQuot_i$.
\end{prop}

\begin{proof}
By Theorem \ref{CB Identity},
\[ 2 (\alpha_i, \dim M) = \dim_{\kappa} \Hom(S_i, M) - \dim_{\kappa} \Ext^1(S_i, M) + \dim_{\kappa} \Hom(M, S_i) . \]
So, if $(\alpha_i, \dim M)>0$ then at least one of $\Hom(S_i, M)$ and $\Hom(M, S_i)$ is nonzero. Since $S_i$ is simple, any nonzero map from $S_i$ is an injection and any nonzero map to $S_i$ is a surjection.
\end{proof}

  Proposition~\ref{CanReflectBricks} tells us that we can often apply reflection functors to bricks.
  We now consider interactions between bricks and reflection functors.
Since $\Sigma_i^{\pm}$ is an equivalence of categories, we have
\[ \End(B) \cong \End(\Sigma_i(B)) \ \mbox{and} \  \End(B) \cong \End(\Sigma^{-1}_i(B)) \]
whenever these are defined. Thus
\begin{prop} \label{ReflectBrick}
The reflection functors preserve the property of being a brick.
\end{prop}

%We will call $B$ a \newword{real brick} if $B$ is a brick and $(\dim B, \dim B) > 0$. 
Combining Propositions~\ref{ReflectionsReflect} and~\ref{ReflectBrick} we deduce:
\begin{prop} \label{ReflectRealBrick}
The reflection functors preserve the property of being a real brick.
\end{prop}

This is a good time to prove Proposition~\ref{RealBrickProperties}, which gives alternative ways to think about the condition that $\dim B$ is a real root.
We restate  Proposition~\ref{RealBrickProperties} for the reader's convenience.

\begin{proposition} \label{RealBrickProperties}
Let $B$ be a brick. The following are equivalent:
\begin{enumerate}
\item The vector $\dim B$ in $V$ is a real root. 
\item We have  $(\dim B, \dim B)>0$.
\item The brick $B$ is rigid, meaning that $\Ext^1_{\Lambda}(B,B) = 0$. 
\end{enumerate}
\end{proposition}

\begin{proof}
The implication $(1) \implies (2)$ is immediate, since we have $(\beta, \beta) > 0$ for any real root $\beta$.

The implication $(3) \implies (2)$ is almost as quick: By Theorem \ref{CB Identity}, we have $(\dim B, \dim B) = \dim_{\kappa} \End(B) - \dim_{\kappa} \Ext^1_{\Lambda}(B,B) + \dim_{\kappa} \End(B)$. Under hypothesis~$(3)$, the right hand side is $2 \dim_{\kappa} \End(B) > 0$.

We now show that $(2)$ implies $(1)$ and $(3)$. Put $\beta = \dim B = \sum c_j \alpha_j$, so $c_j \geq 0$. 
We will prove these statements by induction on $\sum c_j$. In the base case, $\sum c_j =1$, we know that $\beta$ is a simple root, so $B$ is a simple module and $(1)$ and $(3)$ are clear.
Since we are assuming $(\beta, \beta)>0$, we have $\sum c_j ( \alpha_j, \beta) > 0$. So there is some index $i$ for which $(\alpha_i, \beta) > 0$, and we fix $i$ to refer to such an index. 

Since $B$ is a brick, we have either $B \in \NoSub_i$ or $B \in \NoQuot_i$, by Proposition \ref{CanReflectBricks}. We treat the case that $B \in \NoSub_i$, and the other case is similar. Put $B' = \Sigma_i^{-1}(B)$ and $\beta' = s_i \beta$. Then $B'$ is a brick of dimension $\beta'$. Note that $(\beta', \beta') = (\beta, \beta) > 0$. Also, $\beta' = \beta - (\alpha_i^{\vee}, \beta) \alpha_i$ and $(\alpha_i^{\vee}, \beta)>0$ so, by induction, $\beta'$ is a real root and $\Ext^1_{\Lambda}(B', B')=0$.

Then $\beta = s_i \beta'$ is a real root as well. Since $\Sigma_i^{\pm}$ is an equivalence of categories, we have $\End(B) \cong \End(B')$. 
Since $(\beta, \beta) = 2 \dim_{\kappa} \End(B) - \dim_{\kappa} \Ext^1(B,B)$ and $(\beta', \beta') = 2 \dim_{\kappa} \End(B') - \dim_{\kappa} \Ext^1(B',B')$, the inductive fact that $\Ext^1(B', B')=0$ implies $\Ext^1(B,B)=0$.
\end{proof}

%
%Since $B$ is a brick, we have either $B \in \NoSub_i$ or $B \in \NoQuot_i$. We treat the case that $B \in \NoSub_i$, and the other case is similar. Put $B' = \Sigma_i^{-1}(B)$ and $\beta' = s_i \beta$. Then $B'$ is a real brick of dimension vector $\beta'$. By induction, $(2)$ and $(3)$ hold for $B'$. It follows immediately that $(2)$ holds for $B$, since $s_i$ maps positive roots (other than $\alpha_i$) to positive roots. To establish (3), note that $(\beta, \beta) = (\beta', \beta')$ (since $s_i$ preserves the inner product) and $\End(B) \cong \End(B')$ (since $\Sigma_i^{\pm}$ are equivalences of categories). We have $(\beta, \beta) = 2 \dim_{\kappa} \End(B) - \dim_{\kappa} \Ext^1(B,B)$ and $(\beta', \beta') = 2 \dim_{\kappa} \End(B') - \dim_{\kappa} \Ext^1(B',B')$, so $\Ext^1(B', B')=0$ implies $\Ext^1(B,B)=0$.
%\end{proof}

\section{Proofs of the main theorems}

\subsection{Proof of Theorem~\ref{RealBrickRecursion}} \label{ProofOfRealBrickRecursion}
We now prove Theorem~\ref{RealBrickRecursion} from the Introduction, which we restate for the reader's convenience:
\begin{Theorem}\label{RealBrickRecursion}
Let $\beta$ be a positive root and let $s_{i_r} \cdots s_{i_2} s_{i_1} \alpha_j$ be a positive expression for $\beta$. The bricks of dimension $\beta$ are precisely the modules of the form $\Sigma^{\pm_r}_{i_r} \cdots \Sigma^{\pm_2}_{i_2} \Sigma^{\pm_1}_{i_1} S_j$, where the signs must be chosen such that the expression is well-defined.
\end{Theorem}

\begin{proof}
By Proposition~\ref{ReflectRealBrick}, any module of the form $\Sigma_{i_r}^{\pm_r} \cdots \Sigma_{i_2}^{\pm_2} \Sigma_{i_1}^{\pm_1} (S_j)$ is a real brick and, by Proposition~\ref{ReflectionsReflect}, has dimension $s_{i_r}  \cdots s_{i_2} s_{i_1} (\alpha_j) = \beta$. So all modules of this form are real bricks of the correct dimension, and our task is to prove the converse.

Our proof is by induction on $r$. In the base case $r=0$, we have $\beta = \alpha_j$, so the only module of dimension $\alpha_j$ is $S_j$, which is easily seen to be a brick. 

We now consider the inductive case $r>0$. We abbreviate $\beta' =  s_{i_{r-1}}  \cdots s_{i_2} s_{i_1} (\alpha_j)$ and $i = i_r$. So $\beta = s_i \beta'$
and $\beta \succR \beta'$. 
%Note that this implies that $(\alpha_i, \beta) > (\alpha_i, \beta')$. But we also have $(\alpha_i, \beta) = -  (\alpha_i, \beta')$. \margin{Probably make this a lemma somewhere.} So $( \alpha_i, \beta ) > 0 > (\alpha_i, \beta')$. \margin[Will]{Why are these last three sentences necessary for the proof?}

Let $B$ be a brick of dimension $\beta$. By Proposition~\ref{CanReflectBricks}, we know that $B$ is in at least one of $\NoSub_i$ and $\NoQuot_i$. We'll treat the case that $B \in \NoSub_i$; the other case is similar. Put $B' = \Sigma_i^{-1}(B)$. Then $B'$ is, by Propositions~\ref{ReflectRealBrick} and~\ref{ReflectionsReflect}, a brick of dimension $\beta'$. By induction, $B'$ is of the form $\Sigma_{i_2}^{\pm_2} \cdots \Sigma_{i_t}^{\pm_t} (S_j)$, and then $B = \Sigma_i(B) = \Sigma_{i_1} \Sigma_{i_2}^{\pm_2} \cdots \Sigma_{i_t}^{\pm_t} (S_j)$.
\end{proof}

%\subsection{Properties of Real Bricks} \label{CorsOfRealBrickRecursion}
%
%We can now prove several properties of real bricks as easy Corollaries of Theorem~\ref{RealBrickRecursion}.
%
\begin{cor} \label{CorOfMT1}
Let $B$ be a real brick with dimension $\beta$ and recall that $d_{\beta} = \beta/\beta^{\vee} = (\beta, \beta)/2$. Then $\End(B) \cong \kappa(d_{\beta})$. 
\end{cor}

In particular, non-commutative division rings do not appear as endomorphism rings of real bricks.

\begin{proof}
Let
\[ B \cong \Sigma_{i_1}^{\pm} \Sigma_{i_2}^{\pm} \cdots \Sigma_{i_t}^{\pm} (S_j) \]
Since the $\Sigma_i^{\pm}$ are equivalences of categories (between $\NoQuot_i$ and $\NoSub_i$), we have $\End(B) \cong \End(S_j) \cong \kappa(d_j)$.
We just need to check that $d_{\beta} = d_j$. Indeed, we have $\beta = s_{i_1} s_{i_2} \cdots s_{i_t} \alpha_j$ so the result follows from Proposition~\ref{prop:dbeta}.
\end{proof}

%\begin{cor} \label{CorOfMT2}
%Let $B$ be a real brick with dimension $\beta$. Then $\Ext^1(B,B)=0$. 
%\end{cor}
%
%\begin{proof}
%Let $d = \beta^{\vee}/\beta$.
%We have
%\[ \dim \Hom(B,B) - \dim \Ext^1(B,B) + \dim \Hom(B,B) = (\beta, \beta) = 2d_{\beta}. \]
%Corollary~\ref{CorOfMT1} shows that $\dim \Hom(B,B) = d$, so  $\dim \Ext^1(B,B) =0$.
%\end{proof}
%
\subsection{Proof of Theorem~\ref{StabOfAShardModule}} \label{ProofOfStabOfAShardModule}

We now prove Theorem~\ref{StabOfAShardModule}, which we restate for the reader's convenience:
\begin{Theorem}\label{StabOfAShardModule}
Let $\beta$ be a positive root, let $s_{i_r} \cdots s_{i_2} s_{i_1} \alpha_j$ be a positive expression for $\beta$ and let $\pm_1$, $\pm_2$, \dots, $\pm_r$ be a choice of signs such that $\Sigma^{\pm_r}_{i_r} \cdots \Sigma^{\pm_2}_{i_2} \Sigma^{\pm_1}_{i_1} S_j$ is defined. Then the brick $\Sigma^{\pm_r}_{i_r} \cdots \Sigma^{\pm_2}_{i_2} \Sigma^{\pm_1}_{i_1} S_j$ has stability domain $\shard_{i_r}^{\pm_r} \cdots \shard_{i_2}^{\pm_2} \shard_{i_1}^{\pm_1} \left( \alpha_j^{\perp} \right)$.
\end{Theorem}

This will follow by an immediate recursion once we prove the following:
\begin{prop}\label{StabOfAShardModuleIndStep}
Let $\beta$ and $\beta'$ be positive roots with $\beta = s_i \beta'$ and $\beta - \beta' \in \RR_{>0} \alpha_i$. Let $B'$ be a brick of dimension $\beta'$ and let $B = \Sigma_i^{\pm} B'$, where we assume that $\Sigma_i^{\pm} B'$ is well-defined. Then $\Stab(B) = \shard_i^{\pm} \Stab(B')$.
\end{prop}

We will cover the case where $\pm$ is $+$; the other sign is similar.
This comes down to proving the following two results:

\begin{prop}\label{StabOfAShardModuleContainment}
Let $\beta \succR s_i \beta'$ be a cover in the root poset. Let $B'$ be a brick of dimension $\beta'$ in $\NoQuot_i$ and let $B = \Sigma_i B'$. Then
\[
\Stab(B) \subset \{\theta : \langle\theta, \alpha_i\rangle \leq 0\}
\] 
\end{prop}

\begin{proof}
Because $\beta = s_i\beta'$, we also know $\beta' = s_i\beta = \beta - (\alpha_i^\vee, \beta)\alpha_i$. By assumption, $(\alpha_i^\vee, \beta) > 0$, and so $(\alpha_i, \beta) > 0$. By Proposition~\ref{alphaCuts}, $B$ is in at most one of $\NoSub_i$ and $\NoQuot_i$. Since $B$ is an output of $\Sigma_i$, it lies in $\NoSub_i$, so it cannot be in $\NoQuot_i$. Thus $B$ has $S_i$ as a quotient, and so for any $\theta$ such that $B$ is $\theta$-semistable, $\langle\theta, \alpha_i\rangle \leq 0$. 
\end{proof}

\begin{lem} \label{KeyStabilityReflectionLemma}
Let $\beta$ and $\beta'$ be positive roots with $\beta = s_i\beta'$. Let $B'$ be a brick of dimension $\beta'$ in $\NoQuot_i$, and let $B = \Sigma_i B'$. Then
\[
\sigma^+_i\Stab(B') := s_i(\Stab(B')\cap \{\theta : \langle\theta, \alpha_i\rangle\geq 0\}) = \Stab(B)\cap \{\theta : \langle\theta, \alpha_i\rangle \leq 0\}
\]
\end{lem}

%%%%%%%%%%%%%%%%

%\begin{lem} \label{KeyStabilityReflectionLemma}
%Let $\beta$ and $\beta'$ be positive roots with $\beta = s_i\beta'$. Let $B'$ be a brick of dimension $\beta'$ in $\NoQuot_i$. Let $B = \Sigma_i B'$, and suppose $B\notin \NoQuot_i$. Then $\Stab(B) = \sigma^+_i(\Stab(B'))$.
%\end{lem}

%%%%%%%%

%\begin{lem} 
%Let $s_i$, $\beta$ and $\beta'$ be as above. Let $B' \in \NoQuot_i$ and let $B = \Sigma_i B'$. Then  $\Stab(B) = \shard_i(\Stab(B'))$.
%\end{lem}

Given a module $M$, a submodule $K$ and a weight $\theta \in (\dim M)^{\perp}$, we say that $K$ is \newword{$\theta$-destabilizing} if $\langle \theta, \dim K \rangle < 0$. 
So $M$ is $\theta$-semistable if and only if it has no $\theta$-destabilizing submodule. 
Similarly, we say that a quotient module $Q$ of $M$ is \newword{$\theta$-destabilizing} if $\langle \theta, \dim Q \rangle >0$, and it is likewise true that $M$ is $\theta$-semistable if and only if it has no $\theta$-destabilizing quotient.

%FROM HERE, CHANGING SIGNS

\begin{proof}[Proof of Lemma \ref{KeyStabilityReflectionLemma}]
%To spell out the desired conclusion, we must show that (1) if $B$ is $\theta$-semistable then $\langle \alpha_i, \theta \rangle \leq 0$ and (2) given $\theta$ with  $\langle \alpha_i, \theta \rangle \leq 0$, the module $B$ is $\theta$-semistable if and only if $B'$ is $s_i \theta$-semistable. 

%To spell out the desired conclusion, we must show that, given $\theta$ with $\langle \theta, \alpha_i\rangle \leq 0$, the module $B$ is $\theta$-semistable if and only if $B'$ is $s_i\theta$-semistable. From now on, let $\theta$ be a weight with $\langle  \alpha_i, \theta \rangle \leq 0$ and $\langle \beta, \theta \rangle = 0$. Set $\theta' = s_i \theta$.

%We first check that, if $B$ is $\theta$-semistable then $\langle \alpha_i, \theta \rangle \leq 0$.
%Since $B$ is an output of $\Sigma_i$, the brick $B$ lies in $\NoSub_i$. However, $(\alpha_i, \beta) > 0$ so, by Proposition~\ref{alphaCuts}, $B$ cannot be in both $\NoSub_i$ and $\NoQuot_i$, so $S_i$ is a quotient module of $B$\margin[Will]{Changed ``$S_i$'' to ``$B$''.}. Thus, if $B$ is $\theta$-semistable, we must have $\langle  \alpha_i, \theta \rangle \leq 0$.

%The first condition follows from the assumption that $S_i$ is a quotient module of $B$, so we check the second condition.
%We now check the second condition.
%From now on, let $\theta$ be a weight with $\langle  \alpha_i, \theta \rangle \leq 0$ and $\langle \beta, \theta \rangle = 0$. Set $\theta' = s_i \theta$.

Let $\theta$ be any weight with $\langle \theta, \alpha_i\rangle \leq 0$ and $\langle \theta, \beta\rangle = 0$, and let $\theta' = s_i\theta$. Then the desired conclusion is that the module $B$ is $\theta$-semistable if and only if $B'$ is $\theta'$-semistable.

First, suppose that $B$ has a $\theta$-destabilizing submodule $K$, so $\langle \theta, \dim K\rangle < 0$.
We will show that $B'$ has a $\theta'$-destabilizing submodule. 
Since $B \in \NoSub_i$, we also have $K \in \NoSub_i$. So we may apply the functor $\Sigma_i^-$ to the inclusion $K \into B$, producing a new module $K' := \Sigma_i^{-1} K$ and a map $K' \to B'$. 
Recall that, for any $X \in \NoSub_i$, we have a natural transformation $\Sigma_i^{-1}(X) \to X$ with kernel and cokernel of the form $S_i^{\oplus a}$. Define $A$ to be the kernel of $K' \to B'$ and let $I$ be the image. So we have a commutative diagram with exact rows:
%\[ \xymatrix{
%&& A \ar@{(->}[d] \ar@{(-->}[dl] &&& \\
%0 \ar[r] & S_i^{\oplus a} \ar[r] & K' \ar[r] \ar@{->>}[d]  & K \ar[r] \ar@{(->}[dd] & S_i^{\oplus b} \ar[r] & 0 \\
%&& I \ar@{(->}[d] &&& \\
%0 \ar[r] & S_i^{\oplus c} \ar[r] & B' \ar[r] & B \ar[r] & S_i^{\oplus d} \ar[r] & 0 
%} \]
\[ \begin{tikzcd}
&& A \arrow[hook]{d} \arrow[hook, dashed]{dl} &&& \\
0 \arrow{r} & S_i^{\oplus a} \arrow{r} & K' \arrow{r} \arrow[two heads]{d}  & K \arrow{r} \arrow[hook]{dd} & S_i^{\oplus b} \arrow{r} & 0 \\
&& I \arrow[hook]{d} &&& \\
0 \arrow{r} & S_i^{\oplus c} \arrow{r} & B' \arrow{r} & B \arrow{r} & S_i^{\oplus d} \arrow{r} & 0 
\end{tikzcd} \]

Now, the composite $A \into K' \to B' \to B$ is $0$, since $A$ is defined as the kernel of $K' \to B'$, so the composite $A \to K' \to K \into B$ is likewise zero. As $K$ injects into $B$, this means the composite $A \to K' \to K$ is $0$, so the map $A \to K'$ factors through the kernel of $K \to K'$, as shown by the dashed arrow.  So $A$ is of the form $S_i^k$ for some $0 \leq k \leq a$. We deduce that $\dim I = \dim K' - k \alpha_i = s_i (\dim K) - k \alpha_i$. 
Then
\[ \langle \theta', \dim I \rangle = \langle s_i(\theta), s_i(\dim K) - k \alpha_i \rangle = \langle \theta, \dim K \rangle + k \langle \theta, \alpha_i \rangle < 0 . \]
So $I$ is $\theta'$-destabilizing.

We now dualize the argument to show the converse half. % of condition (2).
Suppose that $Q'$ is a $\theta'$-destabilizing quotient of $B'$, so $\langle \theta', \dim Q' \rangle > 0$.
Since $B' \in \NoQuot_i$, we also have $Q' \in \NoQuot_i$ and may apply $\Sigma_i$ to the surjection $B' \onto Q'$, producing a new module $Q := \Sigma_i(Q)$ and a map $B \to Q$. Let $I$ be the image of this map and let $C$ be the cokernel.

\[ \begin{tikzcd}
0 \arrow{r} & S_i^{\oplus c} \arrow{r} & B' \arrow{r} \arrow[two heads]{dd} & B \arrow{r} \arrow[two heads]{d}  & S_i^{\oplus d} \arrow{r} & 0 \\
&&& I \arrow[hook]{d} && \\
0 \arrow{r} & S_i^{\oplus e} \arrow{r} & Q' \arrow{r} & Q \arrow{r}  \arrow[two heads]{d}  & S_i^{\oplus f} \arrow{r} \arrow[dashed, two heads]{dl} & 0  \\
&&& C && \\
\end{tikzcd} \]
%
%\[ \xymatrix{
%0 \ar[r] & S_i^{\oplus c} \ar[r] & B' \ar[r] \ar@{->>}[dd] & B \ar[r] \ar@{->>}[d]  & S_i^{\oplus d} \ar[r] & 0 \\
%&&& I \ar@{(->} [d] && \\
%0 \ar[r] & S_i^{\oplus e} \ar[r] & Q' \ar[r] & Q \ar[r]  \ar@{->>}[d]  & S_i^{\oplus f} \ar[r] \ar@{-->>}[dl] & 0  \\
%&&& C && \\
%} \]
The composite $B' \to B \to Q \to C$ is $0$, so $B' \onto Q' \to Q \to C$ is zero.
Since $B'$ surjects on $Q'$, this means that $Q' \to Q \to C$ is zero, so $C$ is a quotient of the cokernel $S_i^{\oplus f}$, as shown by the dashed arrow.

Thus, $C$ is of the form $S_i^{\oplus \ell}$ for some $0 \leq \ell \leq f$. We compute that $\dim I = \dim Q - \ell \alpha_i$.
So 
\[ \langle \theta, \dim I \rangle = \langle \theta, s_i (\dim Q') - \ell \alpha_i \rangle = \langle \theta', \dim Q' \rangle - \ell \langle \theta, \alpha_i \rangle >0\]
and $I$ is $\theta$-destabilizing, as desired. \end{proof}

\begin{proof}[Proof of Proposition \ref{StabOfAShardModuleIndStep}]
By Proposition \ref{StabOfAShardModuleContainment}, under the assumptions of Proposition \ref{StabOfAShardModuleIndStep}, the right side of the equation in Lemma \ref{KeyStabilityReflectionLemma} is just $\Stab(B)$.
\end{proof}

%\margin{Added this restatement from the introduction, because it me it looked weird having that one theorem left kind of orphaned with a number in section 1, when all the others had higher numbers.} 
%We have now established Theorem~\ref{StabOfAShardModule}. In combination with Theorem~\ref{RealBrickRecursion},
%we immediately obtain

%\begin{Theorem} \label{ShardModuleRecursion}
%Let $\beta$ be a positive root and let $s_{i_r} \cdots s_{i_2} s_{i_1} \alpha_j$ be a positive expression for $\beta$.
%The shard modules  of dimension $\beta$ are precisely those modules  $\Sigma^{\pm_r}_{i_r} \cdots \Sigma^{\pm_2}_{i_2} \Sigma^{\pm_1}_{i_1} S_j$  which are well defined and for which  $\shard_{i_r}^{\pm_r} \cdots \shard_{i_2}^{\pm_2} \shard_{i_1}^{\pm_1} \left( \alpha_j^{\perp} \right)$ has dimension $n-1$.
%\end{Theorem}

\subsection{Proof of Theorem~\ref{StabsAreShards}} \label{ProofOfStabsAreShards}

We now fill in the final steps of the proof of Theorem~\ref{StabsAreShards}, which we restate for the reader's convenience.
\begin{Theorem}\label{StabsAreShards}
The stability domains of shard modules are precisely the shards, and each shard is the stability domain of precisely one isomorphism class of shard modules.
\end{Theorem}

We have already done all the work to show that the stability domain of a shard module is always a shard. 
Let $B$ be a shard module. Then, in particular, $B$ is a real brick so, by Theorem~\ref{RealBrickRecursion}, $B$ is of the form $\Sigma^{\pm_r}_{i_r} \cdots \Sigma^{\pm_2}_{i_2} \Sigma^{\pm_1}_{i_1} S_j$. By Theorem~\ref{StabOfAShardModule}, we have $\Stab(B) = \shard_{i_r}^{\pm_r} \cdots \shard_{i_2}^{\pm_2} \shard_{i_1}^{\pm_1} \left( \alpha_j^{\perp} \right)$ and, by the definition of a shard module, this has dimension $n-1$.  By Theorem~\ref{ShardRecursionTheorem}, this is a shard.

We now address the converse: Starting with a shard $K$, we want to show that there is a shard module $B$ with $\Stab(B) = K$ and that $B$ is unique up to isomorphism.
We prove existence first, and then uniqueness.

Write the shard $K$ as $\shard_{i_r}^{\pm_r} \cdots \shard_{i_2}^{\pm_2} \shard_{i_1}^{\pm_1} \left( \alpha_j^{\perp} \right)$ where $s_{i_r} \cdots s_{i_2} s_{i_1} \alpha_j$ is a positive expression for $\beta$. 
We will show by induction on $r$ that there is a brick with stability domain $K$. For the base case, $r=0$, note $\Stab(S_j) = \alpha_j^{\perp}$.
For the inductive case, abbreviate $\shard_{i_{r-1}}^{\pm_{r-1}} \cdots \shard_{i_2}^{\pm_2} \shard_{i_1}^{\pm_1} \left( \alpha_j^{\perp} \right)$ to $K'$, and let $B'$ be a brick with this stability domain. 
We'll analyze the case that $\pm_r = +$ (the other sign is analogous) and we abbreviate $i_r$ to $i$. 

If we can show that $B' \in \NoQuot_{i}$, then $\Sigma_{i} (B')$ will be a brick with the required properties. So, suppose for the sake of contradiction that $S_{i}$ is a quotient of $B'$. Then $K' = \Stab(B') \subset \{ \phi : \langle \phi, \alpha_{i} \rangle \leq 0 \}$.  But then $\shard_{i}^{+}(K')$ has dimension at most $n-2$, contradicting that $\shard_{i}^{+}(K')$ is supposed to be a shard. This contradiction concludes the induction.

We now prove uniqueness. Once again, we induct on $r$. The base case is obvious: $S_j$ is the only module with dimension $\alpha_j$, so in particular it is the only brick of dimension $\alpha_j$ with stability domain $\alpha_j^{\perp}$.

Let $K = \shard_{i_r}^{\pm_r} \cdots \shard_{i_2}^{\pm_2} \shard_{i_1}^{\pm_1} \left( \alpha_j^{\perp} \right)$ as above. Let $K' = \shard_{i_{r-1}}^{\pm_{r-1}} \cdots \shard_{i_2}^{\pm_2} \shard_{i_1}^{\pm_1} \left( \alpha_j^{\perp} \right)$. By induction, there is a unique brick $B'$ with stability domain $K'$. Let $B$ be a brick with stability domain $K$; we must show that $B \cong \Sigma_{i_r}^{\pm_r}(B')$. Once again, we analyze the case that $\pm_r = +$ and
abbreviate $i_r$ to $i$.

Since $K = \shard_i(K')$, we have $K \subseteq \{ \theta : \langle \theta, \alpha_i \rangle \leq 0 \}$. Because $K$ is $(n-1)$-dimensional, it contains points $\theta$ with $\langle \theta, \alpha_i \rangle < 0$. Thus $B \in \NoSub_i$, so $\Sigma_i^{-1}(B)$ is well defined. Let $B''$ be $\Sigma_i^{-1}(B)$. We know that $B''$ is a real brick, since it is the image of a real brick under $\Sigma_i^{-1}$. Let $K'' = \Stab(B'')$.  We have $B = \Sigma_i(B'')$ and thus  $K = \shard_i(K'')$. Thus, $K''$ must have dimension $n-1$ and must be a shard. There is only one shard $G$ with $\shard_i(G) = K$, namely, the shard $K'$. So $K'' = K'$ and, by induction, $B'' \cong B'$. Then $B \cong \Sigma_i(B')$, as desired.

\section{Shadows of the lattice of torsion classes}
As described in the introduction, part of our motivation for writing this paper is to introduce a complete lattice, extending weak order on $W$, which might be combinatorially more tractable than the lattice of all torsion classes.
To begin with, though, we consider several very general approaches to producing smaller lattices from the lattice of torsion classes, before specializing to the case that motivates us. 

Let $A$ be an algebra, not necessarily finite-dimensional. We are going to work with the category of finite-dimensional $A$-modules on which the Jacobson radical of
$A$ acts nilpotently, for which we will write $\Mod_A$.

Let $\Sc$ be a set of isomorphism classes of indecomposable modules in $A$-mod. We define a \newword{type 1T shadow} to be a subset of $\Sc$ that is of the form $T\cap \Sc$, for a torsion class $T$. We define a \newword{type 1F shadow} to be a subset of $\Sc$ that is of the form $F\cap \Sc$, for a torsion-free class $F$. We refer to $\Sc$ as the \newword{screen set}; in the metaphor which is guiding our nomenclature, we are thinking of $\Sc$ as providing the screen onto which the shadows of torsion or torsion-free classes are being cast. The type 1T shadow was already considered by \cite{GMM}.

We can combine types 1T and 1F by defining a \newword{type 2 shadow} to be an ordered pair $(X,Y)$ such that there is a torsion pair $(T,F)$ with $X=T\cap \Sc$, and $Y=F\cap \Sc$. We order the type 2 shadows by $(X,Y)\leq (X',Y')$ if $X\subseteq X'$ and $Y\supseteq Y'$.

Finally, we define a \newword{type 3 shadow} to be an ordered pair $(X,Y)$ of subsets of $\Sc$ such that $X=\{M\in \Sc \mid \Hom(M,N)=0 \ \forall N\in Y\}$ and
$Y=\{N\in \Sc \mid \Hom(M,N)=0 \ \forall M\in X\}$. Again, we order type 3 shadows by $(X,Y)\leq (X',Y')$ if $X\subseteq X'$ and $Y\supseteq Y'$.

We denote these posets as $\cS^\Sc_{1T}$, $\cS^\Sc_{1F}$, $\cS^\Sc_2$ and $\cS^\Sc_3$. We have maps of posets as shown in Figure~\ref{PosetMaps}.
We now describe the maps and prove that they are maps of posets:

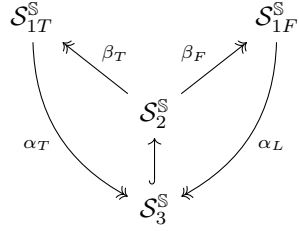
\begin{figure}
 \begin{tikzcd}
\cS^\Sc_{1T} \arrow[bend right=30, two heads]{ddr}[swap]{\alpha_T} && \cS^\Sc_{1F}  \arrow[bend left=30, two heads]{ddl}{\alpha_L}\\
& \cS^\Sc_2 \arrow[two heads]{ul}[swap]{\beta_T}  \arrow[two heads]{ur}{\beta_F} & \\
& \cS^\Sc_3 \arrow[hook]{u} & \\
\end{tikzcd}
%\[ \xymatrix{
%\cS^\Sc_{1T} \ar@/_2.0pc/@{->>}[ddr]_{\alpha_T} && \cS^\Sc_{1F} \ar@/^2.0pc/@{->>}[ddl]^{\alpha_F} \\
%& \cS^\Sc_2 \ar@{->>}[ul]^{\beta_T} \ar@{->>}[ur]^{\beta_F} & \\
%& \cS^\Sc_3 \ar@{(->}[u] & \\
%} \]
\caption{Maps between the posets of shadows} \label{PosetMaps}
\end{figure}

\begin{DP}
  Let $X$ be a type 1T shadow. Let $Y = X^{\perp} \cap \Sc$ and let $\overline{X} = {}^{\perp} \! Y \cap \Sc$.
    Define $\alpha_T(X) = (\overline{X}, Y)$. Then $\alpha_T$ is a map of posets from $\cS^\Sc_{1T}$ to $\cS^\Sc_3$.
\end{DP}

\begin{proof}
  We defined $\overline{X}$ to be $^{\perp} \! Y  \cap \Sc$, so to show that $(\overline{X}, Y) \in \cS_3$, we must simply show that $Y = \overline{X}^{\perp}\! \cap \Sc$. This follows from unwinding definitions. \end{proof}

We define $\alpha_F$ dually, and the analogous result is similar.
We will show later (Corollary~\ref{AlphaSurj}) that $\alpha_T$ and $\alpha_F$  are surjections.

\begin{Proposition}
Every type $3$ shard shadow is also a type $2$ shard shadow. Thus, the identity map $(X,Y) \mapsto (X,Y)$ identifies $\cS^\Sc_3$ with an induced poset of $\cS^\Sc_2$.
\end{Proposition}

\begin{proof}
Let $(X,Y)$ be a type $3$ shard shadow. Define $F = X^{\perp}$ and $T = {}^{\perp}F$, so $(T,F)$ is a torsion pair. We must show that $X = T \cap \Sc$ and $Y = F \cap \Sc$. The definition of a type 3 shard shadow is that $Y = X^{\perp}\! \cap \Sc$, so we have $Y = F \cap \Sc$. 
It remains to check that $X=T \cap \Sc$; we put $X'= T \cap \Sc$. By definition, $X' = {\big(} {}^{\perp}   F  {\big)} \cap \Sc = {\big(} {}^{\perp} (X^{\perp}) {\big)} \cap \Sc$, which clearly contains $X$. But, on the other hand, $F \supset Y$ so ${\big(} {}^{\perp}   F {\big)} \cap \Sc \subseteq {\big(} {}^{\perp}  Y \cap \Sc {\big)} \cap \Sc$  and the right hand side is $X$ by the definition of a Type $3$ shard shadow. So we have shown that $X \subseteq X'$ and $X' \subseteq X$, and we have verified that $X=X'$ as desired.

We have shown that the identity map is an injection from $\cS_3$ to $\cS_2$, and the partial orders on the two sets are defined in the same way, so the proposition is proven.
\end{proof}

\begin{DP}
Let $\beta_T$ and $\beta_F$ be the projection maps sending $(X, Y) \in \cS_2$ to $X$ and to $Y$ respectively. These maps are surjections of posets.
\end{DP}

\begin{proof}
This is immediate from the definitions.
\end{proof}

\begin{prop}
For $(X,Y) \in \cS^\Sc_3$, we have the relations $\alpha_T(\beta_T((X,Y))) = (X,Y)$ and  $\alpha_F(\beta_F((X,Y))) = (X,Y)$.
\end{prop}

\begin{proof}
We check the first claim; the second is parallel. We have $\beta_T((X,Y)) = X$. Let $\alpha_T(X) = (X', Y')$. The definition of $\alpha_T$ gives $Y' = X^{\perp} \cap \Sc$ and the definition of a type 3 shard shadow gives $Y = X^{\perp} \cap \Sc$, so $Y = Y'$. Also, the definition of $\alpha_T$ gives $X' = {}^{\perp}(Y') \cap \Sc$ and the definition of a type 3 shard shadow gives $X = {}^{\perp}(Y) \cap \Sc$. So the equality $Y = Y'$ implies $X = X'$.
\end{proof}

\begin{cor} \label{AlphaSurj}
The maps $\alpha_T$ and $\alpha_F$ are surjections; the restrictions of $\beta_T$ and $\beta_F$ to $\cS^\Sc_3$ (sitting inside $\cS^\Sc_2$) are injections.
\end{cor}

\begin{proof}
This is immediate, since we have just shown that $\alpha_T$ is a left inverse of $\beta_T$, restricted to $\cS^\Sc_3$, and likewise for $\alpha_F$ and $\beta_F$.
\end{proof}

\begin{cor} For a particular choice of $\Sc$, if 
the containment $\cS^\Sc_3 \subseteq \cS^\Sc_2$ is equality, then all the maps in Figure~\ref{PosetMaps} are equalities.
\end{cor}

\begin{proof}
The map $\beta_T$ is surjective by definition, and we have just shown that the restriction of $\beta_T$ to $\cS_3$ is injective. Thus, if $\cS_3 = \cS_2$, then $\beta_T$ is bijective, and its left inverse $\alpha_T$ must also be a bijection. The same applies to $\alpha_F$ and $\beta_F$.
\end{proof}

We pause to consider a simple example.

\begin{example} \label{A2ex} Let $\Lambda$ be the path algebra of an $A_2$ quiver. In Figure \ref{A2}, we have drawn, first, its the poset of five torsion pairs. Each torsion pair is represented by a small schematic copy of the Auslander--Reiten quiver of the $A_2$ module category, with the torsion part drawn in blue, and the torsion-free part drawn in orange.

\begin{figure}
\begin{tikzpicture}
  \node [circle,draw,inner sep=.5mm](a) at (0,0) {\Ta};
  \node [circle,draw,inner sep =.5mm](b) at (-1,1) {\Tb};
  \node [circle,draw,inner sep=.5mm](c) at (-1,2) {\Tc};
  \node [circle,draw,inner sep=.5mm](d) at (0,3) {\Td};
  \node [circle,draw,inner sep=.5mm](e) at (1,1.5) {\Te};
  \draw[dotted] (a)--(b)--(c)--(d);
  \draw [dotted] (a)--(e)--(d);
\end{tikzpicture}
\qquad
\begin{tikzpicture}[halo/.style={line join=round,
 double,line cap=round,double distance=#1,shorten >=-#1/2,shorten <=-#1/2},
 halo/.default=8mm]
  \node [circle,draw,inner sep=.5mm](a) at (0,0) {\Ta};
  \node [circle,draw,inner sep =.5mm](b) at (-1,1) {\Tb};
  \node [circle,draw,inner sep=.5mm](c) at (-1,2) {\Tc};
  \node [circle,draw,inner sep=.5mm](d) at (0,3) {\Td};
  \node [circle,draw,inner sep=.5mm](e) at (1,1.5) {\Te};
  \draw[dotted] (a)--(b)--(c)--(d);
  \draw [dotted] (a)--(e)--(d);
  \begin{scope}[on background layer]
    \draw[halo] (a)--(e);
    \draw[halo] (c)--(d);
    \end{scope}
%  \draw[very thick] (a)--(e); \draw [very thick](c)--(d);
\end{tikzpicture}
\qquad
\begin{tikzpicture}[halo/.style={line join=round,
 double,line cap=round,double distance=#1,shorten >=-#1/2,shorten <=-#1/2},
 halo/.default=8mm]
  \node [circle,draw,inner sep=.5mm](a) at (0,0) {\Ta};
  \node [circle,draw,inner sep =.5mm](b) at (-1,1) {\Tb};
  \node [circle,draw,inner sep=.5mm](c) at (-1,2) {\Tc};
  \node [circle,draw,inner sep=.5mm](d) at (0,3) {\Td};
  \node [circle,draw,inner sep=.5mm](e) at (1,1.5) {\Te};
  \draw[dotted] (a)--(b)--(c)--(d);
  \draw [dotted] (a)--(e)--(d);
  \begin{scope}[on background layer]
    \draw[halo] (c)--(d);
    \end{scope}
 % \draw[very thick](c)--(d);
\end{tikzpicture}
\caption{\label{A2}The lattice of torsion classes of the $A_2$ path algebra, together with, first, the identifications that are made in $\cS^\Sc_{1T}$, and second, the identifications that are made in $\cS^\Sc_{1F}$ and $\cS^\Sc_{2}$}\end{figure}
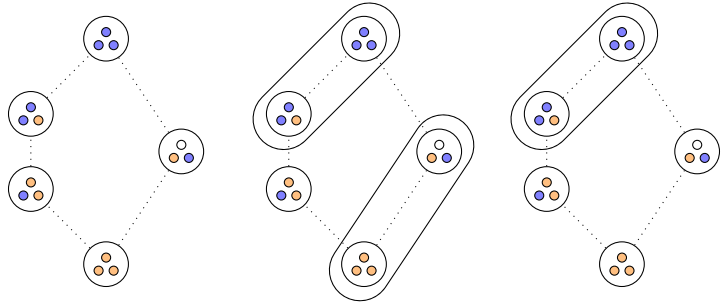

Let $\Sc$ consist of the two indecomposable projectives. (In the Auslander--Reiten quiver, these are represented by the circles at the left and top.)
To the right of the poset of torsion classes, we have redrawn the same figure, with torsion classes which become identified in $\cS^\Sc_{1T}$ encircled. 
The next figure shows which pairs of torsion classes become identified with respect to $\cS^\Sc_{1F}$ and $\cS^\Sc_{2}$.

Of the equivalence classes of $\cS^\Sc_{2}$, there is exactly one which is not in the image of the inclusion from $\cS^\Sc_3$; it is the righthand one. 
\end{example}

Our goal is to study lattices, not posets, and more specifically complete lattices. We recall that a partially ordered set $L$ is a \newword{complete lattice} if, for any subset $\cX$ of $L$, there is a (necessarily unique) element $\bigjoin \cX$ which is a lower bound for all upper bounds of $\cX$ and a (necessarily unique)  element $\bigmeet \cX$ which is an upper bound for all lower bounds of $\cX$.

The collection of torsion classes is a lattice, where $\bigmeet \cX = \bigcap_{T \in \cX} T$ and $\bigjoin \cX = {}^{\perp} \left( \bigcap_{T \in \cX} T^{\perp} \right)$. 
We write $\Tor(\Lambda)$ for the lattice of torsion classes of $\Lambda$.

%It is our hope that all of the versions of shard shadows are the same, that they form a complete lattice, and that this lattice is a complete quotient lattice of $\Tor(\Lambda)$, which will prove more amenable to combinatorial study than $\Tor(\Lambda)$.
%Here is what we can prove in this regard:

\begin{prop} \label{ShardShadowLatticeProps}
The partially ordered sets $\cS^\Sc_{1T}$, $\cS^\Sc_{1F}$ and $\cS^\Sc_3$ are complete lattices. The map $\pi_T(X) = X \cap \Sc$ from $\Tor(\Lambda)$ to $\cS^\Sc_{1T}$ obeys $\pi_T\left( \bigmeet \cT \right) = \bigmeet \pi_T(\cT)$ and the map $\pi_F(X) = X^{\perp} \cap \Sc$ from $\Tor(\Lambda)$ to $\cS_{1F}^\Sc$ obeys $\pi_F\left( \bigjoin \cT \right) = \bigjoin \pi_F(\cT)$.
\end{prop}
%\margin[Hugh]{I didn't like using $\pi$ for two different maps, so I added subscripts, here and in the proof. Subscripting them $T$ and $F$ meant that we shouldn't use $T$ as the argument, so I changed it to $X$. If we decide to use $\mathcal T$ and $\mathcal F$ instead of $T$ and $F$, we can change $X$ back to $T$.}

\begin{proof}
The fact that $\cS^\Sc_3$ is a complete lattice is straightforward; if $\Sha$ is any set and $\to$ any binary relation on $\Sc$, we define $\Pairs(\to)$ to be set of ordered pairs $(U,V)$ where $U = \{ u \in \Sc : u \not\to v \ \forall v \in V \}$ and $V = \{ v \in \Sc : u \not\to v \ \forall u \in U \}$. For $(U_1, V_1)$ and $(U_2, V_2) \in \Pairs(\to)$, we have $U_1 \subseteq U_2$ if and only if $V_1 \supseteq V_2$ and we define a partial order on $\Pairs(\to)$ by saying that $(U_1, V_1) \leq (U_2, V_2)$ in this situation. This will always be a complete lattice, where $\bigjoin \cX = (A, A^{\perp})$ for $A = \bigcap_{(U,V) \in \cX} U$ and $\bigmeet \cX = ({}^{\perp} B, B)$ where $B = \bigcap_{(U,V) \in \cX} V$. 
See, for example, \cite[Theorem V.19]{Birkhoff}.

We now show that $\cS^\Sc_{1T}$ is a complete lattice. Let $\cX \subseteq \cS_{1T}$. Set $A = \bigcap_{X \in \cX} X$. We claim that $A \in \cS^\Sc_{1T}$; if this is true, then $A$ is clearly the greatest lower bound for $\cX$. To show that $A \in \cS^\Sc_{1T}$, for each $X \in \cX$, choose a torsion class $T_X$ with $X = \Sc \cap T_X$. Then $A = \bigcap_{X \in \cX} X = \bigcap_{X \in \cX} ( \Sc \cap T_X ) = \Sc \cap \bigcap_{X \in \cX} T_X$. But $\bigcap_{X \in \cX} T_X$ is a torsion class as well, so the equation $A =  \Sc \cap \bigcap_{X \in \cX} T_X$ shows that $A \in \cS^\Sc_1$. We have now shown that all subsets of $\cS^\Sc_{1T}$ have greatest lower bounds. We then use a standard trick: For $\cX \subseteq \cS^\Sc_{1T}$, let $\cY$ be the set of all upper bounds of $\cX$. The greatest lower bound of $\cY$ will also be the least upper bound of $\cX$.

The argument of the above paragraph also shows that $\pi_T\left( \bigmeet \cT \right) = \bigmeet \pi_T(\cT)$: Let $\cT \subseteq \Tor(\Lambda)$ and let $\cX = \pi_T(\cT) := \{ T \cap \Sc : T \in \cT \}$. Then $\bigmeet \pi_T(\cT) = \bigcap_{X \in \pi_T(\cT)} X = \bigcap_{T \in \cT} ( \Sc \cap T) = \Sc \cap \bigcap_{T \in \cT} T = \Sc \cap \bigmeet \cT = \pi_T \left( \bigmeet \cT \right)$.

This proves the claims regarding $\cS^\Sc_{1T}$; the claims regarding $\cS^\Sc_{1F}$ are similar.
\end{proof}

We cannot prove that $\cS^\Sc_2$ is necessarily a lattice.

A map $\pi : L_1 \to L_2$ from one complete lattice to another is a \newword{complete lattice quotient} if $\pi$ is surjective and we have $\pi\left( \bigmeet X \right) = \bigmeet \pi(X)$ and $\pi\left( \bigjoin X \right) = \bigjoin \pi(X)$ for all subsets $X$ of $L_1$.
In the generality which we are considering, the maps from the lattice of torsion classes to $\cS_{1T}$, $\cS_{1F}$ or $\cS_3$ need not be complete lattice quotients, as Example \ref{A2ex} already shows.

We now explain the connection to the problem that has been motivating us. For $\Lambda$ a preprojective algebra, we would like to let $\Sc$ be the shard modules. In Dynkin type, the shard modules coincide with the bricks, and thus all the shadows with respect to $\Sc$ coincide with the lattice of torsion classes, and thus with (right) weak order on the corresponding Coxeter group $W$. If our preprojective algebra is not of Dynkin type, our hope is that $\cS^\Sc_{1T}$,
$\cS^\Sc_{1F}$, $\cS^\Sc_{2}$, and $\cS^\Sc_{3}$ still all coincide, and are a complete lattice quotient of the lattice of torsion classes.

\begin{remark}
In the generality considered throughout this section, $\Tor(\Lambda)$ is a completely semidistributive lattice, so any lattice which we can prove to be a complete lattice quotient of $\Tor(\Lambda)$ will be a completely semidistributive lattice \cite[Theorem 3.1(a)]{DIRRT}.
\end{remark}

\begin{remark} \label{rem:Biclosed}
Matthew Dyer~\cite{DyerConjectures} has made a series of fascinating conjectures about extending the weak order on $W$ to a partial order on the set of ``biclosed" sets of roots.
These can be thought of as a partition of the positive real roots into two disjoint sets, $\Phi^+ = X \sqcup Y$, where both $X$ and $Y$ satisfy a ``closure" condition. 
If we were to try to associate a biclosed set to a torsion pair $(T,F)$, the natural thing to do would be to consider $(\dim T, \dim F) := (\{ \dim B : B \in T \}, \{ \dim B : B \in F \})$, where $B$ ranges over shard modules (or possibly real brick modules). This construction will associate a pair of subsets of $\Phi^+$ to any type $2$ or type $3$ shard shadow.

There is, however, an immediate problem: It is not clear that $\Phi^+ = (\dim T) \sqcup (\dim F)$. Indeed, it is neither clear that $(\dim T) \cup (\dim F) = \Phi^+$, nor that $(\dim T) \cap (\dim F) = \emptyset$. 
We have $(\dim T) \cap (\dim F) = \emptyset$ for all torsion pairs $(T,F)$ if and only if the following conjecture holds:
\begin{conj}
Let $\beta$ be a positive real root and let $B_1$ and $B_2$ be two shard modules of dimension $\beta$. Then $\Hom(B_1, B_2) \neq 0$. 
\end{conj} 
It is tempting to make some more general conjecture about a condition on $\Stab(B_1)$ and $\Stab(B_2)$ which would imply $\Hom(B_1, B_2) \neq 0 $. % for bricks of different dimensions.
We warn the tempted reader to consult Remark~\ref{WhyBadBrickIsBad} for an example of a brick $B'$ such that $\Hom(S_1, B')=0$, even though $\langle \alpha_1,- \rangle$ is nonnegative on $\Stab(B')$.
\end{remark}

\section{Counterexamples}

In this section, we collect some counter-examples to natural conjectures.

\subsection{Shards can depend on the choice of Cartan matrix}  \label{ShardsSeeCartan}
Often, constructions which are described in terms of the reflection arrangement of a Coxeter group can be shown to depend only on the Coxeter group, and not on the choice of Cartan matrix. 
To be more precise, fix a Coxeter group $W$ with simple generators $s_i$. Let $A$ and $A'$ be two Cartan matrices for $W$. This means that, for $m_{ij}<\infty$, we have $A_{ij} A_{ji} = A'_{ij} A'_{ji} = 4 \cos^2 \tfrac{\pi}{m_{ij}}$ but, for $m_{ij} = \infty$, all we know is that $A_{ij} A_{ji}$ and $A'_{ij} A'_{ji}$ are both $\geq 4$. It is often the case that constructions which are originally defined in terms of the Cartan matrix, and the corresponding reflection representation and root system, in fact depend only on the Coxeter group $W$.  See Lemma~\ref{lem:CombinatorialPoset} and Remarks~\ref{CombinatorialRankTwo} and~\ref{CombinatorialCuttingSet}.

In this section, we will show that is not the case for shards. To be precise, let $W$ be a Coxeter group and let $t$ be a reflection in $W$. 
Let $\{ R_1, R_2, \ldots, R_k \}$ be the list of rank two subsystems cutting $\beta_t^{\perp}$. This list depends on $W$ and $t$ and not on the choice of Cartan matrix (Remark~\ref{CombinatorialCuttingSet}).
However, we will show that the combinatorics of this hyperplane arrangement \textbf{does} depend on the choice of Cartan matrix.

Our example relies on the following trick:
Let $W$ be a Coxeter group with simple generators $s_1$, $s_2$, \dots, $s_n$ and a symmetric Cartan matrix $A$. 
Embed $W$ into a larger Coxeter group $\hW$  with one additional simple generator $s_{n+1}$, such that $s_i s_{n+1}$ has infinite order  for $1 \leq i \leq n$. Embed $A$ into a larger Cartan matrix $\hA$ with $\hA_{ij} = A_{ij}$ for $1 \leq i,j \leq n$ and $\hA_{i(n+1)} = \hA_{(n+1)i}  \geq 2$. 
We write $V$ for the reflection representation of $W$ and $\hV$ for the reflection representation of $\hW$, of dimensions $n$ and $n+1$ respectively.

Let $u$ be any element of $W$ and let $s_{i_r} s_{i_{r-1}} \cdots s_{i_2} s_{i_1}$ be a reduced word for $u$.
Put $\beta = u \alpha_{n+1}$.
\begin{prop} \label{InvToShard}
With the above notation, $s_{i_r} s_{i_{r-1}} \cdots s_{i_2} s_{i_1} \alpha_{n+1}$ is a positive expression for $\beta$. The shard hyperplane arrangement, in the $n$-dimensional space $\beta^{\perp}$, is linearly isomorphic to the hyperplane arrangement $\bigcup_{\delta \in \inv_{\Phi}(u)} \delta^{\perp}$ in the $n$-dimensional space $V^{\ast}$.
\end{prop}

\begin{proof}
We begin by noting that $s_{i_r} s_{i_{r-1}} \cdots s_{i_2} s_{i_1} \alpha_{n+1}$ is a positive expression for $\beta$. 
By Corollary~\ref{EverythingAboutPositiveExpressions}, it is equivalent to see that $s_{i_r} s_{i_{r-1}} \cdots s_{i_2} s_{i_1} s_{n+1} s_{i_1} s_{i_2} \cdots s_{i_{r-1}} s_{i_r}$ is reduced. 
Since $s_i s_{n+1}$ has infinite order  for $1 \leq i \leq n$, no braid moves can be applied to this word that will affect the $s_{n+1}$ in the center, and this will remain true after any braid moves are applied to the left and right of the central $s_{n+1}$. 
Since $s_{i_r} s_{i_{r-1}} \cdots s_{i_2} s_{i_1}$ is assumed reduced, this means no sequence of braid moves can put us in a position to reduce the length of $s_{i_r} s_{i_{r-1}} \cdots s_{i_2} s_{i_1} s_{n+1} s_{i_1} s_{i_2} \cdots s_{i_{r-1}} s_{i_r}$.

Thus, by Corollary~\ref{EverythingAboutPositiveExpressions}, the hyperplanes in the shard hyperplane arrangement are $\{ \delta^{\perp} \cap \beta^{\perp} : \delta \in \inv_{\Phi}(u) \}$. 
We now need to relate this hyperplane arrangement, in $\beta^{\perp}$, to the hyperplane arrangement $\{ \delta^{\perp} : \delta \in \inv_{\Phi}(u) \}$ in $V^{\ast}$. 
The bridge between these two is the arrangement $\{ \delta^{\perp} : \delta \in \inv_{\Phi}(u) \}$ in $\hV^{\ast}$. The canonical inclusion
of $V$ into $\hV$ gives rise to a projection from $\hV^\ast$ to $V^\ast$. For $\delta$ a root of $W$, $\delta^\perp$ in $\hV^\ast$ is the pullback of $\delta^\perp$ in $V^\ast$. The kernel of the projection from $\hV^\ast$ to $V^\ast$ is the line
$V^\perp$. Since $\beta\not\in V$, $V^\perp\cap \beta^\perp=\{0\}$, so intersecting with $\beta^\perp$ removes the one-dimensional lineality factor.
%Since all of the elements of $\inv_{\Phi}(u)$ are in $V$, the dual hyperplanes in $\hV^{\ast}$ can be described as the product of the hyperplanes in $V^{\ast}$ with the line $V^{\perp}$ in $\hV^{\ast}$. 
%Then slicing with $\beta^{\perp}$ removes this one dimensional lineality factor, giving as a linearly isomorphic hyperplane arrangement.
\end{proof}

The use of Proposition~\ref{InvToShard} is that we can use computations with inversions in a rank $n$ Coxeter group to produce examples of shards in a rank $n+1$ Coxeter group. 

We now turn to our example:
Consider the rank $4$ Coxeter group where $m_{12}=m_{13} = m_{14}=0$ and $m_{14} = m_{24} = m_{34} = \infty$. Symmetric Cartan matrices for this group are thus of the form
\[ \begin{sbm} 2 & 0 & 0 & - x \\ 0&2&0 & - y \\ 0&0&2 & - z \\ -x & -y & - z & 2 \\ \end{sbm}  \]
for $x$, $y$, $z \geq 2$.
Let $w = s_1 s_2 s_4 s_1 s_3 s_4 s_2 s_3$, whose inversions correspond to the following roots in the simple root basis: 
\[ \gamma_1= \begin{sbm} 1 \\ 0 \\ 0 \\ 0\end{sbm} ,\   \gamma_2=\begin{sbm} 0 \\ 1 \\ 0 \\ 0\end{sbm} ,\   \gamma_3=\begin{sbm} x \\ y \\ 0 \\ 1\end{sbm} ,\  \gamma_4= \begin{sbm} -1 + x^2 \\ x y \\ 0 \\ x\end{sbm} ,\  \gamma_5= \begin{sbm} x z \\  y z \\ 1 \\ z\end{sbm} ,\]
  \[   \gamma_6=\begin{sbm} -2 x + x^3 + x z^2 \\ -y + x^2 y + y z^2 \\ z \\ -1 + x^2 + 
  z^2\end{sbm} ,\   \gamma_7=\begin{sbm} -x y + x^3 y + x y z^2 \\ -1 + x^2 y^2 + y^2 z^2 \\ y z \\ 
 x^2 y + y z^2\end{sbm} ,\  \gamma_8= \begin{sbm} -3 x z + x^3 z + x z^3 \\ -2 y z + x^2 y z + 
  y z^3 \\ -1 + z^2 \\ -2 z + x^2 z + z^3 \end{sbm} \]
The determinant $\det(\gamma_1, \gamma_2, \gamma_7, \gamma_8)$ is $y (x+z)(x-z)$.
  Thus, if $x=z$, the hyperplanes $\gamma_1^{\perp}$, $\gamma_2^{\perp}$, $\gamma_7^{\perp}$ and $\gamma_8^{\perp}$ intersect in a line, whereas they are transverse if $x \neq z$.

Moreover, even if we restrict to $x \neq z$, the hyperplane arrangements for $x<z$ and $x>z$ will be topologically different. The four lines $\gamma_1^{\perp} \cap \gamma_2^{\perp} \cap \gamma_3^{\perp}$, $\gamma_1^{\perp} \cap \gamma_2^{\perp} \cap \gamma_5^{\perp}$, $\gamma_1^{\perp} \cap \gamma_2^{\perp} \cap \gamma_7^{\perp}$ and $\gamma_1^{\perp} \cap \gamma_2^{\perp} \cap \gamma_8^{\perp}$ all lie in the plane $\gamma_1^{\perp} \cap \gamma_2^{\perp}$. Their circular order depends on whether the cross ratio
\[ \frac{ \det(\gamma_1, \gamma_2, \gamma_3, \gamma_5) \det(\gamma_1, \gamma_2, \gamma_7, \gamma_8) }{ \det(\gamma_1, \gamma_2, \gamma_3, \gamma_8) \det(\gamma_1, \gamma_2, \gamma_5, \gamma_7) } \]
lies in $(-\infty, 0)$, $(0,1)$ or $(1, \infty)$. We compute that this cross ratio is $\tfrac{y (x+z)(x-z)}{(z^2-1) x^2 y^2}$, whose sign depends on whether $x$ is $<z$ or $>z$.

This example is inspired by a simpler example due to Dyer and Wang~\cite{DyerWang}. 
They produce a simpler example of a rank four Coxeter group in which the hyperplane arrangement determined by a finite subset of the roots depends on the entries in the Cartan matrix. 
We have had to adjust their example to make sure that all of the roots involved were inversions of a single element in the Coxeter group.

\subsection{A real brick that is not a shard module} \label{RealBrickNotShard}

In this section, we will give an example of a rank $6$ preprojective algebra (of wild type) and a real brick $B$ for this preprojective algebra whose stability domain has dimension $2$.
We use a symmetric Cartan matrix, so all our vector spaces are over the same field.

Our preprojective algebra corresponds to the quiver
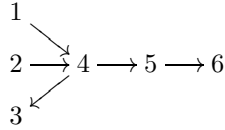
\[ \xymatrix@R=0.1 in@C=0.2 in{
1 \ar@{->}[dr] &&& \\
2 \ar@{->}[r] & 4 \ar@{->}[r] & 5 \ar@{->}[r] & 6 \\
3 \ar@{<-}[ur] &&& \\
} \]
The purpose of the arrows in this diagram is to convey the $\sgn$ function which is used to define the preprojective relation, with $i \overset{a}{\leftarrow} j$ meaning that $\sgn(a)=1$ and $\sgn(a^{\ast}) = -1$.
Thus, if $i \overset{a}{\rightarrow} j \overset{b}{\rightarrow} k$, then $a a^{\ast}$ and $b^{\ast} b$ appear with opposite signs in the preprojective relation. As a convenience to the reader, we have chosen an orientation such that this is the case whenever two nonzero terms appear in a preprojective relation for our module.

Our real brick $B$ has dimension vector $(3,3,2,4,2,1)$. 
We write $B_i$ for the component of $B$ on vertex $i$.

\begin{figure}[h]
\[ \xymatrix@R=0.15 in@C=0.15 in{
&&&4 \ar@{->}[dr]  \ar@{->}[dl] \ar@{->}[d] & \\
2 \ar@{->}[dr] &&1&2 \ar@{->}[d]  &3  \ar@{->}[dl]  \\
&4 \ar@{->}[d] \ar@{->}[dr] &&4 \ar@{->}[dr]  \ar@{->}[d] & \\
2 \ar@{->}[drrr] &1 \ar@{->}[drr]  &3 \ar@{->}[dr]  & 1 \ar@{->}[d] &5\ar@{->}[d]  \ar@{->}[dl] \\
&&&4 \ar@{->}[dr]  &6 \ar@{->}[d] \\
&&&&5 \\
} \]
\caption{A real brick which is not a shard module} \label{BadBrick}
\end{figure}
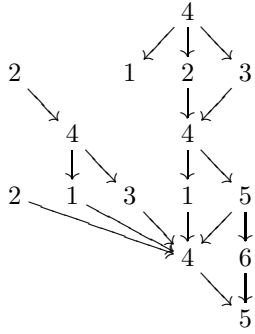

We depict the structure of the module in Figure~\ref{BadBrick}.
Here is how to understand the figure: $B_i$ is the free vector space on the vertices labeled $i$ in the figure. 
For each arrow $i \to j$ in the doubled quiver, the map $B_i \to B_j$ is given by the sum over arrows in the figure whose source is labeled $i$ and whose target is labeled $j$; each such arrow maps the corresponding basis element of $B_i$ to the corresponding basis element of $B_j$.
% weighted by $-1$ if that arrow is so labeled.\margin{HT Rephrased, since we're actually labelling the arrows ``-1'', which I think is a good choice. An alternative would be to colour the negative arrows, which avoids the problem of which arrow a $-1$ belongs to. DES Good catch. To me, coloring $-1$ edges would be hard to read, so I prefer to label them. We can play with the lay out to see if we can place the $-1$'s better.}
We first check that $B$ is a representation of the preprojective algebra. It is sufficient to start at each basis element, follow all the paths of length two from it which lead back to a basis element at the same vertex, and see that they sum to zero. The only length two paths that lead to from a vertex to itself at all are along the sides of the four quadrilateral faces in the figure. For example, between the first and third row, we have a path $4 \to 2 \to 4$ and a path $4 \to 3 \to 4$. Because this portion of the quiver is oriented $2 \to 4 \to 3$, these paths cancel. The computation is analogous for the other three quadrilateral faces. 

We next check that $\dim B$ is a real root. Indeed, it can be obtained by applying reflections to a simple root:
\[
\dim B = s_1s_2s_4s_3s_5s_4s_1s_2s_4s_5(\alpha_6)
\]
%Indeed, 
%\begin{multline*}
%(3 \alpha_1 + 3 \alpha_2 + 2\alpha_3 +4\alpha_4 + 2\alpha_5 +\alpha_6) \cdot (3 \alpha_1 + 3 \alpha_2 + 2\alpha_3 +4\alpha_4 + 2\alpha_5 +\alpha_6) = \\
%2(3^2+3^2+2^2+4^2+2^2+1^2) - 2 (3 \cdot 4 + 3 \cdot 4 + 2 \cdot 4 + 4 \cdot 2 + 2 \cdot 1)  = 2 . 
%\end{multline*}

We now provide a direct verification that $B$ is a brick.
Let $\phi : B \to B$ be an endomorphism. Let $x \in B_4$ be the basis element corresponding to the $4$ in the top row of the diagram, and let $y$ and $z$ be the basis elements of $B_2$ corresponding to the $2$'s at the left ends of the second and fourth rows of the diagram. We will represent elements of the path algebra by recording paths through the doubled quiver: For example, $4 \leftarrow 1 \leftarrow 4$ is the path which goes from $4$ to $1$ to $4$. 

Note that $(5 \leftarrow 4)(x)=0$ and $(4 \leftarrow 1 \leftarrow 4)(x)=0$. 
Therefore, $(5 \leftarrow 4)(\phi(x))=0$ and $(4 \leftarrow 1 \leftarrow 4)(\phi(x))=0$. This implies that $\phi(x)$ is a scalar multiple of $x$, say $\phi(x) = \lambda x$, since the intersection of the kernels of $5 \leftarrow 4$ and $4 \leftarrow 1 \leftarrow 4$ with $B_4$ is spanned by $x$.
Similarly, $(5 \leftarrow 4 \leftarrow 2)(y)=0$, implying that $(5 \leftarrow 4 \leftarrow 2)(\phi(y))=0$, so $\phi(y) = \mu y$ for some scalar $\mu$.
Also, $(1 \leftarrow 4 \leftarrow 2)(z) = 0$ implying that $(1 \leftarrow 4 \leftarrow 2)(\phi(z)) = 0$, so $\phi(z) = \nu z$ for some scalar $\nu$.

Let $w \in B_5$ be the vector corresponding to the bottom $5$. We have $w=(5 \leftarrow 4 \leftarrow 2)(z)$, so $\phi(w) = (5 \leftarrow 4 \leftarrow 2)(\phi(z))  = \nu (5 \leftarrow 4 \leftarrow 2)(z) = \nu w$. Similarly, $w= (5 \leftarrow 6 \leftarrow 5 \leftarrow 4 \leftarrow 3 \leftarrow 4)(x)$ and $w= (5 \leftarrow 4 \leftarrow 3 \leftarrow 4 \leftarrow 2)(y)$, so $\phi(w) = \lambda w$ and $\phi(w) = \mu w$. This shows that $\lambda=\mu=\nu$. Since $x$, $y$ and $z$ generate $B$, the entire endomorphism $\phi$ must be multiplication by the scalar $\lambda$. We have shown that every $\phi \in \End(B)$ is a scalar, as desired.

We now verify that $\dim \Stab(B) \leq 2$. To this end, we will need to know that $B$ has submodules with the dimension vectors $(1,2,1,2,1,1)$, $(1,2,1,2,1,0)$, $(3,1,2,4,2,1)$ and $(1,0,0,0,0,0)$. These submodules correspond to the boxed elements in Figure~\ref{BadBrickSubs}.
\begin{figure}[ht]
\[ \xymatrix@R=0.1 in@C=0.1 in{
&&&4 \ar@{->}[dr]  \ar@{->}[dl] \ar@{->}[d] & \\
\fbox{2} \ar@{->}[dr] &&1&2 \ar@{->}[d]  &3  \ar@{->}[dl]  \\
&\fbox{4} \ar@{->}[d] \ar@{->}[dr] &&4 \ar@{->}[dr]  \ar@{->}[d] & \\
\fbox{2} \ar@{->}[drrr] &\fbox{1} \ar@{->}[drr]  &\fbox{3} \ar@{->}[dr]  & 1 \ar@{->}[d] &5\ar@{->}[d]  \ar@{->}[dl] \\
&&&\fbox{4} \ar@{->}[dr]  & \fbox{6} \ar@{->}[d] \\
&&&&\fbox{5} \\
} \qquad
 \xymatrix@R=0.1 in@C=0.1 in{
&&&4 \ar@{->}[dr]  \ar@{->}[dl] \ar@{->}[d] & \\
\fbox{2} \ar@{->}[dr] &&1&2 \ar@{->}[d]  &3  \ar@{->}[dl]  \\
&\fbox{4} \ar@{->}[d] \ar@{->}[dr] &&4 \ar@{->}[dr]  \ar@{->}[d] & \\
\fbox{2} \ar@{->}[drrr] &\fbox{1} \ar@{->}[drr]  &\fbox{3} \ar@{->}[dr]  & 1 \ar@{->}[d] &5\ar@{->}[d]  \ar@{->}[dl] \\
&&&\fbox{4} \ar@{->}[dr]  &6 \ar@{->}[d] \\
&&&&\fbox{5} \\
}  \]
\[ \xymatrix@R=0.1 in@C=0.1 in{
&&&\fbox{4} \ar@{->}[dr]  \ar@{->}[dl] \ar@{->}[d] & \\
2 \ar@{->}[dr] && \fbox{1}& \fbox{2} \ar@{->}[d]  & \fbox{3}  \ar@{->}[dl]  \\
&\fbox{4} \ar@{->}[d] \ar@{->}[dr] && \fbox{4} \ar@{->}[dr]  \ar@{->}[d] & \\
2 \ar@{->}[drrr] &\fbox{1} \ar@{->}[drr]  &\fbox{3} \ar@{->}[dr]  & \fbox{1} \ar@{->}[d] &\fbox{5}\ar@{->}[d]  \ar@{->}[dl] \\
&&&\fbox{4} \ar@{->}[dr]  & \fbox{6} \ar@{->}[d] \\
&&&&\fbox{5} \\
} \qquad 
 \xymatrix@R=0.2 in@C=0.2 in{
&&&4 \ar@{->}[dr]  \ar@{->}[dl] \ar@{->}[d] & \\
2 \ar@{->}[dr] &&\fbox{1}&2 \ar@{->}[d]  &3  \ar@{->}[dl]  \\
&4 \ar@{->}[d] \ar@{->}[dr] &&4 \ar@{->}[dr]  \ar@{->}[d] & \\
2 \ar@{->}[drrr] &1 \ar@{->}[drr]  &3 \ar@{->}[dr]  & 1 \ar@{->}[d] &5\ar@{->}[d]  \ar@{->}[dl] \\
&&&4 \ar@{->}[dr]  &6 \ar@{->}[d] \\
&&&&5 \\
} 
\]
\caption{Some key submodules of our counterexample} \label{BadBrickSubs}
\end{figure}
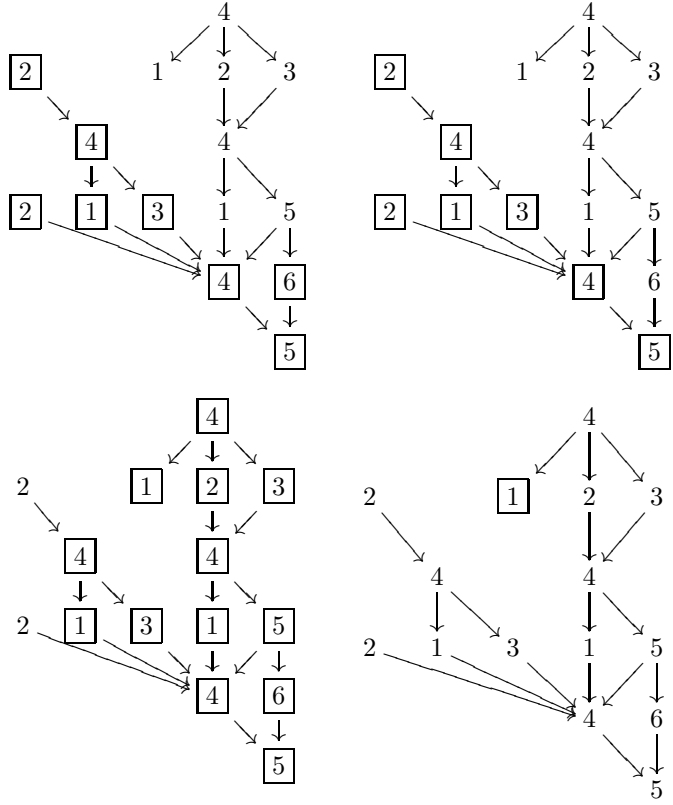
So any $\theta$ in $\Stab(B)$ must pair nonnegatively with the vectors $(1,2,1,2,1,1)$, $(1,2,1,2,1,0)$, $(3,1,2,4,2,1)$ and $(1,0,0,0,0,0)$, and also must pair to $0$ with $\dim B =  (3,3,2,4,2,1)$. 
But we have
\[ 2\cdot(1,2,1,2,1,1) + 2\cdot(1,2,1,2,1,0) +  (3,1,2,4,2,1) + 2\cdot(1,0,0,0,0,0) = 3 \cdot (3,3,2,4,2,1) . \]
So $\theta$ must pair to $0$ with all of these vectors, which restricts $\theta$ to a codimension $4$ linear space!
We have shown that $B$ is not a shard module. This completes our verification that
\begin{theorem}
The module $B$ discussed above is a real brick, but is not a shard module.
\end{theorem}

By Theorem~\ref{RealBrickRecursion}, we know that $B$ should be obtainable by a series of reflection functors applied to a simple module. 
We give the corresponding formula without proof:
\[ B = \Sigma_1^{-} \Sigma_2^{+} \Sigma_4^{+} \Sigma_3^{+} \Sigma_5^{-} \Sigma_4^{-} \Sigma_1^{-} \Sigma_2^{+} \Sigma_4^{+} \Sigma_5^{+} S_6 . \]
The module $B$ was found using a SAGE notebook~\cite{DanaSAGE} that iteratively searched through formulas like the above one; the authors then worked backwards to find the computations given here, so that readers could trust our result without relying on our code.
Using this formula and Theorem~\ref{StabOfAShardModule}, we can compute the stability domain of $B$ exactly -- it is a $2$-dimensional cone with rays $(0,0,-1,0,1,0)$ and $(0,0,1,-1,1,0)$.

\begin{remark} \label{WhyBadBrickIsBad}
It is interesting to consider the module $B' := \Sigma_1^{+} B$,  which is the result of the above formula right before we apply $\Sigma_1^{-}$. 
This module is a shard module, and it lies in $\NoSub_1$ (which is why we can apply $\Sigma_1^+$ to it).
However, its stability domain lies entirely in the closed half space where $\langle \alpha_1, - \rangle$ is nonnegative.
This is why $\shard_1^+(\Stab(B'))$ has lower dimension than $\Stab(B')$ does.
\end{remark}

\appendix

%%%%%%%%%% NOTE: need weird technicality re: picking field basis properly. More generally, should probably reiterate some of the setup.

\section{Proof of Theorem \ref{CB Identity} in the symmetrizable case}
\label{appendix}

The goal of this section is to outline a proof of the following identity:
{\renewcommand{\thethm}{\ref{CB Identity}}
\begin{theorem} %\label{CB Identity}
Let $M$ and $N$ be left $\Lambda$-modules, which are finite dimensional over our ground field $\kappa$. Then
%\label{crawley-boevey appendix}
\[
\dim_\kappa(\Hom_\Lambda(M, N)) - \dim_\kappa(\Ext^1_\Lambda(M, N)) + \dim_\kappa(\Hom_\Lambda(N, M)) = (\dim M, \dim N)
\]
\end{theorem}
\addtocounter{thm}{-1}}
%\margin[Hugh]{We frequently state a theorem once and restate it when we are going to prove it. It's not difficult to make Latex give both theorems the same number. Do we want to do that?}

This identity was originally observed for the case of symmetric Cartan matrices in \cite{Crawley-Boevey}. 
It is difficult to find a reference for the symmetrizable case, which is why we have included this appendix.
Morally, this theorem should be thought of as saying that $(\dim M, \dim N) = \sum_j (-1)^j \dim_{\kappa} \Ext^j_{\Lambda}(M,N)$, combined with a Serre duality property that $\dim_{\kappa} \Ext^j_{\Lambda}(M,N) = \dim_{\kappa} \Ext^{2-j}_\Lambda(N,M)$.
However, this is not literally correct in Dynkin type. In fact, Dynkin type
preprojective algebras are self-injective \cite[Corollary 3.4]{IO}, and
self-injective algebras are necessarily of infinite global dimension \cite[Section IV.3]{ARS}.
%\margin{I want a citation for the fact that $\Lambda$ has infinite projective dimension in Dynkin type. 
%Hugh writes ``
%They are self-injective and therefore of infinite global dimension (one can cite Auslander-Reiten-Smalo for this implication). The history of the fact that they are self-injective is explained in arXiv:1906.11817; the short version is that the first non-crappy proof seems to be by Oppermann and Iyama in 2013."
%I skimmed Oppermann-Iyama and it wasn't obvious what to cite, so leaving this note for now.}

A proof of Theorem \ref{CB Identity} for the symmetric  case appears in \cite[Section 8]{GLS}: here, we explain how to generalize that proof to the framework established in Section \ref{preproj ssec gen} based on \cite{Kulshammer}. Recall in particular that $\{b^{ji}_k\}$ is a $\kappa(d_j)$-basis of $\kappa(d_{ij})$ and $\{(b^{ji}_k)^*\}$ is its dual basis under the trace pairing $\kappa(d_{ij})\times \kappa(d_{ij})\to \kappa(d_j)$. Because $\kappa(d_{ij})$ is the compositum of $\kappa(d_j)$ and $\kappa(d_i)$, we can choose the $b^{ji}_k$ to lie in $\kappa(d_i)$, which will simplify calculations later on.

The strategy of the proof is to construct, for any $\Lambda$-module $M$, the first three terms of a projective resolution
\[
P_2(M)\xrightarrow{d^1} P_1(M) \xrightarrow{d^0} P_0(M) \xrightarrow{\operatorname{mult}} M\to 0
\]
Then, by applying $\Hom_\Lambda(-, N)$, we obtain a complex whose first two homologies compute $\Hom_\Lambda(M, N)$ and $\Ext^1_\Lambda(M, N)$:
\[0\rightarrow\Hom_\Lambda(P_0(M), N) \xrightarrow{d^0_{M, N}} \Hom_\Lambda(P_1(M), N) \xrightarrow{d^1_{M, N}}  \Hom_\Lambda(P_2(M), N) \rightarrow 0\]
The crux of the argument is to show that this complex is isomorphic to the dual (over the ground field $\kappa$) of the corresponding complex with $M$ and $N$ switched:
\[
0\rightarrow \Hom_\Lambda(P_2(N), M)^*\xrightarrow{d^{1*}_{N, M}} \Hom_\Lambda(P_1(N), M)^* \xrightarrow{d^{0*}_{N, M}} \Hom_\Lambda(P_0(N), M)^* \rightarrow 0
\]
This will show that the third homology of our complex computes $\Hom_\Lambda(N, M)^*$. From this we can conclude that
\begin{multline*}
\dim_\kappa(\Hom_\Lambda(M, N)) - \dim_\kappa(\Ext^1_\Lambda(M, N)) + \dim_\kappa(\Hom_\Lambda(N, M)) =\\
\dim_\kappa(\Hom_\Lambda(P_0(M), N)) - \dim_\kappa(\Hom_\Lambda(P_1(M), N)) + \dim_\kappa(\Hom_\Lambda(P_2(M), N))
\end{multline*}
which we then show is equal to $(\dim M, \dim N)$.

We start with a recipe for building projective resolutions in $\Mod(\Lambda)$.

\begin{lemma}[\cite{Kulshammer}, Lemma 5.1]
\label{bimoduleresolution}
The following is the start of a projective $\Lambda$-$\Lambda$-bimodule resolution of $\Lambda$:
\[
\bigoplus_{1\leq i\leq n}V_i \xrightarrow{d^1} \bigoplus_{1\leq i\neq j\leq n} W_{j\ot i} \xrightarrow{d^0} \bigoplus_{1\leq i \leq n} V_i\xrightarrow{\operatorname{mult}} \Lambda\to 0\]
where % is there, in fact, a typo in the definition of d^1?
\[
V_i = \Lambda e_i \otimes_{\kappa(d_i)} e_i \Lambda \text{ and }
W_{j\ot i} = \Lambda e_{j} \otimes_{\kappa(d_j)} E(j\ot i) \otimes_{\kappa(d_i)} e_{i}\Lambda 
\]
and the maps are given by
\begin{align*}
\operatorname{mult}(e_i\otimes e_i) &= e_i \\
d^0(e_j\otimes x \otimes e_i) &= x\otimes e_{i} - e_{j}\otimes x \\
d^1(e_i\otimes e_i) &= \sum_{j\neq i} \sgn(i, j)\sum_{i \overset{a}{\longrightarrow} j}\sum_{k} {\big(} (b^{ji}_k)^\ast a^{\ast} \otimes a b^{ji}_k\otimes e_i + e_i \otimes (b^{ji}_k)^\ast a^{\ast} \otimes a b^{ji}_k {\big)}
\end{align*}
\end{lemma}

Recall that the elements $e_i$ appearing here are the idempotents of $\Lambda$ corresponding to the length-0 paths at each vertex. If $M$ is any left $\Lambda$-module, then $M_i := e_i M$ is the vector space attached to vertex $i$ when we view $M$ as a quiver representation. We note the following useful identifications, which are standard properties of idempotents:
\begin{lemma}
There are canonical isomorphisms
\begin{align*}
&f:e_i\Lambda \otimes_\Lambda M \xrightarrow{\sim} M_i \\
&g:\Hom_\Lambda(\Lambda e_i, M) \xrightarrow{\sim} M_i
\end{align*}
given by
\begin{align*}
f(x\otimes m) &= xm \\
g(\phi) &= \phi(e_i)
\end{align*}
\end{lemma}

Now we obtain a partial projective resolution of $M$ by tensoring it with the above resolution. 
\begin{lemma}
\label{resolution}
The following is the start of a projective left $\Lambda$-module resolution of $M$:
\[
\bigoplus_{1\leq i\leq n} V_i^M \xrightarrow{d^1} \bigoplus_{1\leq i\neq j\leq n} W_{j\ot i}^M \xrightarrow{d^0} \bigoplus_{1\leq i \leq n} V_i^M \xrightarrow{\operatorname{mult}} M\to 0
\]
where
\[
V_i^M = \Lambda e_i \otimes_{\kappa(d_i)} M_i \text{ and }
W_{j\ot i}^M = \Lambda e_{j} \otimes_{\kappa(d_j)} E(j\ot i) \otimes_{\kappa(d_i)} M_i 
\]
with maps
\begin{align*}
\operatorname{mult}(e_i\otimes m) &= e_im \\
d^0(e_j\otimes x \otimes m) &= x\otimes m - e_{j}\otimes xm \\
d^1(e_i\otimes m) &= \sum_{j\neq i} \sgn(j, i)\sum_{i \overset{a}{\longrightarrow} j}\sum_{k} (b^{ji}_k)^\ast a^{\ast} \otimes a b^{ji}_k\otimes m + e_i \otimes (b^{ji}_k)^\ast a^{\ast} \otimes a b^{ji}_k m
\end{align*}
\end{lemma}

\begin{proof}
Because the resolution in Lemma \ref{bimoduleresolution} is a right $\Lambda$-module resolution of $\Lambda$, and $\Lambda$ is a flat right $\Lambda$-module, the sequence will remain exact when we apply $-\otimes_\Lambda M$. The result is the resolution given here.
\end{proof}

Finally, we apply $\Hom_\Lambda(-, N)$ to this resolution in order to compute $\Ext_\Lambda^1$. We note that by the tensor-Hom adjunction,
\begin{align*}
&\Hom_\Lambda(\Lambda e_i \otimes_{\kappa(d_i)} M_i, N) \cong \Hom_{\kappa(d_i)}(M_i, \Hom_\Lambda(\Lambda e_i, N)) \cong \Hom_{\kappa(d_i)}(M_i, N_i) \\
&\Hom_\Lambda(\Lambda e_{j} \otimes_{\kappa(d_j)} E(j\ot i) \otimes_{\kappa(d_i)} M_{i}, N) \cong \Hom_{\kappa(d_j)}(E(j\ot i)\otimes_{\kappa(d_i)} M_{i}, N_{j})
\end{align*}
We thus obtain a complex
\[\bigoplus_{1 \leq i\leq n} V_i^{M,N} \xrightarrow{d^0_{M, N}} \bigoplus_{1 \leq i \neq j \leq n} W_{j \from i}^{M,N} \xrightarrow{d^1_{M, N}}  \bigoplus_{1 \leq i\leq n} V_i^{M,N} \]
where
\[ V_i^{M, N} =  \Hom_{\kappa(d_i)}(M_i, N_i) \ \text{and} \ W_{j\from i}^{M,N} = \Hom_{\kappa(d_j)}(E(j\ot i) \otimes_{\kappa(d_i)} M_i, N_j) . \]
%\begin{align*}
%0\to \bigoplus_{1\leq i\leq n} \Hom_{\kappa(d_i)}(M_i, N_i) &\xrightarrow{d^0_{M, N}} \bigoplus_{1\leq i\neq j\leq n} \Hom_{\kappa(d_j)}(E(j\ot i) \otimes_{\kappa(d_i)} M_i, N_j) \\
%&\xrightarrow{d^1_{M, N}} \bigoplus_{1\leq i\leq n} \Hom_{\kappa(d_i)}(M_i, N_i)
%\end{align*}
with
\[
d^0_{M, N}((f_i)_{i})(x\otimes m) = \left(xf_{i}(m) - f_{j}(xm)\right)_{i\neq j}
\]
\begin{multline*}
d^1_{M, N}((g_{j,i})_{i\neq j})(m) =\\ (\sum_{j\neq i} \sgn(j, i) \sum_{i \overset{a}{\longrightarrow} j}\sum_{k} (b^{ji}_k)^\ast a^{\ast} g_{j,i}(a b^{ji}_k \otimes m) 
 				 + g_{i,j}((b^{ji}_k)^\ast a^{\ast} \otimes a b^{ji}_k m))_{i}
\end{multline*}
The first two homologies of this complex will compute $\Hom_\Lambda(M, N)$ and $\Ext^1_\Lambda(M, N)$. The symmetry in Theorem \ref{CB Identity} will then come from dualizing it.

\begin{lemma}
\label{self-duality}
The above complex is isomorphic to its $\kappa$-dual with $M$ and $N$ swapped:
\[
\bigoplus_{1 \leq i\leq n} (V_i^{N,M})^* \xrightarrow{d^{1*}_{N, M}} \bigoplus_{1 \leq i \neq j \leq n} (W_{j \from i}^{N,M})^* \xrightarrow{d^{0*}_{N, M}}  \bigoplus_{1 \leq i\leq n} (V_i^{N,M})^*
\]
\end{lemma}

\begin{proof}
First, we should identify the terms of the two complexes, which we do by constructing perfect pairings between them. There is a perfect pairing % does characteristic / separability / anything else matter here? Why perfect?
\[\bigoplus_{1\leq i\leq n} V_i^{M, N}\times \bigoplus_{1\leq i\leq n}V_i^{N, M}\to \kappa\] 
defined by 
\[
((f_i)_{i}, (\phi_i)_{i}) \mapsto \sum_{i}\tr_\kappa(\phi_i\circ f_i)
\]
%and through this, we can identify the $\kappa$-dual of $\bigoplus_{1\leq i\leq n} V_i^{M, N}$ with $\bigoplus_{1\leq i\leq n} V_i^{N, M}$. 
We call this the \newword{vertex pairing}.

Next, we have an isomorphism
\[
\theta: \Hom_{\kappa(d_j)}(E(j\ot i) \otimes_{\kappa(d_i)} M_{i}, N_{j}) \xrightarrow{\sim} \Hom_{\kappa(d_i)}(M_{i}, E(i\ot j)\otimes_{\kappa(d_j)} N_{j})
\]
%Here the first step follows from tensor-Hom adjunction, the second from the fact that $E(j\ot i)$ is finite-dimensional over $\kappa(d_j)$, and the third step from the pairing $\tr_{\kappa(d_j)}(-\circ -):E(i\ot j) \times E(j\ot i) \to \kappa(d_j)$. Specifically, chaining together these isomorphisms, we get the map
which we can define using our chosen $\kappa(d_j)$-basis of $E(j\ot i)$:
\[
\theta(g)(m) = \sum_{i \overset{a}{\longrightarrow} j}\sum_{k} (b^{ji}_k)^\ast a^{\ast} \otimes g(a b^{ji}_k \otimes m)
\]
%\sum_{i \overset{a}{\longrightarrow} j}\sum_{k} &(b^{ji}_k)^\ast a^{\ast} g_{j,i}(a b^{ji}_k \otimes m) \\
% 				 &+ g_{i,j}((b^{ji}_k)^\ast a^{\ast} \otimes a b^{ji}_k m))_{i}
We can then combine this identification with the trace pairing and twist by the sign factor to get a perfect pairing
\[
\bigoplus_{1\leq i\neq j\leq n} W_{j\ot i}^{M, N}\times \bigoplus_{1\leq i\neq j\leq n} W_{j\ot i}^{N, M}\to \kappa
\]
given by
\begin{align*}
((g_{j,i})_{i\neq j}, (\psi_{j, i})_{i\neq j}) &\mapsto \sum_{i\neq j} \sgn(j, i)\tr_\kappa(\psi_{i, j}\circ \theta(g_{j, i}))\\
&= \sum_{i\neq j} \sgn(j, i)\tr_\kappa\left(\sum_{i \overset{a}{\longrightarrow} j}\sum_{k} \psi_{i,j}((b^{ji}_k)^\ast a^{\ast} \otimes g_{j,i}(a b^{ji}_k \otimes -))\right)
\end{align*}
%\sum_{i \overset{a}{\longrightarrow} j}\sum_{k} &(b^{ji}_k)^\ast a^{\ast} g_{j,i}(a b^{ji}_k \otimes m) \\
% 				 &+ g_{i,j}((b^{ji}_k)^\ast a^{\ast} \otimes a b^{ji}_k m))_{i}
We call this the \newword{edge pairing}.

The presence of the field elements $b^{ji}_k$ is an additional complication compared to the proof in \cite{GLS}, and the discrepancy between $\{b^{ji}_k\}$ being a $\kappa(d_j)$-basis of $\kappa(d_{ij})$ and $\{b^{ij}_k\}$ being a $\kappa(d_i)$-basis prevents a direct adaptation. However, we can eliminate the $b$'s.

\begin{lemma}
\label{edgepairingsimplification}
The edge pairing is given by
\[
((g_{j,i})_{i\neq j}, (\psi_{j, i})_{i\neq j}) \mapsto \sum_{i\neq j} \sgn(j, i)\tr_\kappa\left(\sum_{i \overset{a}{\longrightarrow} j} g_{j, i}(a \otimes \psi_{i,j}(a^*\otimes -))\right)
\]
\end{lemma}

We show this using the following lemma on the trace pairing.
\begin{lemma}
\label{tracepairingfact}
Let $L / K$ be a separable field extension. Let $b_1,\ldots,b_n$ be a $K$-basis of $L$, and let $b_1^\ast,\ldots, b_n^\ast$ be the dual basis with respect to the trace pairing $\tr_K(-, -): L\times L\to K$. Then $\sum_{i=1}^n b_i^\ast b_i = 1$.
\end{lemma}

\begin{proof}
First, we note that the value of $\sum_{i=1}^n b_i^\ast b_i$ is independent of our choice of basis. This is because, under the isomorphism $L\otimes_K L \cong L^*\otimes_K L \cong \Hom_K(L, L)$, where the first map is induced by the trace pairing, the element $\sum_{i=1}^n b_i^\ast \otimes b_i$ corresponds to the identity map regardless of basis.

Let $\mathbb{K}$ be the algebraic closure of $K$ and let $A$ be the $\KK$-algebra $L \otimes_K \KK$. The trace pairing $L \times L \to K$ extends to a bilinear pairing $A \times A \to \KK$, and the ring structure on $L$ extends to a $\KK$-algebra structure on $A$, so we can do the computation in $A$.
But, since $L/K$ is separable, $A \cong \KK^{\times n}$ as a ring. If $e_1$, $e_2$, \dots, $e_n$ is a basis of primitive idempotents for $A$, then $(e_1, e_2, \ldots, e_n)$ forms a self dual basis of $A$ and $\sum e_i e_i = \sum e_i = 1$, as desired.
%
%Then since $L/K$ is separable, we can choose a primitive element $\alpha\in L$ such that $1,\alpha, \ldots, \alpha^{n-1}$ is a $K$-basis of $L$.
% Set $b_i = \alpha^i$ and let $b^{\ast}_i$ be the corresponding dual basis.
%We use a formula of Lang~\cite[Proposition~VI.5.5]{Lang}\margin[Hugh]{correct chapter number} for the dual basis: 
%Let $f(X)$ be the minimal polynomial of $\alpha$, let $f'(X)$ be its derivative and let 
%$\tfrac{f(X)}{X - \alpha} = \beta_0 + \beta_1 X + \ldots + \beta_{n-1} X^{n-1}$. 
%Then $b_i^{\ast} = \tfrac{\beta_i}{f'(\alpha)}$. 
%%Then the dual basis of $1, \alpha,\ldots, \alpha^{n-1}$ is 
%%\[
%%\frac{\beta_0}{f'(\alpha)},\ldots, \frac{\beta_{n-1}}{f'(\alpha)}
%%\]
%
%So $\sum_{i=0}^{n-1} b_i b_i^{\ast} = \tfrac{1}{f'(\alpha)}  \sum_{i=0}^{n-1} \alpha^i \beta_i$. 
%The sum $\sum_{i=0}^{n-1} \alpha^i \beta_i$ is the polynomial $f(X) / (X - \alpha)$ evaluated at $X = \alpha$, so the sum equals $f'(\alpha)$. We deduce that $\sum_{i=0}^{n-1} b_i b_i^{\ast} = \tfrac{1}{f'(\alpha)} f'(\alpha) = 1$.
%%Plugging $\alpha$ into the polynomial $f(X) / (X - \alpha)$ gives $f'(\alpha)$, but by the lemma this is equal to $f'(\alpha)\sum_{i=1}^n b_i^\ast b_i$.
\end{proof}

\begin{proof}[Proof of Lemma \ref{edgepairingsimplification}]
By cyclically permuting within it, we can rewrite the trace term appearing in the definition of the edge pairing as
\[
\tr_\kappa\left(\sum_{i \overset{a}{\longrightarrow} j}\sum_{k} a b^{ji}_k \otimes \psi_{i,j}((b^{ji}_k)^\ast a^{\ast} \otimes g_{j,i}(-))\right)
\]
Because we assumed $b^{ji}_k\in \kappa(d_i)$, we can use linearity to pass it under $\psi_{i, j}$, next to $(b^{ji}_k)^*$, and then apply Lemma \ref{tracepairingfact} to eliminate the inner sum.
\end{proof}

With formulas for the vertex and edge pairings in place, it is tedious but straightforward to check that, under the resulting identifications $(V_i^{M, N})^* \cong V_i^{N, M}$ and $(W_{j\ot i}^{M, N})^*\cong W_{i\ot j}^{N, M}$, $d^{1*}_{N, M}$ is identified with $d^0_{M, N}$ and $d^{0*}_{N, M}$ is identified with $-d^1_{M, N}$. One can do the same computations as in \cite{GLS} while eliminating $b$'s using the trick of Lemma \ref{tracepairingfact}.

%Now we note that if we swap the roles of $M$ and $N$ and the order of the arguments in the definition of the vertex pairing, it remains the same, while if we do so in the definition of the edge pairing, it is multiplied by $-1$. Thus, starting with the correspondence between $d^{1*}_{N, M}$ and $d^0_{M, N}$, switching the roles of $M$ and $N$, dualizing, and multiplying the edge pairing by a corrective factor of $-1$, we can conclude that, under the original vertex and edge pairings, $d^{0*}_{N, M}$ corresponds to $-d^1_{M, N}$. This is sufficient to give the desired isomorphism.
\end{proof}

\begin{proof}[Proof of Theorem \ref{CB Identity}]

Lemma \ref{self-duality} tells us that, in addition to first two homologies of our complex computing $\Hom_{\Lambda}(M, N)$ and $\Ext^1_\Lambda(M, N)$, the third is isomorphic to $\Hom_{\Lambda}(N, M)^*$. Thus we can compute the alternating sum
\[
\dim_\kappa(\Hom_\Lambda(M, N)) - \dim_\kappa(\Ext^1_\Lambda(M, N)) + \dim_\kappa(\Hom_\Lambda(N, M))
\]
as the alternating sum of the $\kappa$-dimensions of our complex.

We have
\[
\dim_{\kappa}(\Hom_{\kappa(d_i)}(M_i, N_i)) = d_i\dim_{\kappa(d_i)}(M_i)\dim_{\kappa(d_i)}(N_i).
\]

Additionally,
\begin{align*}
\dim_{\kappa(d_j)}(E(j\ot i)\otimes_{\kappa(d_i)} M_i) &= \frac{d_i}{d_j}\dim_{\kappa(d_i)}(E(j\ot i)\otimes_{\kappa(d_i)} M_i) \\
&= -\frac{d_i A_{ij}}{d_j}\dim_{\kappa(d_i)}(M_i)
\end{align*}
so
\[
\dim_\kappa(\Hom_{\kappa(d_j)}(E(j\ot i) \otimes_{\kappa(d_i)} M_i, N_j)) =
  -d_i A_{ij}\dim_{\kappa(d_i)}(M_i)\dim_{\kappa(d_j)}(N_j)
\]
Taking the alternating sum, we get
\[
\sum_{1\leq i\leq n} 2d_i\dim_{\kappa(d_i)}(M_i)\dim_{\kappa(d_i)}(N_i) + \sum_{1\leq i\neq j\leq n} d_i A_{ij}\dim_{\kappa(d_i)}(M_i)\dim_{\kappa(d_j)}(N_j)
\]
which is exactly $(\dim M, \dim N)$.
\end{proof}

\end{document}